\documentclass{amsart}[14pt]
\usepackage{amssymb, amsfonts}
\usepackage{hyperref}

\newtheorem{theorem}{Theorem}[section]
\newtheorem{lemma}[theorem]{Lemma}
\newtheorem{proposition}[theorem]{Proposition}
\newtheorem{corollary}[theorem]{Corollary}

\theoremstyle{definition}
\newtheorem{definition}[theorem]{Definition}
\newtheorem{example}[theorem]{Example}
\newtheorem{notation}[theorem]{Notation}

\theoremstyle{remark}
\newtheorem{remark}[theorem]{Remark}

\newcommand{\Ad}{\mathop{\mathrm{Ad}}\nolimits}
\newcommand{\ad}{\mathop{\mathrm{ad}}\nolimits}

\newcommand{\spin}{\mathop{\mathrm{spin}}\nolimits}
\newcommand{\Spin}{\mathop{\mathrm{Spin}}\nolimits}
\newcommand{\SO}{\mathop{\mathrm{SO}}\nolimits}
\newcommand{\so}{\mathop{\mathrm{so}}\nolimits}
\newcommand{\Kil}{\mathop{\mathrm{Kil}}\nolimits}

\newcommand{\grad}{\mathop{\mathrm{grad}}\nolimits}

\newcommand{\cA}{{\mathcal A}}
\newcommand{\cB}{{\mathcal B}}
\newcommand{\cC}{{\mathcal C}}
\newcommand{\cD}{{\mathcal D}}
\newcommand{\cE}{{\mathcal E}}

\newcommand{\cH}{{\mathcal H}}

\newcommand{\cK}{{\mathcal K}}
\newcommand{\cL}{{\mathcal L}}
\newcommand{\cM}{{\mathcal M}}

\newcommand{\cS}{{\mathcal S}}
\newcommand{\cT}{{\mathcal T}}
\newcommand{\cU}{{\mathcal U}}

\newcommand{\cW}{{\mathcal W}}

\newcommand{\ciA}{{\mathcal A^\infty}}
\newcommand{\ciB}{{\mathcal B^\infty}}

\newcommand{\bC}{{\mathbb C}}
\newcommand{\bR}{{\mathbb R}}

\newcommand{\fg}{{\mathfrak g}}
\newcommand{\fh}{{\mathfrak h}}
\newcommand{\fk}{{\mathfrak k}}
\newcommand{\ft}{{\mathfrak t}}
\newcommand{\fm}{{\mathfrak m}}
\newcommand{\fn}{{\mathfrak n}}

\renewcommand{\a}{\alpha}
\renewcommand{\b}{\beta}
\renewcommand{\d}{\delta}
\newcommand{\D}{\Delta}
\newcommand{\e}{\epsilon}
\newcommand{\g}{\gamma}

\renewcommand{\k}{\kappa}
\renewcommand{\l}{\lambda}
\newcommand{\om}{\omega}
\renewcommand{\O}{\Omega}

\newcommand{\s}{\sigma}
\renewcommand{\th}{\theta}

\newcommand{\cfg}{{\mathfrak g^\bC}}
\newcommand{\cfh}{{\mathfrak h^\bC}}
\newcommand{\cfk}{{\mathfrak k^\bC}}
\newcommand{\cfm}{{\mathfrak m^\bC}}
\newcommand{\rd}{{\mathrm{d}}}
\newcommand{\bg}{{\boldsymbol{\g}}}

\newcommand{\<}{\langle}
\renewcommand{\>}{\rangle}

\newcommand{\dia}{\diamond}

\newcommand{\md}{\mu_\diamond}

\newcommand{\cad}{\mathrm{cad}}  
\newcommand{\Cad}{\mathrm{Cad}}   
\newcommand{\fmc}{\fm'^\bC}
\newcommand{\clfm}{\bC\ell(\fm')}
\newcommand{\clf}{\bC\ell(}

\newcommand{\dom}{\mathrm{dom}}

\newcommand{\sfD}{{\mathsf{D}}}

\newcommand{\tn}{{\mathsf{T}}{\mathsf{N}}}
\newcommand{\ppq}{propinquity{ }}

\newcommand{\Lat}{{Latr\'emoli\`ere}}

\numberwithin{equation}{section}

\begin{document}

\title[Dirac operators for matrix algebras converging ]
{Dirac operators for matrix algebras converging \\ to coadjoint orbits}

\author{Marc A. Rieffel}
\address{Department of Mathematics\\
University of California\\
Berkeley, CA\ \ 94720-3840}
\email{rieffel@math.berkeley.edu}

   


\begin{abstract}
In the high-energy physics literature one finds 
statements such as ``matrix algebras converge to the sphere''.
Earlier I provided a general precise setting for understanding such
statements, in which the matrix algebras are viewed as
quantum metric spaces, and convergence is with respect
to a quantum Gromov-Hausdorff-type distance. 

But physicists want even more to treat structures on spheres 
(and other spaces), such as vector bundles, Yang-Mills functionals, 
Dirac operators, etc., and they want to approximate these by corresponding 
structures on matrix algebras. In the present paper
we provide a somewhat unified construction of Dirac operators on coadjoint 
orbits and on the matrix algebras that converge to them. This enables us 
to prove our main theorem, whose content is that, for the quantum 
metric-space structures determined by the 
Dirac operators that we construct,
the matrix algebras do indeed converge to the coadjoint orbits, 
for a quite strong version of quantum Gromov-Hausdorff distance. 

\end{abstract}

\maketitle

\tableofcontents

\section*{Introduction}

In the literature of theoretical high-energy physics one finds statements along the lines 
of ``this sequence of matrix algebras converges to the sphere'' 
and ``here are the Dirac operators on the 
matrix algebras that correspond to the Dirac operator on the sphere''.  But one also finds 
that at least three inequivalent definitions of Dirac operators have been proposed in this context.  
See, for example, \cite{AIM,BIm,BKV,CWW97,GP2,HQT,Yd1,Ydr} and the references 
they contain.  
In \cite{R6,R7,R21} I provided definitions and theorems that give a precise meaning to 
the idea of
the convergence of matrix algebras to spheres. This involved equipping the matrix 
algebras in a natural way with the structure of a non-commutative (or ``quantum'') 
compact metric space (motivated by the ``spectral triples'' 
of Connes \cite{Cn7, Cn5, Cn3, CnMr}), 
and developing non-commutative (or ``quantum'') versions 
of the usual Gromov-Hausdorff distance between 
compact metric spaces. 
These results were developed in the general context of coadjoint orbits of compact Lie groups, 
which is the appropriate context for this topic, as is clear from the physics literature. (In \cite{Ln2} it is shown that the matrix algebras form a 
strict quantization of the coadjoint orbits. A much-simplified proof of 
much of that
fact, in the more general context of compact quantum groups, is given
in proposition 4.13 of \cite{Sai}.) 

The purpose of the present paper is to provide a somewhat unified construction 
of Dirac operators (always $G$-invariant) that handles both the 
matrix algebras and the corresponding coadjoint orbits.
This enables us to prove the main theorem of this paper, which states that
for the quantum metrics determined by the Dirac operators 
on the matrix algebras and on the coadjoint orbits, 
the matrix algebras do indeed converge to the coadjoint orbits, in fact for a stronger
version of quantum Gromov-Hausdorff distance than I had used earlier.

Roughly speaking, physicists who try to develop quantum field 
theory on spaces like the sphere, 
have found attractive the idea of approximating the spaces by 
means of matrix algebras because 
if they try  ``lattice approximations'' by a collections of 
points, they lose the action of the symmetry group, 
whereas the matrix algebras can be viewed as 
quantum finite sets, on which the symmetry group still acts.
It is my hope that when results of the kind contained in this paper are extended to further 
types of structure of interest to quantum physicists, such as Yang-Mills functionals \cite{CnR, R34} 
and other action functionals, then the fact that we have quantified the idea of distance 
between matrix algebras and spaces will help in quantifying the size of the error made 
by approximating 
quantities in quantum field theory on the spaces by corresponding quantities obtained 
for quantum field theory on the approximating matrix algebras. (Another type of structure 
for which this kind of approximation theory has already been worked out consists of vector 
bundles \cite{R32, Ltr6}.)

Here is a brief imprecise sketch of what our construction of Dirac operators looks like
for the case of the matrix algebras 
that converge to the (2-dimensional) sphere, 
for which the symmetry group is $ G = SU(2) $. 
For each positive integer $n $ let $( \cH^n, U^n )$ be
the irrep (irreducible unitary representation) of highest weight $n $ of $ G $ 
(so $\dim(\cH^n) =n +1 $).
(Labeling by highest weights is especially convenient in the general case 
of coadjoint orbits.)
Let $\cB^n = \cB(\cH^n) $ be the algebra of all linear operators on $\cH^n $ --
our full matrix algebra. Then let $\a $ be the action of $G $ on $\cB^n $ by 
conjugation by the representation $ U^n $. We then somewhat follow the steps that 
are used in constructing
Dirac operators on compact Riemannian manifolds, especially 
Riemannian homogeneous spaces. We think of the Lie algebra, 
$\fg =su(2)$, as the tangent 
vector space at some nonexistent point, so that 
the (complexified) tangent bundle is $\cB^n \otimes_\bR \fg $. 
Let $\fg'$ be the vector space dual of $\fg $, so that the cotangent bundle is 
$\O^n =\cB^n \otimes_\bR \fg' $, viewed as a right module over $\cB^n $. 
We let $\a $ also denote the corresponding 
infinitesimal action of $\fg $ on $\cB^n $. Then
for each $ T\in \cB^n $ its total differential, $dT $, is the element 
of $\O^n $ defined by $dT(X) =\a_ X(T)$ for all $X \in \fg$.
As inner product on $\fg $ we use the negative of the 
Killing form on $\fg $, and as inner product on $\fg'$ we use the dual
of the inner product on $\fg $. Then the Riemannian metric on 
$\cB^n $, viewed as a $\cB^n$-valued inner product
on the cotangent bundle $\O^n $, is defined on 
elementary tensors by
\[
\< S\otimes \nu, T\otimes \mu\> = S^*T\<\nu,\mu\>
\]
for $ S, T\in \cB^n $ and $\nu,\mu \in \fg'$. We then let $\bC\ell(\fg')$
be the complex Clifford algebra corresponding to the inner product
on $\fg'$. The Clifford bundle over $\cB^n $ is then $\cB^n \otimes \bC\ell(\fg') $.
As spinors we can choose the Hilbert space $\cS $ of one of the two
irreducible $*$-representations of $\bC\ell(\fg')$ (necessarily 2-dimensional).
The spinor bundle is then $\cB^n \otimes \cS$. The Dirac operator, $ D $,
is then an operator on the spinor bundle. It is defined as follows. 
Let $\{E_j\}_{j=1}^3 $ be an orthonormal basis for $\fg $, 
and let $\{\e_j\}_{j=1}^3$ be the 
dual orthonormal basis for $\fg'$. Let $\k $ denote the representation of
$\bC\ell(\fg')$ on $\cS$.
Then for any elementary tensor $ T\otimes \psi $ in the 
spinor bundle (where $ T\in \cB^n $ and
$\psi \in \cS$) we set
\begin{equation}
\label{eqdefdir1}
D(T \otimes \psi) = \sum \a_{E_j}(T) \otimes \k_{\e_j}(\psi).
\end{equation}
This Dirac operator has many attractive properties, which 
we will describe later. In particular, we will see that it is closely 
related to the Casimir element for the chosen inner product
on $ \fg $, and that this gives an attractive way of calculating 
the spectrum of $ D $. When the main theorem of this paper,
Theorem \ref{lengthlim},
is applied to the case of the sphere, it tells us that when the sphere
is equipped with its usual metric (which corresponds to its Dirac operator), 
and the matrix algebras $\cB^n $
are equipped with the quantum metrics (defined in Section \ref{cmet}) 
determined by the 
Dirac operators defined just above,
then the matrix algebras  $\cB^n $ converge to the sphere for a 
quite strong form of quantum Gromov-Hausdorff distance 
(namely the propinquity
of \Lat \ \cite{Ltr2}, which we will describe later). The proof of our main 
theorem makes essential use of coherent states and Berezin symbols.

Our construction of Dirac operators on matrix algebras sketched above 
is closely related to some of the proposals in the physics literature.
At the end of Section \ref{fuzzy} we give a substantial discussion of the 
relationship between our construction and various proposals, 
for the case of the sphere. 
In section \ref{compare} we further discuss how our construction relates
to proposals in the physics literature for Dirac operators on 
other coadjoint orbits, 
such as projective spaces.

Our general construction has some deficiencies. We discuss these deficiencies
in some detail in Section \ref{secdefit}. But briefly,
Dirac operators on Riemannian manifolds are usually constructed using the 
the Levi--Civita connection, which is the unique torsion-free connection 
compatible with the Riemannian metric. But our construction 
in essence uses the ``canonical connection", and as is well-known, 
and explained in section 6 of \cite{R22}, for homogeneous spaces
the Levi--Civita connection agrees with the canonical connection only for 
symmetric spaces. 
The consequence is that our construction is very satisfactory for the case in which
the coadjoint orbit is a symmetric space, 
such as the sphere or a projective space, 
but somewhat less satisfactory otherwise, though we will see that 
the metric from the canonical connection agrees with the metric
from the Levi--Civita connection, so our
sequence of matrix algebras does converge to the coadjoint orbit
for the Dirac operator for its Levi--Civita connection. 

A related deficiency, discussed further in Section \ref{secdefit}, comes from 
the fact that each highest weight determines 
a $G$-invariant K\"ahler structure on its coadjoint orbit, 
which includes not just a Riemannian structure,
but also a closely related complex structure and symplectic structure. 
But these Riemannian 
metrics are not related to the Killing form unless 
the coadjoint orbit is a symmetric space.
When that is not the case, our sequence of matrix algebras will not converge
to the coadjoint orbit equipped with the Riemannian metric from the K\"ahler structure,
because it converges
to the coadjoint orbit equipped with the Riemannian metric related to the Killing form.

This paper concludes with a final section that discusses how
the main results of this paper are related to the very interesting
``spectral propinquity" of \Lat \ \cite{Ltr7}, which is a metric on 
equivalent classes of spectral triples. The conclusion that is reached is that the spectral propinquity is too strong, in the sense that the spectral triples
that we construct for matrix algebras do not converge to the spectral triples of coadjoint orbits for the spectral propinquity. On the other hand, the present paper says nothing about convergence of these spectral triples. It only 
shows that for the C*-metrics coming from the spectral triples, the matrix 
algebras do converge to the coadjoint orbits for \Lat's Gromov-Hausdorff
propinquity \cite{Ltr2}. It would be very desirable to 
have a weaker form of convergence for spectral triples than the spectral
propinquity, for which our spectral triples for the matrix 
algebras do converge to the spectral triples for the coadjoint orbits.
Perhaps there might be something roughly along the lines that Lott used for ordinary Dirac operators \cite{Lt1, Lt2}.

There is a substantial relationship between the Dirac operators on matrix algebras
that we construct here and the more general ``matrix models" that are being
intensively explored in the literature of theoretical high-energy physics. 
See \cite{Stn21} (some of whose examples are ``fuzzy'' spaces of the kind 
studied in the present paper) and its many references. It
would be very interesting to explore how some of the ideas in the present paper
might be extended to various matrix models, such as those 
in \cite{AoNT, PSn, PSn2, PSn3,  Stn21, StnZ, Ydr17}.


\section{Ergodic actions and Dirac operators}
\label{sec1}

In this section we construct  a rough approximation to the Dirac 
operators that we seek. We do this in the general setting of ergodic
 actions of compact Lie groups on unital C*-algebras. This rough approximation 
 will be crucial for the proofs of some of our main results.

Let $G$ be a connected compact Lie group. In this paper we are concerned 
with two types of actions of $G$ on unital C*-algebras. One type consists of 
actions by translation on the C*-algebras $C(G/K)$ of continuous
complex-valued functions on homogenous spaces $G/K$, where $K$ is a 
closed subgroup of $G$. The other type consists of actions on the C*-algebra
$\cB(\cH)$ consisting of all the operators on the Hilbert space $\cH$ of 
an irrep of $G$, where the action of $ G$ on $\cB(\cH)$ 
is by conjugation by the representation. In both cases the action is 
ergodic in the sense that the only elements of the algebra that are 
invariant under the action of $G$ are scaler multiples of the identity element. 
Consequently it is natural to begin by considering general ergodic actions 
on unital C*-algebras.

Thus let $\a$ be an ergodic action of $G$ on a unital C*-algebra $\cA$. 
Using the fact that $ G$ is a Lie group, we can define the subalgebra
$\ciA$ of smooth elements of $\cA$, consisting of those elements 
$a \in \cA$ such that the function $x \mapsto \a_x(a)$ from $G$
to $\cA$ is infinitely differentiable for the norm on $\cA$.
This means that for any $X \in \fg$, where $\fg$ is the Lie algebra of $G$,
there is an element, $\a_X(a) \in \cA$, such that 
$D^t_0(\a_{\mathrm{exp}(tX)}(a)) = \a_X(a)$,
(where $D_0^t$ means ``derivative in $t$ at $t = 0$''),
and similarly for higher derivatives. 
It is a standard fact that 
$X \mapsto \a_X$ is a Lie algebra homomorphism from $\fg$
into the Lie algebra of $*$-derivations of $\ciA$, and that,
using the
G\"arding smoothing argument, $\ciA$ is dense in $\cA$.
(See Sections 3 and 4 of chapter III of \cite{Knp}, whose discussion
for unitary representations on Hilbert spaces adapts very easily
to isometric actions on any Banach space.)
It is easily verified that $\ciA$ is a unital $*$-subalgebra of $\cA$.

Let $\fg'$ denote the dual vector space of $\fg$, and 
let $\O_o = \ciA \otimes \fg'$ (where $\otimes$ is necessarily over $\bR$).
Here and throughout the paper such tensor products are algebraic 
tensor products with one factor finite dimensional.
View $\O_o$ as an $\ciA$-bimodule by letting $\ciA$ act on itself by 
left and right multiplication. For any $a \in \ciA$ let $\rd a:\fg \to \ciA$ be
defined by $\rd a(X) = \a_X(a)$. Then $\rd a$ can be viewed as an element of
$\O_o$, and it is easily verified that $\rd$ is then a derivation from $\ciA$
into $\O_o$. (Thus $(\O_o, d)$ can be viewed as a ``first-order differential 
calculus" for $\ciA$ \cite{BggM, GVF}.) 
We can view $\O_o$ as a cotangent bundle
for $\cA$. (But later we will want to do better.)

We remark that here, and in later paragraphs of this section, 
we do not need the more general structure of ``connections", 
because the $\ciA$-modules 
we consider in this section are finitely generated free $\ciA$-modules.
Connections will be needed in Section \ref{dhom} and later.

We want $\rd: \ciA \to \O_o$ to be equivariant for an action of $G$.
This will usually not be true if we take the action on $\O_o$ to be
$\a \otimes I^{\O_o}$, where $I^{\O_o}$ denotes the identity
operator on $\O_o$. Let $\Cad$ denote
the coadjoint representation of $G$ on $\fg'$, dual to
the adjoint action $\Ad$ on $\fg $ (so
$ \<X, \Cad_x(\mu)\>=\<\Ad^{-1}_x(X), \mu\> $ for $ X\in\fg $, $x \in G$, 
and $\mu\in \fg'$, where here $\<\cdot, \cdot\>$ denotes the pairing between
$\fg$ and $\fg'$).

\begin{notation}
\label{act}
We let $\g$ be the diagonal action of $G$ on $\O_o$ defined by
\[
\g_x = \a_x \otimes \Cad_x     
\]
for any $x \in G$.
\end{notation}

\begin{proposition}
For any $a \in \ciA$ and $x \in G$ we have
\[
\g_x(\rd a) = \rd(\a_x(a)),
\]
that is, $\rd$ is equivariant for the actions $\a$ and $\g$.
\end{proposition}

\begin{proof}
Let $X \in \fg$. Then
\begin{align*}
(\g_x(\rd a))(X) &= ((I^\cA \otimes \Cad_x)(\a_x \otimes I^{\fg'})\rd a)(X)   \\
&= ((\a_x \otimes I^{\fg'})\rd a)(\Ad_x^{-1}(X)) = \a_x(\a_{Ad^{-1}_x(X)}(a))   \\
&= \a_X(\a_x(a)) =\rd (\a_x(a))(X)    .
\end{align*}
\end{proof}

In section 4 of \cite{R4} we briefly constructed some Dirac-type operators
for $\ciA$, and proved that for the corresponding Leibniz semi-norms 
$\cA$ becomes a C*-metric space. More recently, in \cite{GaG} the authors
obtained important basic properties of these Dirac operators, such as their
essential self-adjointness and having compact resolvant. We now
recall the construction, slightly reformulated in a way that will be more 
convenient later, and influenced by \cite{GaG}. The usual Dirac operators
on a manifold
are defined in terms of a Riemannian metric on the tangent bundle of 
the manifold, or equivalently, on the cotangent bundle. In non-commutative 
geometry ``cotangent bundles" are more commonly available 
then ``tangent bundles", where now the analogs of vector bundles are 
modules, usually finitely generated projective, corresponding to 
the modules of smooth cross-sections of ordinary vector bundles. 
To aid the reader's intuition, in this 
paper it seems best to still refer to these modules as ``bundles''.
Note that since $\fg'$ is finite-dimensional, $\O_o$ is a 
finitely-generated free (right or left) module over $\ciA$.

To construct a Dirac operator in our general framework, we 
need a ``Riemannian metric"
on $\O_o$. For this purpose we choose an inner product 
on $\fg'$ that is $\Cad$-invariant. (Because $ G $ is compact,
these always exist, as seen by averaging any inner product on $\fg'$ using the 
action $\Cad$
and the Haar measure of $G$.)
We fix such an inner product, and denote it just by $\<\cdot, \cdot\>$. We view 
$\O_o$ as a right $\ciA$-module, and our Riemannian metric on $\O_o$ is the 
$\ciA$-valued inner product on $\O_o$
that is given on elementary tensors (necessarily over $\bR$) by
\[
\<a\otimes \mu, b\otimes \nu\>_\cA = a^*b\<\mu, \nu\>    
\]
(so we define our inner product to be linear in the second variable, 
as done in \cite{GVF, Hll, R27, R33}).
This ``bundle metric" is easily checked to respect the action $\g$
in the sense that
\[
\<\g_x(\om_1), \g_x(\om_2)\>_\cA = \a_x(\<\om_1, \om_2\>_\cA)   
\]
for all $\om_1, \om_2 \in \O_o$ and all $x \in G$.

To construct a Dirac-type operator we must first define the Clifford bundle. 
For our chosen inner product on $\fg'$ we form its complex Clifford algebra.
Much as in \cite{LwM,GVF}, we denote it by $\bC\ell(\fg')$.  
It is the complexification 
of the real Clifford algebra for $\fg'$ with our chosen inner product.  
We follow the convention that the defining relation is
\[
\mu\nu + \nu\mu = -2\<\mu,\nu\>1
\]
for $\mu, \nu \in \fg'$.
We include the minus sign for consistency with \cite{LwM,R22}.  
Thus if one wants to apply the results of the first pages of 
chapter~5 of \cite{GVF} one must let the $g$ there to be the 
negative of our inner product.  After exercise 5.6 of \cite{GVF} it 
is assumed that the inner product is positive, so small changes 
are needed when one uses the later results in \cite{GVF} but 
with our different convention.  The consequence of including 
the minus sign is that in the representations which we will 
construct the elements of $\fg'$ will act as skew-adjoint operators, 
just as they do for $\fg$ for orthogonal or unitary representations of $G$, 
rather than as self-adjoint operators as happens when 
the minus sign is omitted. (The involution on $ \bC\ell(\fg')$ takes $\mu$
to $-\mu$ for $\mu \in \fg$.)
The corresponding Clifford bundle is then the C*-algebra
$\cA \otimes \bC\ell(\fg')$, or its smooth version $\ciA \otimes \bC\ell(\fg')$.

We need a $*$-representation, $\k$, of $ \bC\ell(\fg')$
on a finite-dimensional Hilbert space $\cS$ (for ``spinors"). For the moment 
we do not require that it be irreducible or faithful.
We then form $\ciA \otimes \cS$. It is a right module over
$\ciA $, and is the analog of the bundle of ``spinor fields''
for a Riemannian manifold.

We are now in position to define a Dirac-type operator, $D_o$, on
$\ciA \otimes \cS$. It is simply the composition of the following 
three operators. The first is the operator $\rd \otimes I^\cS$ from 
$\ciA \otimes \cS$ into
$\O_o \otimes \cS = \ciA \otimes \fg' \otimes \cS$, where $I^\cS$
denotes the identity operator on $\cS$. The second is the operator 
$I^\ciA   \otimes i \otimes I^\cS$ from $\ciA \otimes \fg' \otimes \cS$
into $\ciA \otimes  \bC\ell(\fg') \otimes \cS$, where $i$ is the 
inclusion of $\fg'$ into $ \bC\ell(\fg')$. The third is the operator
$I^\ciA \otimes \k$ from $\ciA \otimes\bC\ell(\fg') \otimes \cS$
into $\ciA \otimes \cS$ coming from the ``Clifford 
multiplication'' $\k$. Briefly:
\[
 \ciA \otimes \cS
\overset \rd \longrightarrow
\ciA \otimes \fg' \otimes \cS
\overset i \longrightarrow 
\ciA \otimes  \bC\ell(\fg') \otimes \cS
\overset \k \longrightarrow 
 \ciA \otimes \cS
\ . 
\]
This can be expressed in the somewhat cryptic form 
\[
 D_o\Psi =\k(\rd \Psi) 
\]
for all $\Psi\in \ciA \otimes \cS$.
We can obtain a more explicit form for $D_o$ by choosing a 
basis, $\{ E_j\}$,
for $\fg$ and letting $\{\e_j\}$ be the dual basis for $\fg'$. 
Then for any $X \in \fg$ we have $X = \sum \<X, \e_j\>E_j$, 
and so for any $a \in \ciA$ we have 
$\rd a(X) = \sum \<X, \e_j\>\a_{E_j}(a)$, so that
$\rd a = \sum \a_{E_j}(a) \otimes \e_j$. Then for
$\psi \in \cS$ we have
\begin{equation}
\label{eqdefdir}
D_o(a \otimes \psi) = \sum \a_{E_j}(a) \otimes \k_{\e_j}(\psi).
\end{equation}
Notice that $D_o$ only depends on the action of $\fg$ on $\ciA$,
not on the action of $G$ itself. Thus we can replace $ G$ by its
simply-connected covering group whenever convenient. 
This observation will be useful later. It is also important to 
notice that the expression for $D_o$ is independent of the 
choice of the basis $ \{E_j\}$.

To view $D_o$ as an (unbounded) operator on a Hilbert space, we let
$\tau$ be the (unique by \cite{HLS}) $\a$-invariant tracial state
on $\cA$ \cite{HLS},
and we define an inner product on $\cA \otimes \cS$ by
\[
\<a_1\otimes \psi_1, a_2\otimes \psi_2\>  = \tau(a^*_1 a_2) \<\psi_1, \psi_2\>
\]
on elementary tensors. (Much as before, 
we choose our inner product on $\cS$ 
to be linear in the second variable, )
On completing, we obtain a Hilbert space, 
$L^2(\cA, \tau) \otimes \cS$ where $L^2(\cA, \tau)$
is the GNS Hilbert space for $\tau$,
with $D_o$ 
defined on the dense subspace $\ciA \otimes \cS$.

\begin{definition}
\label{gendir}
We will call the operator $ D_o $ ( $= D_o^\cA $) constructed above the 
\emph{general Dirac-type operator} for the ergodic action
$\a $ of $ G $ on $\cA $ and the given inner product on $\fg'$.
\end{definition}

A simple calculation (given in the proof of proposition 2.12 of \cite{GaG}) 
shows that $ D_o $ is symmetric. Because $\tau$ is $\a$-invariant, the action 
$\a$ is unitary with respect to the inner product that $\tau$ determines on
$\ciA$, and so it extends to a unitary representation on $L^2(\cA, \tau)$.
Because $G$ is compact, the irreducible $\a$-invariant subspaces of
$\cA$ are finite-dimensional, and so are contained in $\ciA$, and the 
span of all of them together is dense in $\cA$. It follows that $L^2(\cA, \tau)$ 
decomposes into a (Hilbert space)-direct sum of orthogonal 
finite-dimensional $\a$-invariant subspaces of $\ciA$. Consequently, since
$\cS$ is finite-dimensional,
$L^2(\cA, \tau) \otimes \cS$ decomposes into a direct sum of orthogonal 
finite-dimensional subspaces of $\ciA \otimes \cS$ that are each invariant 
under the action $\a \otimes I^\cS$ of $G$. From the definition of
$D_o$ it is evident that $D_o$ carries each of these finite-dimensional 
subspaces into itself. From this and the fact that $D_o$ is symmetric,
it is easy to obtain:

\begin{proposition}
\label{selfad}
As an operator on $L^2(\cA, \tau) \otimes \cS$ with dense domain
$\ciA \otimes \cS$, the operator $D_o$ is essentially self-adjoint, 
and there is an orthonormal basis for  $L^2(\cA, \tau) \otimes \cS$ 
consisting of elements of
$\ciA \otimes \cS$ that are eigenvectors of $D_o$.
\end{proposition}

What is not so easy to see is that when $\cA$ is infinite-dimensional 
the eigenvalues of $D_o$, counted with multiplicity, 
converge in absolute value to $\infty$ (i.e. $D_o$
has ``compact resolvent''). A somewhat indirect 
proof of this fact is given in theorem 5.5 of \cite{GaG}. 

For $a \in \ciA$ let $M_a$ denote the operator on
$\ciA \otimes \cS$ corresponding to the left regular 
representation of $\ciA$ on itself. As seen in \cite{R4},
a simple calculation shows that 
for any $b \in \ciA$ and $\psi \in \cS$ we have
\[
[D_o, M_a](b \otimes \psi) = \sum (M_{\a_{E_j}(a)} \otimes \k_{\e_j})(b \otimes \psi)  .
\]
To simplify notation we will usually write $\a_{E_j}(a) $ instead of 
$ M_{\a_{E_j}(a)} $ from now on.    Then we see that we obtain:
\begin{proposition}
\label{dcom}
With notation as above, for any $a \in \ciA $ we have
\begin{equation*}
[D_o, M_a] = \sum \a_{E_j}(a) \otimes \k_{\e_j}  ,
\end{equation*}
acting on $\ciA \otimes \cS$. 
\end{proposition}
From this it is clear that $[D_o, M_a]$
is a bounded operator. Because of this, we can define a seminorm on
$\ciA$ that will play a central role in this paper.

\begin{notation}
\label{dnorm}
For notation as above, define 
$L^{D_o}$ on $\ciA$ by
\[
L^{D_o}(a) = \|[D_o, M_a]\|   .
\]
\end{notation}
We will develop the properties of $L^{D_o}$ in later sections.

The results described in this section are crucial for some of the proofs 
of our main results. But they are often not entirely adequate for the 
following reason. The usual definition of a first-order 
differential calculus requires that the bimodule
$\O$ be generated by the range of $\rd$.
As we will see, for many examples the bimodule $\O_o$
that we have used above does not have this property. 
But before turning to those examples, we will soon consider two 
classes of example for which $\O_o$ essentially 
does have this property. 
For now we will make 
a general observation about $ \O $ for our situation. Notice that its 
elements will be  finite sums of the form
$ \sum_k a_k \rd b_k $ for $a_k, b_k \in \ciA $.
(One uses the derivation property to show that right multiplication
of such sums by elements of $\ciA $ are again of this form.)
Let $\cK $ be a Hilbert space, and let $\cB(\cK) $ be the algebra of bounded 
operators on $\cK $. Let $\cA $ be a $* $-subalgebra of $\cB(\cK) $, and let 
$ D $ be a (possibly unbounded) self-adjoint operator on $\cK $ having the 
property that $[D,a] $ extends to a bounded operator on $\cK $ for any 
$a \in \ciA $ (so that if a few more axioms are satisfied, 
$(\cA, \cK, D) $ is a spectral 
triple as defined by Connes \cite{Cn7, Cn5, Cn3, CnMr}). 
We can view $\cB(\cK) $ as an 
$\cA $-$\cA $-bimodule in the evident way. Then the map
$\d: a \to [D,a] $ is a derivation from $\ciA $ into $\cB(\cK) $. Connes sets
$\O_ D $ to be the $\ciA $-$\ciA $-subbimodule of $\cB(\cK) $ generated by 
the range of $\d $ (see section 7.2 of \cite{Lndi}.), so that its elements are finite sums of the form
$ \sum_k a_k[D, b_k] $ for $a_k, b_k \in \ciA $. Connes views $\O_D$
as a space of first-order differential forms on $\cA$.

Let us see what $\O_{D_o} $ is for $D_o $ as defined earlier in this section.
Thus $\cK =\ciA\otimes \cS $, completed. For $ M $ denoting the left regular 
representation of $\cA $ on itself, we see that $ M\otimes \k $ is a 
representation of $\cA\otimes \bC\ell(\fg')$ on $\cK $. Its restriction to
$\ciA\otimes \fg'$ is an $\ciA $-$\ciA $-bimodule map of $\O_o $
into $\cB(\cK) $. For an element $ \sum_k a_k \rd b_k $ of
$\O $ we have
\[
(M\otimes \k)( \sum_k a_k \rd b_k ) =  \sum_k M_{a_k} \sum_j M_{E_j(b_k)} \otimes \k_{\e_j}
= \sum_k M_{a_k} [D_o, M_{b_k}].
\]
We thus obtain:

\begin{proposition}
\label{cond}
For notation as above, $\O_ {D_o} $ is the image of $\O_o $ under the map $ M\otimes \k $.
\end{proposition}


\section{Invariance under the group actions}
\label{gpact}

Since the action $\Cad$ of $G$ on $\fg'$ is by orthogonal operators 
for our chosen inner product, it extends to $\bC\ell(\fg')$ as an 
action by $*$-algebra automorphisms 
(Bogoliubov automorphisms), which we still denote by $\Cad$. 
Since $G$ acts on both $\cA$ and $\bC\ell(\fg')$, we let $\g$ denote the 
corresponding diagonal action $\a \otimes\Cad$  on $\cA \otimes \bC\ell(\fg')$.
It is just an extension of our earlier action $\g$ of $G$ 
on $\O_o$ defined in Notation \ref{act}. 
It carries the smooth version into itself. 

We want 
to have a representation of G by unitary operators on the Hilbert-space
completion of
$\ciA \otimes \cS$ that is compatible with the action $\a$ on $\cA$, and 
that commutes with $D_o$. As discussed in \cite{GaG}, the action 
of $G$ on  $\cS$ corresponding to the action of $G$ 
on $ \bC\ell(\fg')$ can, in general, at best be implemented by a 
projective representation on $\cS$. In \cite{GaG} it is shown 
how to handle the case of projective representations, but it is 
also remarked there that if $G$ is a simply connected semisimple 
compact Lie group then every projective representation of $G$ 
is equivalent to an ordinary (unitary) representation of $G$, so
that for these groups there is a unitary representations on $\cS$ 
corresponding to the action of $G$ on $ \bC\ell(\fg')$. 

So for the moment, we will simply assume that there is a  
unitary representation, $\s$, of $G$ on $\cS$ that is 
compatible with the action $\Cad$ of $G$ on $ \bC\ell(\fg')$ and 
the representation $\k$ of  $ \bC\ell(\fg')$ on $\cS$ in the 
sense that  
\begin{equation}
\label{compt}
\k_{\Cad_x(q)} = \s_x \k_q \s^*_x
\end{equation}
for all $q \in \bC\ell(\fg')$ and $x \in G$. We then define a 
representation, $\tilde \s$, of $G$ on $\ciA \otimes \cS$
by $\tilde \s = \a \otimes  \s$. 

\begin{proposition}
\label{commut1}
With assumptions and notation as above, 
the operator $D_o $ on $\ciA \otimes \cS$ 
commutes with the representation $\tilde \s$ of $G $ 
on $\ciA \otimes \cS$ .
\end{proposition}

\begin{proof}
Let $a \in \ciA$ and $\psi \in \cS$ and $x \in G $. Then
\begin{align*}
D_o(\tilde \s_x(a \otimes \psi)) &= D_o(\a_x(a) \otimes \s_x(\psi))
= \sum \a_{E_j}(\a_x(a))) \otimes \k_{\e_j}(\s_x(\psi))    \\
    &= \sum \a_x(\a_{\Ad_x^{-1}(E_j)}(a)) 
    \otimes \s_x(\k_{\Cad_x^{-1}(\e_j)}(\s_x^{-1}(\s_x(\psi))))   \\
&= (\a_x \otimes \s_x)
(\sum (\a_{\Ad_x^{-1}(E_j)}(a)) \otimes (\k_{\Cad_x^{-1}(\e_j)}(\psi))   \\
&= \tilde \s_x(D_o(a \otimes \psi)),
\end{align*}
where for the third equality we have used the compatibility condition, 
and for the fifth equality we have used the fact that $\{\Ad_x^{-1}(E_j)\} $
is equally well a basis for $\fg$, with dual basis $\{\Cad_x^{-1}(\e_j)\} $.
\end{proof}
The representation $\tilde \s$ is unitary, so it extends to a unitary
representation on the Hilbert space completion 
of $\ciA \otimes \cS$. We will discuss further the possible existence 
of the representation $\s $ in Section \ref{secsp}.

\begin{proposition}
\label{invar}
With assumptions and notation as above, 
the seminorm $L^{D_o} $ defined in Notation \ref{dnorm}
is invariant under the action $\a $ in the sense that
\[
 L^{D_o}(\a_x(a)) =  L^{D_o}(a) 
\]
for all  $a \in \ciA $ and $x \in G$. 
\end{proposition}

\begin{proof}
Note that for any $a, b \in \ciA $, $\psi \in \cS$ and $x \in G$ we have
\begin{align*}
 (\tilde \s_x M_a  \tilde \s_x^{-1})(b\otimes \psi)  
& =  (\a_x \otimes\s_x)(M_a \otimes I^\cS) (\a_x^{-1} \otimes \s_x^{-1} )(b\otimes \psi)  \\
&= (\a_x (a (\a_x^{-1} (b)) \otimes \psi = M_{\a_x(a)} ( b\otimes \psi) ,
\end{align*}
where we have used $M$ to denote the left action of $\cA$ on both
$\cA$ and $\cA \otimes \cS$. Thus
\[
(\tilde \s_x M_a  \tilde \s_x^{-1}) = M_{\a_x(a)} .
\]
Since $\tilde \s_x $ commutes with $D_o$, it follows that
\[
[D_o, M_{\a_x(a)} ] = D_o \tilde \s_x M_a  \tilde \s_x^{-1}  
- \tilde \s_x M_a  \tilde \s_x^{-1} D_o               
=  \tilde \s_x [D_o, M_a]  \tilde \s_x^{-1}  .
\]
Since $ \tilde \s_x $ is a unitary operator for the inner product on
$\cA \otimes \cS$, we obtain the desired equality.
\end{proof}


\section{Charge Conjugation}
\label{charge}

In \cite{GaG} it is shown that a Dirac operator constructed in the way described above possesses an  
important structure on its domain, namely, a ``real structure'',
related to a charge conjugation operator, as first used by Connes \cite{Cn3}
for non-commutative geometry. 
We do not need to use this structure later,
so we will give here only a brief description of it, with references to 
the literature for further details.

Since $ \bC\ell(\fg')$ is the complexification of the Clifford algebra 
$C\ell(\fg')$ over $\bR $, it has the standard complex conjugation operator,
which is a conjugate-linear algebra automorphism. We denote it by
$q \mapsto \bar q$ for $q \in \bC\ell(\fg')$. Let $\chi $ denote
the usual grading automorphism on $ \bC\ell(\fg')$ determined by
$ \chi(\mu) =-\mu $ for $\mu\in \fg'$. By definition, the charge conjugation
at the level of the Clifford algebra is the conjugate-linear automorphism,
$c $, obtained by composing complex conjugation with $\chi $. Note that on
the even subalgebra
$ \bC\ell^e(\fg')$ it is just complex conjugation.

But the operator that we need is a conjugate-linear operator 
$\mathsf{C}_\cS$ on $\cS $ that implements
$c $ for the representation $\k $, that is, such that
\[
\k_{c(q)} = \mathsf{C}_\cS \k_q \mathsf{C}_\cS^{-1}
\]
for all $q \in \bC\ell(\fg')$. Furthermore, the operator $\mathsf{C}_\cS $ is required 
to respect the inner product on $\cS $ in the sense that
$\< \mathsf{C}_\cS (\psi), \mathsf{C}_\cS\phi \> =\<\phi, \psi\> $ 
for all $\psi, \phi \in \cS $.
Then $\mathsf{C}_\cS $ is unique up to a scaler multiple of modulus 1.
It is normalized to satisfy  $\mathsf{C}_\cS^2=\pm I^\cS $, where the sign 
depends on $\dim(\fg') $.

Then the charge-conjugation operator, $ \mathsf{C} $, on 
$\cA \otimes \cS$ is defined on elementary tensors by 
\[
 \mathsf{C}(a \otimes \psi) = a^* \otimes \mathsf{C}_\cS(\psi) .
\] 
It is conjugate-linear,
and respects the $\cA $-valued inner product on $\cA \otimes \cS$.
Its most important property is that
\[
[a, \mathsf{C}b^* \mathsf{C}^{-1}] = 0
\]
for all $a,b\in \cA $, so that $b\mapsto \mathsf{C}b^* \mathsf{C}^{-1} $ 
gives a right action of
$\cA $ on $\cA \otimes \cS$ such that $\cA \otimes \cS$ is an
$\cA $-$\cA $-bimodule. When $\dim(\fg') $ is even we also have
$ \mathsf{C}\bg =\pm \bg \mathsf{C} $, 
where the sign depends on $\dim(\fg') $.
All this is summarized by saying that the charge conjugation operator
$ \mathsf{C} $ provides a ``real structure'' on $\cA \otimes \cS$.

For the Dirac operator $D_o $ constructed above we then have the very 
important condition that
\[
[[D_o,a] , \mathsf{C}b^* \mathsf{C}^{-1}] = 0
\]
for all $a,b\in \ciA $. It is called the ``first-order condition'', and reflects 
the fact that $ D_o $ is like a differential operator of order one. Furthermore,
we have $ \mathsf{C}D_o =\pm D_o\mathsf{C} $, where again the sign 
depends on $\dim(\fg') $.  When $\cH $ is the Hilbert-space completion of
$\cA \otimes \cS$, the triple $(\cA, \cH, D_o) $ is an example of the notion of
``spectral triple'' introduced by Connes, and $(\cA, \cH, D_o, C) $
is an example of a ``real spectral triples". For details about all of this see
\cite{Brt, Cn3, Cn5, CnMr, DbD, GVF, Vrl, vSj}. 

\begin{remark}
\label{ques1}
The results obtained in these first sections suggest the following questions:

1. For a given compact connected Lie group $G$, how can one characterize which of its ergodic 
actions $(\cA, \a)$ have the property that the sub-bimodule, $\O$, of $\O_o = \ciA \otimes \fg'$
generated by the range of the derivation $\rd $ is (finitely generated) projective as a
right $\ciA $-module?

2. Among those actions for which $\O $ is a projective module, how does one characterize those
such that $\O $ also admits a ``real structure'' of the kind sketched above? 

I have not investigated these questions.

Note that much concerning the classification of ergodic actions
of connected compact Lie groups is a mystery.
Good answers are known only for $G$ commutative \cite{OPT},
or for $ G= SU(2)$ \cite{Wss} (and subsequent papers) as far as I know.
\end{remark} 


\section{Casimir operators}
\label{cas}

Dirac found his famous equation that predicted the existence of the positron 
because he was looking for a first-order 
differential operator that is a square root of the Klein-Gordan operator, 
which is the appropriate
Laplace-type operator for flat space-time. So it is of interest to examine 
the square of the
Dirac operator that we have defined above to see whether it is related to a 
Laplace operator. For a Lie group, $G$, whose Lie algebra 
$\fg$ has a chosen $Ad$-invariant
(non-degenerate) inner product, the appropriate Laplace operators 
come from the degree-2 Casimir element in the universal enveloping 
algebra $\mathcal{U} (\fg)$ of $\fg$.
The Casimir element depends on the choice of inner product. Let
$\{E_j \} $ be a basis for $\fg$ that is orthonormal for the chosen inner product.
Then the Casimir element, $ C $, is defined by $ C = \sum_j (E_j)^2$.
(We do not include the minus sign which is often used, so the image
of $C$ by the infinitesimal version of a 
unitary representations of $ G $ will be a non-positive operator.).
Then $C$ is in the center of  $\mathcal{U} (\fg)$, 
and so for any irrep $(U, \cH)$
of $G $, with corresponding representation of $\fg $ and $\mathcal{U} (\fg)$, 
the operator
$U_ C $ will be a scalar multiple of the identity operator on $\cH $. 
We will find this useful in Section \ref{fuzzy} for determining the spectrum 
of the Dirac operator when $ G =SU(2)$. Anyway, for $D_o $ acting on
$\ciA \otimes \cS$ as $D_o = \sum \a_{E_j} \otimes \k_{e_j} $ we have
(as in the first displayed equation in the proof of 
theorem 6.3 of \cite{GaG}):

\begin{proposition}
\label{square}
Let $\{\e_j\}$ be the basis for $\fg'$ dual to the orthonormal basis $\{E_j\}$
for $\fg$. Then
\[
D_o^2 = -\a_C \otimes I^\cS + \sum_{j<k} \a_{[E_j,E_k]} \otimes \k_{\e_j\e_k}
\]
\end{proposition}

\begin{proof}
Since in $\bC\ell(\fg')$ we have $\e_j \e_k = - \e_k \e_j$ and $\e_j^2 = -1$, we have
\begin{align*}
D_o^2 &= \sum_{j,k} \a_{E_j}\a_{E_k} \otimes \k_{\e_j}\k_{\e_k} \\
&= \ \sum_j (\a_{E_j})^2 \otimes (\k_{\e_j})^2  + \sum_{j<k}  \a_{E_j}\a_{E_k} \otimes \k_{\e_j \e_k} 
+  \sum_{j>k}  \a_{E_j}\a_{E_k} \otimes \k_{\e_j \e_k}   \\
&= -\a_C\otimes I^\cS + \sum_{j<k} ( \a_{E_j}\a_{E_k} -  \a_{E_k}\a_{E_j})  \otimes \k_{\e_j \e_k}  \\
&= -\a_C\otimes I^\cS + \sum_{j<k} \a_{[E_j,E_k]} \otimes \k_{\e_j \e_k} \  ,
\end{align*}
as desired.
\end{proof}

It is appropriate to view $\a_C$ as the Laplace operator on $\ciA$, and so the term
$ -\a_C\otimes I^\cS$ is very analogous to what one obtains for the square of the Dirac operator
on flat $\bR^d$. The second term can be viewed as some kind of curvature term
analogous to the curvature term in the Lichnerowicz formula \cite{BGV, Frd, LwM},
but I do not know how to define a general version of curvature that
would make this precise.

For later use in Section \ref{fuzzy} 
we need the following result about the image of $ C $ under the 
representation
$\tilde \s = \a \otimes  \s$ defined just before Proposition \ref{commut1}.
We let $\s $ also denote the corresponding representation of $\fg $.
Then $\tilde \s $, as a representation of $ \fg $, is given by
$\tilde \s_X =\a_X \otimes I^\cS + I^\cA \otimes \s_X $ for any $ X\in \fg $,
where $I^\cS$ and $I^\cA$ denote the identity operators on $\cS $ 
and $\cA $ respectively.
 
\begin{proposition}
\label{bigcas}
For notation as above, we have
\[
\tilde \s_C = \a_C \otimes I^\cS +  2 \sum_j (\a_{E_j} \otimes \s_{E_j})  + I^\cA \otimes \s_C  ,
\]

\end{proposition}
\begin{proof}
\begin{align*}
\tilde \s_C &= \sum_j (\a_{E_j} \otimes I^\cS + I^\cA \otimes \s_{E_j})^2    \\
&=  \sum_j (\a_{E_j})^2 \otimes I^\cS + 2 \sum_j (\a_{E_j} \otimes \s_{E_j}) + \sum_j  I^\cA \otimes (\s_{E_j})^2  \\
&=  \a_C \otimes I^\cS + 2 \sum_j (\a_{E_j} \otimes \s_{E_j}) +  I^\cA \otimes \s_C .
\end{align*}
\end{proof} 
The operator $ \sum_j (\a_{E_j} \otimes \s_{E_j}) $ from the middle term above, looks somewhat
like our expression for the Dirac operator. We will see in Section \ref{fuzzy} that for the case of
$G = SU(2) $ it does in fact essentially coincide with the Dirac operator. This gives us
a strong tool for computing the spectrum of the Dirac operator.

The results above have some resonance with results in the neighborhood of equation 2.7, 
theorem 2.13 and theorem 2.21 of
\cite{Kst} relating Dirac operators and Casimir elements, but that paper is aimed only at
homogeneous spaces $G/K $ where $G $ is a compact semisimple group and
$K $ is a connected subgroup of $ G $ whose rank is the same as that of $ G $.
That setting is very pertinent to the coadjoint orbits that we will consider later.


\section{First facts about spinors}
\label{fspin}

Let $\fm$ be a finite-dimensional Hilbert space over $\bR$, and
let $\bC\ell(\fm)$ denote the complex 
Clifford algebra for $\fm$, 
much as discussed above for $\fg'$. 
For its standard involution $\bC\ell(\fm)$ is a C*-algebra.

There is a special 
element, $\bg$, of $ \bC\ell(\fm)$, called the ``chirality element" \cite{GVF},
that is a suitably normalized 
product of all the elements of a basis for 
$\fm \subset \bC\ell(\fm)$, and that has the properties that
$\bg^2 = 1$, $\bg \neq 1$ and $\bg^* = \bg$. 
Let $n = \dim(\fm)$. 
If $n$ is even then $ \bC\ell(\fm)$ is
isomorphic to a full matrix algebra. For this case $\bg$
splits an irreducible representation of $\bC\ell(\fm)$ 
into two subspaces of equal dimension that are carried 
into themselves by the subalgebra $\bC\ell^e(\fm)$ of even elements of 
$\bC\ell(\fm)$. There is a standard way of explicitly constructing 
an irreducible representation of $\bC\ell(\fm)$ on a Fock-space.
We will need to use some of the main steps of that construction 
in Section~\ref{mspin}.

When $n$ is odd, $\bg$ is in the center of $\bC\ell(\fm)$
and it splits $\bC\ell(\fm)$ into the direct sum of 2
full matrix algebras. Thus up to equivalence $\bC\ell(\fm)$
has two inequivalent irreducible representations, neither of which is faithful. 
The subalgebra,  $\bC\ell^e(\fm)$, of even element of  $\bC\ell(\fm)$ 
is itself a Clifford algebra 
on a vector space of even dimension, and so is a full matrix algebra,
which has a unique irreducible representation (up to isomorphism). 
We will view the two irreducible representations of $\bC\ell(\fm)$
as being the irreducible representation of 
$\bC\ell^e(\fm)$, extended to  $\bC\ell(\fm)$ by sending  $\bg $
to $ +I $ for one of them, and to $-I $ for the other. (Notice that the restriction to
$\bC\ell^e(\fm)$ of any (unital) representation of  $\bC\ell(\fm)$ will be a faithful 
representation of $\bC\ell^e(\fm)$.)

Now let $\k $ be a representation of $\bC\ell(\fm)$ on 
a finite-dimensional Hilbert space $\cS $. 
To deal with the fact that $\k $ need not be a faithful representation
of $\bC\ell(\fm)$, 
we need the following technical result.

\begin{lemma}
\label{oddsp}
Let notation be as above, and let $\e_1, \cdots, \e_p $ be elements of $\fm $.
Let $\cD $ be any unital C*-algebra, and let $(\cK, M) $ be a faithful 
representation of $\cD $. Let
$d_1, \cdots, d_p $ be elements 
of $\cD $, and let $ t =\sum d_j \otimes \e_j $, viewed as an 
element of the C*-algebra $\cD \otimes \bC\ell(\fm)$. Then 
\[
\|(M \otimes \k )(t)\| =\|t\|   .
\]
\end{lemma}

\begin{proof}
Since $ \bC\ell(\fm)$ is finite dimensional, there is a unique
C*-algebra norm on $\cD \otimes \bC\ell(\fm)$.
If the dimension $n $ of $\fm $ is even, then the representation
$\k$ of $\bC\ell(\fm)$ must be faithful, and so the homomorphism
$M \otimes \k $ between C*-algebras must be faithful and so 
isometric. This gives the desired result.

If $n $ is odd, then
because $\bg $ is an odd unitary element of $\bC\ell(\fm)$, left
multiplication by $\bg $ is an isometry from the odd subspace of
$ \bC\ell(\fm)$ onto $ \bC\ell^e(\fm)$. Thus multiplication by
$ 1_\cD \otimes \bg $ is an isometry that carries $t $ into
$\cD \otimes  \bC\ell^e(\fm)$. 
But the restriction of $\k $ to 
$ \bC\ell^e(\fm)$ is faithful, and so $M \otimes \k $ is faithful
and so an isometry from 
$\cD \otimes \bC\ell^e(\fm)$ into $\cB(\cK) \otimes \cB(\cS) $ .
Consequently
\[
\|t\| = \|(1_\cD \otimes \bg )t\| =      \|(M \otimes \k)( 1_\cD \otimes \bg)  t\| 
=\|(M \otimes \k)(t)\|,
\]
since $\k_{\bg} =\pm I^\cS $.
 \end{proof}

The following proposition, which will be of use later, 
is a typical way in which we will apply the above Lemma.

\begin{proposition}
\label{prodnorm}
Let $\cA $ be a unital C*-algebra, and let $\a $ be an ergodic
action of the Lie group $G $ on $\cA $. Fix a choice of a 
finite-dimensional spinor 
space $\cS $, and use it to construct the Dirac operators $ D_o^\cA $ 
for $\cA $ as in equation \ref{eqdefdir}. Then
\[
 L^{D_o^\cA}(a) = \|[D_o, M_a]\| =  \|\sum \a_{E_j}(a) \otimes {\e_j}\|,
\]
where the norm on the right is that on the C*-algebra $\cA \otimes \bC\ell(\fg')$.
Thus $ L^{D_o^\cA} $  is independent of the choice of the spinor 
space $ \cS $.
\end{proposition}
\begin{proof}
We use Lemma \ref{oddsp} for the last equality in
\begin{align*}
\|[D_o, M_a]\| &= \|\sum M_{\a_{E_j}(a)} \otimes \k_{\e_j}\|    \\
&=   \|(M \otimes \k)(\sum \a_{E_j}(a) \otimes {\e_j})\|   
= \|\sum \a_{E_j}(a) \otimes {\e_j}\| .
\end{align*}
\end{proof}

The following proposition will be important for the proof of our main theorem.
In that proof the $\theta $ here will be a Berezin symbol map.

\begin{proposition}
\label{ucplip}
Let $\cA $ and $\cB $ be unital C*-algebras, and let $\a $ and $\b $ be ergodic
actions of the Lie group $G $ on $\cA $ and $\cB $. Fix a choice of a 
finite-dimensional spinor 
space $\cS $, and use it to construct Dirac operators $ D_o^\cA $ and 
$ D_o^\cB $ for $\cA $ and $\cB $ as above. 
Let $L^{D_o^\cA} $ and $ L^{D_o^\cB} $ be the seminorms determined 
by $ D_o^\cA $ and $ D_o^\cB $ as
in Notation \ref{dnorm}. 
Let $\theta $ be a 
unital completely positive operator from
$\cA $ to $\cB $ that intertwines the actions $\a $ and $\b $,
so that it carries $\ciA $ into $\ciB $.
Then for any $a\in\ciA $ we have
\[
 L^{D_o^\cB} (\theta(a))  \leq L^{D_o^\cA}(a)  .
\]
\end{proposition}

\begin{proof}
According to Proposition \ref{prodnorm}
\begin{align*}
L^{D_o^\cB}({\theta(a)}) &= \|\sum \b_{E_j}(\theta(a)) \otimes \e_j\| =
\|\sum \theta(\a_{E_j}(a)) \otimes \e_j \|  \\
&= \|(\theta\otimes I^{\bC\ell(\fg')})\sum \a_{E_j}(a) \otimes \e_j \| 
\leq \|\sum \a_{E_j}(a) \otimes \e_j \|  .
\end{align*}
because $\| (\theta\otimes I^{\bC\ell(\fg')}) \| =1$ since
$\theta$ is a unital completely positive operator.
(See section II.6.9 of \cite{Blk2}.)
\end{proof}

\begin{proposition}
\label{ucp}
With the assumptions of Proposition \ref{ucplip},
let $\hat \theta =\theta\otimes I^\cS $, viewed as a 
map from $\ciA \otimes \cS $ to 
$\ciB \otimes \cS $.
Then
\[
\hat \theta D_o^\cA =D_o^\cB \hat \theta
\] 
as operators from $\ciA \otimes \cS $ to 
$\ciB \otimes \cS $. Consequently
 if $\Psi $ is an eigenvector for $D_o^\cA $ with eigenvalue $\l $ then 
$\hat \theta \Psi $ is an eigenvector for $D_o^\cB $ with eigenvalue $\l $ if
$\hat \theta \Psi \neq 0$.
\end{proposition}

\begin{proof}
For $a\in\ciA $ and $\psi\in\cS $ we have
\begin{align*}
\hat \theta (D_o^\cA(a \otimes \psi)) &=\hat \theta(\sum \a_{E_j}(a)\otimes \k_{\e_j}(\psi))
= \sum \theta(\a_{E_j}(a))\otimes \k_{\e_j}(\psi)  \\
& =  \sum \b_{E_j}(\theta(a))\otimes \k_{\e_j}(\psi)  = D_o^\cB (\hat \theta(a\otimes \psi)),
\end{align*}
as needed.
\end{proof}


\section{The C*-metrics}
\label{cmet}
In this section we examine the C*-metrics that are determined 
by the Dirac operators constructed in Section \ref{sec1}.
In the literature there are small variations in the definition of a ``C*-metric".
The following definition is appropriate for this paper. For this purpose 
a ``C*-normed $* $-algebra'' means a normed $* $-algebra whose 
norm satisfies the C*-algebra identity $\|a^*a\| =\|a \|^2 $, so that its 
completion is a C*-algebra. For us the main class of examples 
is $\ciA$ as used earlier.

\begin{definition}
\label{defcmt}
Let $\cA $ be a unital C*-normed $* $-algebra. By a C*-metric on 
$\cA $ we mean a seminorm $ L $ on $\cA $ having the following properties.
For any $a, b\in \cA $:
\begin{enumerate}
\item $ L(a) = 0 $ if and only if $a\in \bC 1_\cA $.
\item $ L(a^*) = L(a) $.
\item $ L $ is lower semi-continuous with respect to the norm of $\cA $.
\item $ L $ satisfies the Leibnitz inequality
\[
L(ab) \leq L(a)\|b\| + \|a\| L(b).
\] 
\item Let $ S(\cA) $ be the state space of $\cA $. Define a metric, $\rho_L$, 
on $ S(\cA) $ by
\begin{equation}
\label{stmet}
\rho^L(\mu, \nu) = \sup\{|\mu(a)-\nu(a)|: a^* = a \ \mathrm{and} \ L(a) \leq 1\}.
\end{equation} 
Without further hypotheses this metric can take value $+\infty$. We require 
that the topology on $ S(\cA) $
from this metric coincide with the weak-$*$ topology on $S(\cA) $. Then,
in particular,
$\rho^L$ will never take value $+\infty$. (The condition $a^* = a$ in the
definition of $\rho^L(\mu, \nu)$ can be omitted without changing
$\rho^L(\mu, \nu)$, as explained just before definition 2.1 of \cite{R6}.)
\end{enumerate}    
If $\cA $ is actually a C*-algebra, and if $ L $ is a seminorm on $\cA $
that is permitted to take the value $+\infty$, but is semi-finite in the sense that
$\cA_f =\{a: L(a) < \infty\}$ is dense in $\cA $, and if the restriction of
$L $ to $\cA_f $ satisfies the 5 properties above, then we will also call
$ L $ a C*-metric (on $\cA$).    \newline
\indent
A pair $(\cA, L)$ consisting of a unital C*-normed $*$-algebra $\cA$ and a C*-metric
$L$ on $\cA$ is called a \emph{compact C*-metric space}.
\end{definition}    

Here is the motivating example. Let $ (X,\rho) $ be a compact metric space, 
and let $ A $ be the C*-algebra $ C (X) $ of all continuous complex-valued 
functions on $X $. Let $L^\rho $ assign to each function its Lipschitz constant,
that is,
\[
L^\rho (f) =\sup\{|f(x)-f(y|/\rho(x,y): x,y \in X \ \mathrm{and} \ x \neq y\}.
\]
Then $ L^\rho$ is a C*-metric. Furthermore, one can recover $\rho $ from
$L^\rho$. To see this, notice that the state space $ S (A) $ is just the 
set of probability measures on $X $. Let $\rho^{L^\rho}$ be the metric on
$ S (A) $ defined by equation \eqref{stmet} for $ L^\rho$. Then
$\rho(x,y) = \rho^{L^\rho}(\d_x,\d_y)$, where $\d_z$ denotes
the delta-measure at $z$ for any $z \in X$.

We remark that property 5 is often the most difficult to verify for examples, 
but having $S(\cA)$ compact for the $\rho^L$-topology (property 
5) is crucial for the definitions of quantum Gromov-Hausdorff
distance which we will discuss later. 

We will not explicitly need all of the next few remarks, 
but they give important context to the definition of C*-metrics.
Suppose that $ L $ is a C*-metric on a 
unital C*-normed $* $-algebra $\cA $, and let
\begin{equation}
\label{ball}
\cL_\cA^1 = \{a \in \cA:  L(a) \leq 1\}   .
\end{equation}
Let ${\bar \cA}$ be the completion of $\cA $, 
let ${\bar \cL}_\cA^1$ be the closure of $\cL_\cA^1$ in the C*-algebra 
${\bar \cA}$, and let
${\bar L}$ denote the corresponding ``Minkowski functional'' 
on ${\bar \cA}$, defined by setting, for $c \in {\bar \cA}$,
\[
{\bar L}(c) = \inf\{r \in \bR^+: c \in r{\bar \cL}_\cA^1\},
\]
with value $+\infty$ if there is no such $r$.  
Then ${\bar L}$ is a seminorm on
${\bar \cA}$ (often taking value $+\infty$), 
and the proof of proposition~$4.4$ of \cite{R5} tells us
that because $L$ is lower semicontinuous, ${\bar L}$ is an extension
of $L$.  We call ${\bar L}$ the {\em closure} of $L$.  We see
that the set $\{c \in {\bar \cA}: {\bar L}(c) \le 1\}$ is closed in
${\bar \cA}$.  We say that the original seminorm $L$ on $A$ is {\em
closed} if $\cL_\cA^1$ is already closed in ${\bar A}$, or, equivalently, is
complete for the norm on $A$.  Clearly if $L$ is closed, then it is
lower semicontinuous.  If $L$ is closed and is not defined on all of
${\bar A}$, then ${\bar L}$ is obtained simply by giving it value
$+\infty$ on all the elements of ${\bar A}$ that are not in $A$.  It is
clear that if $L$ is semifinite then so is ${\bar L}$. We recall that
a unital subalgebra $B$ of a unital algebra $A$ is said to be spectrally
stable in $A$ if for any $b \in B$ the spectrum of $b$ as an element
of $B$ is the same as its spectrum as an element of $A$, or equivalently,
that any $b$ that is invertible in $A$ is invertible in $B$. From
proposition 3.1 of \cite{R21} one easily obtains:

\begin{proposition}
\label{cleib}
Let $L$ be a C*-metric on a 
unital C*-normed $* $-algebra $\cA $.
Then the closure of $L$ is a C*-metric.
If $L $ is a closed C*-metric, then
 $A^f$ is a spectrally-stable subalgebra of
$\bar A$ that is carried into itself under the 
holomorphic functional calculus of $\bar A$.
\end{proposition}

We now continue our discussion of Dirac operators.
Our discussion is very close to that given in section 4 of 
\cite{R4}, but part of it will be 
general enough to also apply in later sections.

We continue with the notation of Section \ref{sec1}.
Thus, for $a\in\ciA $,
\[
[D_o, M_a] = \sum \a_{E_j}(a) \otimes \k_{\\e_j}   ,
\]
acting on $\ciA \otimes \cS$, and we define a seminorm, $L^{D_o}$, on $\ciA$ by
\[
L^{D_o}(a) = \|[D_o, M_a]\|.
\] 
It is shown in theorem 4.2 of \cite{R4} that 
$L^{D_o}$ satisfies property 5 of Definition \ref{defcmt}.
This will also follow from the considerations below.

For the proof of our main theorem we need certain 
bounds on $ L^{D_o} $. For this purpose
we define a different seminorm, $L_d $, on $\ciA$ by
$
L_d(a) = \|\rd a\|,
$
where we view $\rd a$ as a linear transformation from $\fg$
to $\ciA$, each of which is a normed space, with the norm on $\fg$
coming from its inner product dual to the 
inner product on $\fg'$. Thus
\begin{equation}
\label{dif}
L_d(a) = \sup\{\|\a_X(a)\| : X \in \fg\ \ \mathrm{and} \ \|X\| \leq 1\}.
\end{equation}
It is shown in theorem 3.1 of \cite{R4} that 
$L_d$ satisfies property 5 of Definition \ref{defcmt}.
It is easy to check the other conditions of Definition \ref{defcmt}, 
and thus we conclude that $L_d$ is in fact a C*-metric.

We will obtain bounds on $ L^{D_o} $ in terms of $L_d$. For this purpose 
we can for convenience choose the basis vectors $\e_j$
for $\fg'$ to be orthonormal. Then as elements of $ \bC\ell(\fg')$
they will satisfy the relations $\e_j^* = -\e_j$, $\e_j^2 = -1$ and
$\e_j\e_k = -\e_k\e_j$ if $j \neq k$. The following lemma is probably 
well known, but I have not seen it in the literature.

\begin{lemma}
\label{cliff}
Let $\cA$ and $\cC$ be unital C*-algebras, and let
$\e_1, \dots, \e_m$ be elements of $\cC$ that satisfy
the relations $\e_j^* = -\e_j$, $\e_j^2 = -1$ and
$\\e_j\e_k = -\e_k\e_j$ if $j \neq k$. Let $a_1, \dots, a_m$
be elements of $\cA $, and let $t = \sum a_j \otimes \e_j$, 
an element of $A\otimes C$ (for any C*-norm). 
Then
\[
\sup\{\|a_j\|: 1\leq j \leq m\} \ \leq \|t\| \leq \sum\{\|a_j\|: 1\leq j \leq m\} .
\]
\end{lemma}

\begin{proof}
For each $ j $ let $p_j$ be the spectral projection of $\e_j$ for the 
eigenvalue $+ i$.  Note that $p_j \neq 0$. From the third 
relation above one quickly sees that $p_j \e_k p_j = 0$ if $j \neq k$.
It follow that
$(1\otimes p_j)t(1\otimes p_j) = i(a_j \otimes p_j)$,
so that $\|a_j\|\leq \|t\|$, and the first inequality holds,
The second inequality holds immediately from the definition of $t$.
\end{proof}

\begin{proposition}
\label{ineq}
In terms of the notation used before the lemma, for
any 
$a \in \ciA$ we have 
\[
L_d(a) \leq L^{D_o}(a) \leq nL_d(a)  ,
\]
where $n = \dim(\fg)$.
\end{proposition}

\begin{proof}
According to Proposition \ref{prodnorm} we have 
$L^{D_o}(a) = \| \sum \a_{E_j}(a) \otimes {\e_j}\| $,
where the norm is that on $\cA \otimes \bC\ell(\fg')$.
So we can apply the above Lemma with with $\cC = \bC\ell(\fg') $. 
Any $X \in \fg$ with $\|X\| = 1$ can serve as one element, say $E_1$, of an 
orthonormal basis for $\fg$. Thus for any such $X$ we conclude 
from the above Lemma that
$\|\a_X(a)\| \leq \|\sum \a_{E_j}(a) \otimes \e_j \| \leq \sum \|\a_{E_j}(a)\|$.
\end{proof}

From the inequalities in Proposition \ref{ineq} it is easily seen that the 
metric on $S(\cA) $ from $ L^ {D_o} $ is equivalent (not necessarily equal) 
to the metric from $ L_d$. 
Consequently we conclude, as in theorem 4.2 of \cite{R4}, that $ L^ {D_o} $ itself
satisfies property 5 of Definition \ref{defcmt}. 
Furthermore, $L^{D_o}$ is lower semi-continuous, as seen in example 2.4
of \cite{R21}. It is easily verified that $ L^{D_o} $ satisfies the Leibniz inequality, and even 
is strongly Leibniz in the sense defined in definition 1.1 of \cite{R21}. Thus 
we have obtained:

\begin{proposition}
\label{dirmet}
With notation as above, 
$ L^{D_o} $ is a C*-metric on $\cA$.
\end{proposition}

But there is an even more 
basic general C*-metric on $\cA$, 
which plays the principal role in \cite{R4, R7, R29},
and which will be crucial for the proof of our main theorem.
From the chosen inner product on $\fg'$, and its
dual on $\fg $, we obtain a 
Riemannian metric on $G $, which is both right and left-invariant
because the inner product is  $\Cad $-invariant.
From that Riemannian metric we obtain  
a corresponding continuous length-function, 
$\ell $, on $ G $ (coming from path lengths determined 
by the Riemannian metric), which is constant on conjugacy classes. 
For any ergodic action $\a $ of $ G $
on a unital C*-algebra $\cA $ we define a seminorm $ L_\ell $ on $\ciA $ by
\begin{equation}
\label{defell}
L_\ell(a) =  \sup\{\|\a_x(a)- a\|/\ell(x) : x \in G\ \ \mathrm{and} \ x  \neq 0\}.
\end{equation}
(where the value $ +\infty $ is permitted).
Notice that this seminorm is well-defined for any continuous length-function
on $ G $ (of which there are many), and that $G $ need not be a Lie group. 
In fact, for any compact group $G $, any continuous length-function on $ G $
(which for our definition is constant on conjugacy classes),
and any ergodic action of $ G $ on any unital C*-algebra $A $,
theorem 2.3 of  \cite{R4} tells us that
$L_\ell $ satisfies property 5 of Definition \ref{defcmt}.
It is easy to check the other conditions of Definition \ref{defcmt}, 
and thus we conclude that $L_\ell $ is in fact a C*-metric. 

We then have the following inequality. The short proof is contained in
the proof of theorem 3.1 of  \cite{R4}, but we include the 
proof here for the reader's convenience.

\begin{proposition}
\label{metineq}
With notation as above, for any $a \in \ciA$ we have 
\[
L_\ell(a) \leq L_d(a)  .
\]
\end{proposition}

\begin{proof}
Let $a\in \ciA$, and let $c$ be a smooth path in $G$ 
from $e_ G$ to $x\in G$. Then $\phi$, defined
by $\phi(t)=\a_{c(t)}(a)$, is smooth, and so we have 
\begin{align*}
\|\a_x(a)-a\| &=\|\int\phi'(t)dt\| \leq
\int \|\a_{c(t)} (\a_{c'(t)}a)\|dt    \\
&= \int \|\a_{c(t)} (\rd a(c'(t))\|dt \leq \|da\| \int \|c'(t)\|dt \ .
\end{align*}
But the last integral is just the length of $c$. Thus from the definition
of the ordinary metric on $G$ as in infimum over all smooth paths, 
with its length function $\ell$ using paths from $e_ G$ to $x\in G$, 
we obtain $\|\a_x(a)-a\|\leq \|da\|\ell(x) $. Thus for all $x\in G$
\[
\|\a_x(a)-a\|/\ell(x)        \leq \|da\| \ ,
\]
from which the desired result follows immediately.
\end{proof}

On combining Propositions \ref{ineq} and \ref{metineq}, we obtain:

\begin{corollary}
\label{metcor}
With notation as above, for any $a \in \ciA$ we have 
\[
L_\ell(a) \leq L^{D_o}(a)  .
\]
\end{corollary}

By means of this key corollary we will in Section \ref{qgh} 
be be able to apply the important bounds on $ L_ \ell $ obtained in \cite{R7, R29} 
to prove that matrix algebras converge to coadjoint orbits for the
C*-metrics corresponding to Dirac operators. 


\section{More about spinors}
\label{secsp}

It is best if as our spinor bundle we can 
use a representation of the Clifford algebra on $\cS$ that is irreducible. 
For some compact 
Riemannian manifolds this can not be done. In this section we 
collect further algebraic facts and establish the
conventions and notation that we need in order to understand when this
can be done. Proofs of the 
assertions made below for which no proof is given here
can be found in \cite{BGV, Frd}. Also useful are \cite{GVF, LwM}, but
they use slightly different conventions.

Let $\fm$ be a finite dimensional Hilbert space over $\bR$
of dimension at least 3.
Let $\bC\ell(\fm)$ denote the complex 
Clifford algebra for $\fm$, 
much as discussed in Section \ref{sec1} for $\fg'$. 
Let $\Spin(\fm)$ be the subgroup of the group of invertible 
elements of $\bC\ell(\fm)$ generated by products of two 
elements of $\fm$ of length 1. Conjugation of $\bC\ell(\fm)$ by 
elements of $\Spin(\fm)$ carries $\fm$ into itself, and this gives a 
group homomorphism of $\Spin(\fm)$ onto $\SO(\fm)$
whose kernel is $\{1, -1\}$. In this way $\Spin(\fm)$ is the 
simply-connected covering group of $\SO(\fm)$, and these two 
groups have naturally isomorphic Lie algebras. The $\bR$-linear span 
of products of two orthogonal elements of $\fm$ is a Lie 
$\bR$-sub-algebra, $\spin(\fm)$, of $\bC\ell(\fm)$ with its additive 
commutator as Lie bracket. Exponentiation in $\bC\ell(\fm)$
carries $\spin(\fm)$ onto $\Spin(\fm)$, and one finds in this way 
that $\spin(\fm)$  is the Lie algebra of $\Spin(\fm)$.
It follows, in particular, that $\spin(\fm) \cong \so(\fm)$ naturally.

Suppose now that $\cS$ is a finite-dimensional Hilbert space over $\bC$, 
and that $\k$
is a $*$-representation of the C*-algebra $\bC\ell(\fm)$ on
$\cS$. Then the restrictions of $\k$ to $\Spin(\fm)$ and $\spin(\fm)$
give a unitary representation of that group, and a corresponding 
representation of its Lie algebra, on $\cS$. Let $\b$ 
denote the action of $\Spin(\fm)$ on $\bC\ell(\fm)$ by conjugation. 
Then the representation $\k$ of 
$\Spin(\fm)$ on $\cS$ is manifestly compatible with the the 
action $\b$ and the action $\k$
of $\bC\ell(\fm)$ on $\cS$ in the sense much as given above in 
equation \eqref{compt}, that is,
\[
\k_{\b_y(q)} = \k_y \k_q \k^*_y
\]
for $q \in \bC\ell(\fm)$ and $y \in \Spin(\fm)$.

Suppose now that $G$ is a connected Lie group with Lie 
algebra $\fg$, and that $\pi$ is a representation of $G$ on $\fm$
by orthogonal transformations. 
That is, $\pi $ is a homomorphism from $G$ into $\SO(\fm)$.
Let $\pi$ also denote the corresponding homomorphism from $\fg$
into $\so(\fm)$. From our natural identification of $\so(\fm)$ with $\spin(\fm)$ 
we can view $\pi$ as a homomorphism from $\fg$ to  $\spin(\fm)$. 
(A formula for that homomorphism can be obtained by applying 
formula 5.12 of \cite{GVF} or
formula 3.4 of \cite{BGV}.)
From a fundamental theorem for Lie groups (theorem 5.6 of \cite{Hll}) 
it follows that 
there is a homomorphism from the simply connected covering 
group, $\hat G$, of $G$ into $\Spin(\fm)$ whose corresponding Lie 
algebra homomorphism is $\pi$. Let $\s $ denote the composition 
of this homomorphism 
with the action $\k$ of $\Spin(\fm)$ on $\cS$, so that $\s$ 
is a unitary representation of $\hat G$ on $\cS$.
When we combine this with the earlier observations, we obtain:

\begin{proposition}
\label{compat}
Let $G$ be a connected simply-connected 
compact Lie group, and let $\pi$
be a representation of $G$ by orthogonal transformations on 
a finite-dimensional Hilbert space  $\fm$ over $\bR$,
 and so by 
automorphisms of $ \bC\ell(\fm)$.
Let $\tilde \pi$ denote the corresponding homomorphism from 
$G$ into $\Spin(\fm) \subset \bC\ell(\fm)$. Let $\k$ be a 
$*$-representation of $\bC\ell(\fm)$ on a finite-dimensional 
Hilbert space $\cS$, and let 
$\s$ be the composition of $\tilde \pi$ with $\k$, so that $\s$ is 
a unitary representation of $G$ on $\cS$.
Notice that the action $\pi$ of $G$ on $\bC\ell(\fm)$ is obtained by
composing $\tilde \pi$ with the action of $\Spin(\fm)$ on $\bC\ell(\fm)$
by conjugation. Then $\s$ satisfies the compatibility condition
\[
\k_{\pi_x(q)} = \s_x \k_q \s^*_x
\]
for $q \in \bC\ell(\fm)$ and $x \in G$.
\end{proposition}

\begin{example}
\label{ex1}
Let $G= \SO(3)$ and let $\pi$ be its standard representation on
$\fm = \bR^3$ with its standard inner product. Then $\bC\ell(\fm)$ 
is of dimension 8, and is isomorphic to the direct sum of 2 copies
of $M_2(\bC)$, so its irreducible representations are of dimension 2.
But $G$ does not have any irreducible representation of dimension 2
so it can not act on 
the spinor spaces for $\bC\ell(\fm)$.
However, the simply connected covering group of $G$ is $SU(2)$,
and it has irreducible representations of dimension 2. It will act on 
the spinor spaces for $\bC\ell(\fm)$, compatibly with its action on
 $\bC\ell(\fm)$ via $G$.
\end{example}


\section{Dirac operators for matrix algebras}
\label{sec3}

In this section we consider the action $\a$ of $G$ 
on $\cA = \cB(\cH)$ where $\cH$ is the Hilbert space of an irrep 
$U$ of $G$, and the action $\a$ on $\cB(\cH)$ 
is by conjugation by this representation. Since $\cA$ is finite 
dimensional, we have $\ciA = \cA$. For any $ X\in\fg $ and 
$ T\in\cA$ we have $\a_ X(T) =[U_ X, T] $.

Since the center of $G$ will act 
trivially on $\cA$, we can factor by the connected component of the 
center. The resulting group will be semisimple. Thus for the remainder 
of this section we will assume that $G$ is a compact connected 
semisimple Lie group.

In general the representation $(\cH, U)$ need not be faithful. Its kernel
at the Lie-algebra level is an ideal of $\fg$. But $\fg$, as a semisimple
Lie algebra, is the direct sum of its minimal ideals, each of which
is a simple Lie algebra. 
Denote the Lie-algebra-kernel of $U$ by $\fg_0$. It must be the direct sum of
some of these minimal ideals. Denote the direct sum of the remaining
minimal ideals by $\fg_U$, so that $\fg = \fg_U \oplus \fg_0$. Clearly
$U$ is faithful on $\fg_U$ and trivial on $\fg_ 0 $. 
We identify $\fg_U'$ with the subspace
of $\fg'$ consisting of linear functionals on $\fg$ that take value 0
on $\fg_0$, and similarly for $\fg_0'$. Clearly  $\fg_U$ and $\fg_ 0 $ are
$\Ad $-invariant, and so $\fg_U'$ and $\fg_ 0' $ are
$\Cad $-invariant.

For any $ T\in \cA $ it is clear that $\rd T(X)=\a_ X(T) = 0 $  for any $X$ in $\fg_0$,
so that $\rd T\in \cA \otimes\fg_ U' $. Thus the range of $\rd $ 
is in $ \cA \otimes\fg_ U' $. In \cite{R33} it is shown that
for the quotient of $G$ by the kernel of $U$, the sub-bimodule
of its $\O_o$ that is generated by the range of $\rd$ is 
exactly that $\O_o$ itself. But $\fg_ U $ is exactly
the Lie algebra of the quotient of $G$ by the kernel of $U$. We see in this 
way that the cotangent bundle, $\O(\cA) $, of $\cA $ is exactly
\[
\O(\cA)  =  \cA \otimes\fg_ U' . 
\] 
Notice that for the original $\O_o $ we have 
$\O_o = \cA  \otimes\fg' = \O(\cA) \oplus (\cA  \otimes\fg_ 0')  $.

We assume that, as before, we have chosen a $\Cad $-invariant inner 
product on $\fg'$. It restricts to a $\Cad $-invariant inner 
product on $\fg_ U'$. 
We then form the complex Clifford algebra $\bC\ell(\fg_ U')$. We note
that it is a  $\Cad $-invariant unital C*-subalgebra of $\bC\ell(\fg')$.
Since $\O(\cA)  =  \cA \otimes\fg_ U'  $, the Clifford
bundle for $\cA$ is exactly
\[
\bC\ell(\cA) = \cA \otimes \bC\ell(\fg_ U')   .
\]

Let $(\cS, \k) $ be a choice of an irreducible $* $-representation
of $ \bC\ell(\fg_ U') $. 
We then form the spinor bundle $\cA \otimes \cS$. 
It is a Hilbert space with its inner product coming from using 
the (unique) tracial state on $\cA$. We then form the corresponding
Dirac operator, $ D $, on  $\cA \otimes \cS$ much as in Section \ref{sec1},
so that
\[
D(T) = (I^\cA \otimes \k) \circ (\rd\otimes  I^\cS)(T)  
= \sum \a_{E_j}(T) \otimes \k_{\e_j}   
\]
for $T \in \cA$, where now $\{E_j\} $ is a basis for $\fg_ U $.

\begin{definition}
\label{matdir}
With notation as above,  the operator $D $ defined just above is the Dirac operator
for $\cA = \cB(\cH) $ for the given inner product on $\fg' $. (Thus if the representation
$(\cH, U)$ is faithful on $\fg$ then $D$ coincides with $D_o$.)
\end{definition}

Let $\hat G$ be the simply-connected covering group of $G$.
Then Proposition \ref{compat} tells us that there is a representation $\s$
of $\hat G$ on $\cS$ that is compatible with 
the action of $\hat G$ (through $G$)
on $\bC\ell(\fg'_U)$ in the 
sense that  
\[
\k_{\Cad_x(q)} = \s_x \k_q \s^*_x
\]
for all $q \in \bC\ell(\fg'_U)$ and $x \in \hat G$. We then define a 
representation, $\tilde \s$, of $\hat G$ on $\ciA \otimes \cS$
by $\tilde \s = \a \otimes  \s$. 
It follows from a very slight modification of Proposition \ref{commut1} 
to account for our using $\fg'_U$ instead of $\fg'$,
that $D$ commutes with 
this action. Then, a very slight modification of Proposition 
\ref{invar} tells us that
the corresponding  C*-metric $ L^D $ is invariant under the action 
$\a$ of $\hat G$ in the sense that for
any $a \in \cA$ we have
\[
L^D(\a_x(a)) = L^D(a)
\]
for all $x \in \hat G$.

According to Proposition \ref{prodnorm}, for any $T\in\cA $ 
\[
L^{D_o}(T) =
\|[D_o, M_a]\| =  \|\sum \a_{E_j}(T) \otimes {\e_j}\|
\]
where the norm on the right is that on $\cA \otimes \bC\ell(\fg')$.
If we choose a basis, $\{ E_j: j \in {I_U}\}$, for $\fg_U $ and adjoin to it a basis for
$\fg_ 0 $, then from the fact that $\a_X(a) = 0 $ for any $ X\in\fg_ 0 $ we see 
that 
\[
L^{D_o}(T) = \|\sum_{ j \in I_U} \a_{E_j}(T) \otimes {\e_j}\|  .
\]
But $\sum_{ j \in I_U} \a_{E_j}(a) \otimes {\e_j} $ is in 
$\cA\otimes \bC\ell(\fg_ U') = \bC\ell(\cA) $. Then regardless of the choice of
spinors for $\bC\ell(\fg_ U') $, it follows from  Proposition \ref{prodnorm}
that for the corresponding Dirac operator $D $ we have
\[
L^D(T) =  \|\sum_{ j \in I_U} \a_{E_j}(T) \otimes {\e_j}\|   .
\]
Since the norm on $ \bC\ell(\fg_ U') $ is just the the restriction to 
$ \bC\ell(\fg_ U') $ of the norm on $ \bC\ell(\fg') $, the norm on the right
can be taken to be the norm of $\cA \otimes \bC\ell(\fg')$.
We thus obtain the following proposition, which will be important 
for the proof of our main theorem.

\begin{proposition}
\label{deq}
For any compact connected Lie group $G $ and any irrep
$(\cH, U) $ and corresponding action $\a $ of $G$ on $\cA =\cB(\cH) $, 
and for notation as above, we have
\[
 L^D =  L^{D_o}   .
\]
\end{proposition}


\section{The fuzzy sphere}
\label{fuzzy}

We now illustrate our general results obtained so far by working them 
out explicitly for the case that has received the most attention in 
the quantum-physics literature, namely the case of $ G = SU(2) $ 
and its irreps. This will permit us to compare our 
Dirac operator with those proposed in the physics literature. 

We now establish our conventions and notation. As a basis for 
$\fg =su(2) $ we take the product by $i $ of the Pauli matrices \cite{Tlm}, 
so
\[
E_1 = \begin{pmatrix}
0 & i \\
i & 0
\end{pmatrix}
\quad \quad
E_2 = \begin{pmatrix}
0 & -1 \\
1 & 0
\end{pmatrix}
\quad  \quad
E_3 = \begin{pmatrix}
i & 0 \\
0 & -i
\end{pmatrix}  .
\]
Then $ E_1 E_2 = E_3 $, and cyclic permutations of the indices. 
Consequently $[E_ 1, E_ 2] = 2 E_ 3 $ and cyclic permutations 
of this, which are Lie algebra relations. Furthermore $(E_j)^2 =-I_ 2 $ 
for $j = 1, 2, 3 $. We see that $\mathrm{trace}((E_j)^2) =-2 $. 
This leads us to define the $Ad$-invariant inner product on
$\fg $ by $\< X, Y\> =-(1/ 2)\mathrm{trace}(XY) $, so that the 
$ E_j $'s form an orthonormal basis for $\fg $.
As Cartan subalgebra for $\fg^\bC =sl(2 , \bC)$ we choose 
the $\bC$-span of $ E_ 3 $, 

Let $(\cH^n, U^n) $ be an irrep of $ G $, 
where the label $n $ is the highest weight of the representation.
Thus $\dim(\cH^n) =n + 1 $. We let $ U^n $ also denote the 
corresponding ``infinitesimal" representation of $\fg^\bC $ on $\cH^n $.
As done in the introduction, 
we set $\cB^n =\cB(\cH^n) $, and we let $\a $
be the action of $ G $ on $\cB^n $ by conjugation by $ U^n $.
Its infinitesimal version is given by $\a_ X(T) = [U^n_ X, T] $ 
for all $ T\in \cB^n $ and $ X\in \fg^\bC $.

According to Theorem 4.1 of \cite{R33}, the cotangent bundle for $\cB^n $ is 
$\O_n =\cB^n \otimes_\bR \fg' = \cB^n \otimes_\bC {\fg'^\bC}$.
We let $\{\e_j\}_{j=1}^3 $ be the basis for $\fg'$ that is dual to the
basis $\{ E_j\} $ for $ \fg $ chosen above. 
It will be orthonormal for the inner product 
on $\fg'$ that is dual to our chosen inner product on $\fg $. Any other 
$\Cad $-invariant inner product on $\fg'$ will be a scalar multiple
of this one. If needed, that scalar can be pulled along through the 
calculations we do below.

The Clifford algebra $\bC\ell(\fg')$ is generated by $\{\e_j\}_{j=1}^3 $
with the relations
\[
\e_j\e_k + \e_k\e_j = -2\<\e_j, \e_k\> = -2\d_{jk}  .
\]
Since $\fg $ is 3-dimensional,  $\bC\ell(\fg')$ is isomorphic to 
$M_2(\bC) \oplus M_2(\bC) $, while $\bC\ell^e(\fg')$ is isomorphic to 
$M_2(\bC)$ . Let $\k $ be an irreducible representation of
$\bC\ell^e(\fg')$ on a Hilbert space $\cS$ (necessarily of dimension 2, 
and unique up to unitary equivalence).
The chirality element for $\bC (\fg')$ is $\bg =\e_1\e_2\e_3 $. We extend
$ \k $ to $\bC (\fg')$ by setting either $\k_\bg = I^\cS $ or $\k_\bg = -I^\cS $
(as discussed in Section \ref{fspin})
to obtain the two inequivalent irreducible representations of $\bC (\fg')$.
The Clifford bundle for $\cB^n $ is $\bC\ell(\cB^n ) = \cB^n\otimes \bC\ell(\fg') $, 
and the spinor bundle is $\cS(\cB^n ) = \cB^n \otimes \cS $.
The Dirac operator on  $\cS(\cB^n ) $ is defined on elementary tensors by
\[
D(T\otimes \psi ) = \sum \a_{E_j}(T) \otimes \k_{\e_j}(\psi) = \sum [U^n_{E_j}, T] \otimes  \k_{\e_j}(\psi)
\]
for $ T\in \cB^n $ and $\psi \in \cS $,
that is,
\[
D =  \sum_j \a_{E_j} \otimes  \k_{\e_j} .
\]

Since $ SU(2) $ is simply connected, Proposition \ref{compat} tells us that
corresponding to its action $\Cad$ on $\fg'$ 
there is a homomorphism of it into the subgroup $\Spin(\fg')$ 
of the group of invertible elements
of the even subalgebra $ \bC\ell^e(\fg') $ of $ \bC\ell(\fg') $. 
This homomorphism, which we will denote by $\widehat \Cad $, 
plays the role of the $\tilde\pi $ of Proposition \ref{compat}. It
is easy to describe, and we will need a precise description of it below. 
The $E_j $'s happen to also be elements of $SU(2) $.
To try to avoid confusion, we denote them by $x_j = E_j $ when we view 
them as elements of $ SU(2) $. 
We note that $(x_j)^{-1} = -x_j$. It is easily checked that
$\Ad_{x_j}(E_k) = E_k$ if $k=j$ and $= -E_k$ if $k \neq j$. 
Consequently $\Cad_{x_j}(\e_k) = \e_k$ if $k=j$ and $= -\e_k$ if $k \neq j$.
Set $\hat\e_1 = \e_2\e_3$ and cyclic permutations of the indices. 
The $\hat \e_j $'s are elements of $\Spin(\fg')$.
It is easily checked that conjugation by $\hat \e_j $ takes
$\e_k$  to $ \e_k$ if $k=j$ and to $= -\e_k$ if $k \neq j$.
Consequently $\widehat\Cad_{x_j} = \hat \e_j$ for each $j$.

Furthermore, if we set $x_0 = I_ 2 $, then every element of $ SU(2) $
can be expressed (see proposition VII.5.5 of \cite{Smn}) uniquely as 
$\sum_{j=0}^3 r_j x_j $ where the vector $r =(r_0,r_1,r_2,r_3) $ in $\bR^4 $ satisfies
$\|r\| =1 $.
(So here we view $ SU(2) $ as the unit sphere in the quaternions.)
Set $\hat \e_ 0 = 1_{ \bC\ell} $ where $ 1_{ \bC\ell}$ is the identity 
element of $ \bC\ell(\fg') $.
Thus we obtain the first statement of:

\begin{proposition}
\label{wcad}
With notation as above, $\widehat\Cad $ is the mapping that takes 
$\sum_{j=0}^3 r_j x_j $ to $\sum_{j=0}^3 r_j \hat \e_j $ for $\|r\| =1 $. 
It is an isomorphism from $ SU(2) $ onto $\mathrm{Spin}(\fg')$.
The corresponding isomorphism, $ \widehat\cad $, from $su(2) $ 
to $\spin(\fg') $ is the linear map that sends $ E_j $ to $\hat \e_j $
for $j =1, 2, 3 $ .
\end{proposition}

\begin{proof}
The proof of the statement concerning $\widehat \cad $ is easily 
obtained by examining, for each $j $, the derivative at 0 of the curve
$t \mapsto \sin(t) x_j + \cos(t) x_0 $ in $ SU(2) $ and of the image
of this curve under $\widehat \Cad $. 
\end{proof}

Let $\s $ be the representation of $ SU(2) $ on $\cS $ 
obtained by composing $\widehat\Cad $ with $\k $,
much as done in Proposition \ref{compat}. Then $\s $ and $\k $ manifestly 
satisfy the compatibility condition \ref{compt}. 
We let $\tilde \s$ be the action of $ G $ on the spinor bundle 
$\cA \otimes \cS $ defined by $\tilde \s =\a\otimes \s $.
It follows from Proposition \ref{commut1} that $ D $ commutes with this action.

But from Proposition \ref{wcad} we see that, in turn, we can
let $\cS =\bC^2 $ and let $\k $ be the inverse of the
 isomorphism $ \widehat\cad $, so sending  $\hat \e_j $
 to $ E_j $, and then extending this to $ \bC\ell^e(\fg') $ by
 sending $1_{ \bC\ell}$ to $I_2$. This then also gives the isomorphism
 from $\mathrm{Spin}(\fg')$ onto $ SU(2) $.
 We can then extend this $\k$ to $ \bC\ell(\fg') $ by sending
 $\bg $ to $\pm I_2$.

The charge conjugation operator, $C^\cS$, on $\bC^2 $
is then defined by  $C^\cS(v) = \s_2 \bar v$, where $\bar v$
is the standard complex conjugation on $\bC^2$ applied
to $v \in \bC^2$, and $\s_2$ is the standard Pauli spin matrix.
See the proof of Proposition 3.5 of \cite{GVF}, or for
the general setting see section 2.3 of \cite{DbD}.
The charge-conjugation operator on 
$\cA \otimes \cS$ is then defined exactly as in Section \ref{charge}.

We now compute the spectrum of $ D $, by relating it to the Casimir
element, $ C $, somewhat as done in \cite{DLV}. The following 
relation appears to be related to equation 99 of \cite{Brt}.

\begin{theorem}
\label{dcas}
For notation as above, 
\[
D = \pm (1/2) (\tilde \s_C - \a_C \otimes I^\cS - I^\cA \otimes \s_ C),
\]
where the sign depends on the choice of the spinor 
representation on $\cS = \bC^2 $.
\end{theorem}

\begin{proof}
We examine, 
for our case of $ G = SU(2) $, the 
term $\sum_j (\a_{E_j} \otimes \s_{E_j}) $ that appears in the formula for
$\tilde \s_C$ in Proposition \ref{bigcas}. From our discussion just above, 
$\s_{E_j} =\k(\widehat\cad_{E_j}) $,
where $\k$ also denotes here
the corresponding homomorphism of Lie algebras. 
Then from Proposition \ref{wcad} we see that $\s_{E_j} =\k_{\hat \e_j} $.
But it is easily checked that $\bg \e_j =-\hat \e_j $, and so
if we choose the spinor representation for which $\k(\bg) =-I_2 $,
we find that $ \k_{\e_j}  =\k(\hat \e_j) =\s_{E_j}$. Consequently
\[
\sum_j \a_{E_j} \otimes \s_{E_j} = \sum_j \a_{E_j} \otimes   \k_{\e_j} = D   .
\]
On the other hand, 
if we choose the spinor representation for which $\k(\bg) =+I_2 $,
then we will obtain
\[
\sum_j (\a_{E_j} \otimes \s_{E_j}) =-D.
\]
When this is combined with  the formula for
$\tilde \s_C$ in Proposition \ref{bigcas}, 
we obtain the formula in the statement of the theorem.
\end{proof}

To use the formula of the above theorem in order to 
determine the spectrum of $D $, we need to recall the 
well-known facts about the action of the Casimir
element $C $ on irreducible representations of $ SU(2) $.
We define the following elements of $\fg^\bC $ . 
\[
H =-i E_3, \quad E =-(E_2 +i E_1)/ 2  , \quad  F =(E_2 -i E_1)/ 2  .
\]
They satisfy the familiar relations
\[
[H, E] = 2 E, \quad [H, F] =-2F, \quad [E, F] = H.
\]
Then in $\cU(\fg^\bC) $ we find (as in equation 1.3.8 of \cite{HwT})
\[
C = \sum_{j = 1}^3 E_j^2 =-(H^2 + 2(EF + F E)).
\]
For our purposes it is useful to change this using
$ EF -FE = H $ to obtain
\[
C = -(H^ 2 + 2 H + 4 F E).
\]
This permits us to easily determine, as follows,
the scaler multiple of the identity operator to which $C $ is taken 
by any irrep of $ SU(2) $.
Let $(\cH^n, U^n) $ be the irrep with 
highest weight $n $.
This means that $\cH^n $ contains an eigenvector, 
$\xi_n$, for $ H $ of eigenvalue $n $ such that $ E\xi_n = 0 $.
Consequently, $ U_C^n \xi_n =-(H^2+2H)\xi_n  = -n(n + 2)\xi_n $. Thus:
\begin{proposition}
\label{ceig}
For the irrep $(\cH^n, U^n) $ of
$ SU(2) $ of highest weight $n$ we have $ U_ C^n =-n(n + 2) I^{\cH^n} $.
\end{proposition}
 
We now determine the spectrum of $ D $.
We choose the spinor representation for which $\k(\bg) =-I_2 $.
The description of the spectrum of $D$ is most attractive if 
we state it in terms of $D' = D + 2$.

\begin{theorem}
The spectrum of $D'$ consists of the even numbers 
$\pm 2k $ for  $1 \leq k \leq n$ together with $2(n+1)$. For each $k$ with $1 \leq k \leq n$ 
the multiplicity of the eigenvalue $\pm 2k$ is $2k =\dim(\cH^{2k-1})$, while the  multiplicity of 
$2(n+1)$ is $2(n+1) = \dim(\cH^{2n+1})$. The spectrum of $D$ is just the spectrum of
$D' $ shifted by $-2$.
\end{theorem}

\begin{proof}
It suffices to determine the spectrum of 
\[
 \tilde \s_C - \a_C \otimes I^\cS - I^\cA \otimes \s_ C,
\]
which is an operator on $\cB^n \otimes \cS $. This operator involves 
three different representations of $SU(2) $. Now for any compact group
$ G $ and unitary representation $(\cH, U) $ the representation $\a $ on
$\cB(\cH) $ by conjugation by $ U $ is unitarily equivalent to the representation
$ U\otimes \overline{ U }$ on $ \cH\otimes \overline{ \cH }$, where
$(  \overline{ \cH }, \overline{ U })$ is the representation dual to
$(\cH, U) $. But for $ G = SU(2) $ all representations are self-dual \cite{Smn}.
Thus for an irrep $(\cH^n, U^n) $ of
$ SU(2) $ the corresponding representation $\a $ on $\cB^n $ is equivalent to 
the representation $ U^n \otimes U^n $ on $ \cH^n \otimes \cH^n $.
But this representation is known (see section 8-3 of \cite{Tlm}) to decompose in such a way that
$(\a, \cB^n) $ is equivalent to
$\bigoplus_{k =0}^n (\cH^{2k} , U^{2k}).$ Thus
$\cB^n \otimes \cS $  is equivalent to
$\bigoplus_{k =0}^n (\cH^{2k} \otimes \cH^1)$ with corresponding 
representations $\a  \otimes I^\cS $,  $I^\cA \otimes \s $, and $\tilde \s =\a\otimes \s $. 
For each $k $ set $ W_k = \cH^{2k} \otimes \cH^1 $.

Since $U^1_ C = -3 I^{\cH^1}$, we see that 

\quad $I^\cA \otimes \s_ C $ on $ W_k $ is $ -3 I^{W_k} $.

We also see that

\quad $\a_ C \otimes I^\cS $ on $ W_k $ is $ -2k(2k+2) I^{W_k} 
= -4(k^2 + k) I^{W_k}$.

Furthermore, $ W_k = \cH^{2k} \otimes \cH^1$ decomposes for $\tilde \s$ 
as $\cH^{2k+1} \oplus \cH^{2k-1} $ for $k\geq 1 $, while $ W_0 = \cH^1 $. 
Thus we see that for $k\geq 1 $ we have

\quad $\tilde \s _ C $ on $\cH^{2k+1} \subset W_k $ is $-(2k+1) (2k+3)I^{\cH^{2k+1}}$, 
while

\quad $\tilde \s _ C $ on $\cH^{2k-1} \subset W_k $ is $-(2k-1) (2k+1)I^{\cH^{2k-1}}$,
while

\quad $\tilde \s _ C $ on $ W_0 = \cH^1$ is $-3I^{\cH^1}$   .

It follows that for $k\geq 1 $ we have

\quad $\pm 2D $ on $\cH^{2k+1} \subset W_k $ is 

\quad $-((2k+1) (2k+3) - 4(k^2 + k) - 3)I^{\cH^{2k+1}} = -4k I^{\cH^{2k+1}} $,
while

\quad $\pm 2D $ on $\cH^{2k-1} \subset W_k $ is 

\quad $-((2k-1) (2k+1) - 4(k^2 + k) - 3)I^{\cH^{2k-1}} = +4(k+1) I^{\cH^{2k-1}} $,
while

\quad $\pm 2D $ on $ W_0 = \cH^1$ is $((-3) - (0) - (-3))I^{\cH^1} = 0I^{\cH^1}$.

We can conveniently assemble all of this in the following way.
Let us keep the minus sign, and divide by 2. We obtain  for $k\geq 1 $

\quad $ D $ on $\cH^{2k+1} \subset W_k $ is $+2k I^{\cH^{2k+1}} $  ,

while

\quad $D $ on $\cH^{2k-1} \subset W_k $ is $-2(k+1) I^{\cH^{2k-1}} $  ,

while

\quad $D $ on $ W_0 = \cH^1$ is $ 0I^{\cH^1}$.

\noindent
Then let $D' = D + 2$. For each $k$ with $1 \leq k \leq n$ we see that on the 
copy of 

\quad $\cH^{2k-1}$  in $W_{k-1}$ we have  $D' = 2k I^{\cH^{2k-1}}$ , 

while on the copy of 

\quad $\cH^{2k-1}$ in $ W_k$
we have $ D'=-2k I^{\cH^{2k-1}}$.

Finally, for 

\quad $\cH^{2n+1}$  in $W_n$ we have  $D' = 2(n +1) I^{\cH^{2n+1}}$ .
\end{proof}

The asymmetry of the spectrum of $D$ and $D'$ precludes their
having a grading (chirality) operator, so they do not give ``even" spectral
triples. 

Let $D_{DLV}$ denote the Dirac operator obtained in 
proposition 3.5 of \cite{DLV}. Then part (iii) of that proposition
makes clear that our $D'$ is equal to $2D_{DLV}$.
Thus we see that our $D$ corresponds to $D_{DLV}$ 
once differences in our conventions are taken into account.
In \cite{DLV} the formula used to define $D_{DLV}$
is presumably guided in part by the formulas used in \cite{Kst},
but, as discussed in \cite{DLV}, their formula is closely related
to the formula obtained in \cite{GK2}.

In section VI of \cite{Brt} Barrett discusses various versions
of fuzzy spheres, beginning with a Dirac operator that is essentially
our Dirac operator $D'$ above. He points out that not only
is there no chirality operator, 
but this Dirac operator also does not have
the desired KO-dimension (defined there). He shows that
both of these problems are solved by forming the direct sum
of $D $ with its negative, so that the space of spinors is
$\bC^4 $ instead of $\bC^2 $. He also proposes a corresponding modified 
Dirac operator for the 2-sphere itself and makes comments about why this 
may be a reasonable thing to do. It would be interesting
to explore whether these ideas can be usefully extended
to all coadjoint orbits to which matrix algebras converge.

In \cite{DLV} there is a substantial discussion of how their Dirac 
operator compares to others in the literature. That discussion
applies equally well to the Dirac operator we have constructed.
We will not
repeat here most of that discussion, but we now summarize
some of it. 

The Dirac operator constructed in section IIIB of \cite{BlIm}
by using the Ginsparg-Wilson relations,
which also appears in section 8.3.2 of \cite{BKV}, and is
the spin-1/2 case in \cite{BlPd, DiPa}, is essentially the same
as the Dirac operator that we have constructed above.
However, the ``chirality operator" defined there is not
a true grading.

For the Dirac operator constructed in \cite{CWW97} and \cite{CWW98} 
and used in \cite{CWW00} there 
is a true chirality operator, but the Dirac operator is of ``second-order"
and has spectrum very different for the spectrum we described above.
In section 3.7.2 of \cite{Yd1} there is a detailed discussion of the relations
between the Dirac operator in \cite{CWW98} and that in \cite{GK2}.

The Dirac operators on matrix algebras defined in \cite{GKP} 
are not mentioned in \cite{DLV}.
They are constructed by a quite different method, which gives a quite different
formula for the Dirac operators. 
A useful discussion of the role of supersymmetry in the construction is
given in the introduction to \cite{Klm}.
A good discussion of the contrast between 
that method and the methods used in the papers mentioned in the
paragraphs above (and so in the 
present paper) is given in the introduction to \cite{DHMO7}, which lays the 
foundation for applying that method to complex projective spaces in
\cite{DHMO}. Section 2 of \cite{HQT} is devoted to an exposition
of the construction in preparation for extending it to the q-deformed sphere
(continued in \cite{HQT2, Hrk, Lft, Lft2}). 
 In \cite{BMS} the same method is
used to treat the case in which $G = SU(2)\times SU(2)$
acting on $S^2\times S^2$, a 4-dimensional space, so of
special interest to physics. 
In the papers using that method there is no mention of 
charge conjugation.
Thus it would be interesting to know whether the 
finite-dimensional
spectral triples constructed by that method can be equipped with a
``real'' structure. It would also be interesting to know whether the method
can be generalized to apply to all coadjoint orbits of integral weights
of all compact semisimple Lie groups, and whether it even has a primitive 
version for ergodic actions of the kind discussed in the early sections of the 
present paper. It is not clear to me how to make a
precise comparison between the Dirac operators obtained by that method 
and the ones constructed in the present paper.


\section{Homogeneous spaces and the Dirac operator for G}
\label{homg}

In this and the next sections we reformulate the (usual) construction of
Dirac operators on coadjoint orbits in terms just of algebras and modules, 
without any mention of coordinate patches or charts, refining the approach in 
Section \ref{sec1}. This is essential for the proof of our main theorem. 
We follow fairly closely parts of \cite{R22, R27}, but here we emphasize
cotangent bundles rather than tangent bundles.

We begin working in the following generality.
We let $G $ be any compact connected Lie group. We let $ K $
be any closed subgroup of $G $, and we form the corresponding 
homogeneous space $ G/K$. Many of the results in this section are 
adaptations to our context of some of the results in \cite{R22}.
We let $\cA = C(G/K)$ and we let $\a $ be the action of $G $ on $\cA $
by left translation. Then $\ciA $, formed using $\a $, is also the smooth 
structure on $ G/K $ coming from that on $ G $. 

As in Section \ref{sec1}, we set $\O_o = \ciA \otimes \fg'$, we chose 
a spinor space $\cS$ for $\bC\ell(\fg')$ (which need not be irreducible),
and we define $D_o $ on $\ciA\otimes \cS$. Since now 
$\ciA = C^\infty(G/K)$, we can
express $D_o $ in terms of functions 
as follows. Let $f,g \in \ciA$ and $\psi \in \cS$. Then 
$[D_o, M_f](g \otimes \psi)$ is in $\ciA \otimes \cS $, so we can evaluate 
it at any $y \in G/K$, to obtain an element of
$\cS$, as follows:
\begin{align*}
[D_o, M_f](g \otimes \psi)(y) &= \sum\a_{E_j}(f)(y)g(y) \otimes \k_{\e_j}\psi   \\
&= g(y)\k_{\sum \a_{E_j}(f)(y)\e_j}\psi   .
\end{align*}
But $\rd f_y$ is the element $\sum \a_{E_j}(f)(y)\e_j $ of $\fg' \subset \bC\ell(\fg')$, so we 
see that $[D_o, M_f] $ is given by the operator-valued function
\[
[D_o, M_f]_y = 1_\cA \otimes \k_{\rd f_y}   .
\]
(Compare this with proposition 9.11 of \cite{GVF}.) Consequently 
$\|[D_o, M_f]\| = \|\k_{\rd f_y}\| $. But for any $\mu \in \fg'$ we have
$\|\k_{\mu}\|^2 = \|\<\mu, \mu\>1_\cA\| = \|\mu\|^2$, and
it is easy to check that this extends to $\mu \in \fg'^\bC$. Thus we obtain
the following result, which will be of importance for the proof 
of our main theorem.
\begin{proposition}
\label{bigb}
With notation as above we have
\[
L_{D_o}(f) = \|[D_o, M_f]\| = \sup\{\|\rd f_y\|: y \in G/K\}   ,
\]
where the norm comes from the inner product on $\fg'$.
\end{proposition}
We can put this in a more familiar form in terms of the dual inner product
on $\fg$. For all $y$ there will exist an element, $\mathrm{grad}_y f$,
of $\fg$ such that $\rd f_y(X) = \<X, \mathrm{grad}_y f\>$ for all $X\in \fg$.
Thus 
\[
L_{D_o}(f) =  \sup\{\| \mathrm{grad}_y f\|: y \in G/K\}.
\]

All of the above applies to the case in which $K =\{e_G\}$, that is,
to the algebra $\cA = C(G) $.
Because any Lie group is parallelizable, the cotangent bundle for $ G$ actually is
$\ciA \otimes \fg'$. So this case is also one to which all of the results 
of Section \ref{sec1} immediately apply. Once one applies 
Proposition \ref{compat} concerning 
the possibility of choosing	the space $\cS$ to be irreducible,
one obtains:

\begin{proposition}
Let $G$ be a connected simply-connected 
compact Lie group, and let a $\Cad$-invariant inner product be chosen 
for $\fg'$. Then an irreducible spinor space $\cS $ with 
compatible action of $ G$ can be found, and the 
operator $D_o $ on $\ciA \otimes \cS$ constructed as in Section \ref{sec1} 
is the correct Dirac
operator for $G $, invariant under the action of $ G$ on itself 
by left translation. We denote it by $D^G$.
\end{proposition}	

We return to the general situation with $\cA = C(G/K)$. 
Since for any $f \in \ciA$ and $ X\in \fg$ the expression $\a_X(f)$ gives 
the same function regardless of whether $f $ is viewed as an element
of $\ciA $ or of $C^\infty(G) $, 
it is clear from the formulas above that $D_o $
is simply the restriction of $ D^G $ to $\ciA \otimes \cS$.

A standard argument (see the text 
around proposition 9.12 of \cite{GVF})
shows that $L_{D^G}(f) $ is the Lipschitz constant 
for any $f \in  C^\infty(G) $ for the 
ordinary metric on $G$ coming from the Riemannian metric determined 
by our chosen inner product on $\fg'$. 


\section{Dirac operators for homogeneous spaces}
\label{dhom}

We continue to work here in the following generality.
We let $G $ be any compact connected Lie group. We let $ K $
be any closed subgroup of $G $, and we form the corresponding 
homogeneous space $ G/K$. Again, many of the results in this section are 
adaptations to our context of results in \cite{R22, R27}.

We let $\fk $ be the Lie algebra of $K $. Let $ f\in \ciA = C^\infty(G/K)$,
and let $ X\in\fg $. As before, we view $ f $ as a function on $ G$
when convenient.
Then for any $x\in G$ we have
\[
\rd f_x(X) = D^t_o f((\exp(-tX))x) = D^t_o f(x\exp(-t\Ad_x^{-1}(X)) ) .
\] 
Since $f(xs) = f(x)$ for all $x\in G$ and $ s \in K$, we see from this 
that $\rd f_{xs}(X) = df_x(X)$ for all $x\in G$ and $ s \in K$, and 
that if $\Ad_x^{-1}(X) \in \fk$ for some $x \in G$ then $df_x(X) = 0 $.
Let 
\[
\fm' =\{\mu \in \fg': \mu(X) = 0 \ \mathrm{for \ all} \ X \in \fk\} ,
\]
and let $\fmc$ be its complexification. 
Notice that $\fmc$ is carried into itself by the restriction of $\Cad $ to $ K $. 
If $ X\in\fk $
then $\Ad_x^{-1}(\Ad_x(X)) \in \fk $, so $\rd f_x(\Ad_x(X)) = 0 $.
This means that $\Cad_x^{-1}(df_x) \in \fm'^\bC$ for all $x\in G$.
Let 
\[
\O(G/K)  =  \{\om \in C^\infty(G/K, \fg'^\bC):  
\om(x) \in \Cad_x(\fmc) \ \mathrm{for \ all} \  x \in G \} .
\]
The calculations above show that $\rd f \in \O(G/K) $. 
For pointwise operations, $\O(G/K) $ is an $\ciA$-module.
It is easy to check that $\O(G/K) $ is carried into itself by the action $\g $ of $G $
on $\O_o(G)$.
By very minor adjustments to the discussion in section 6 of \cite{R33}  one sees that 
$\O(G/K) $ is indeed the complexified cotangent bundle for $ G/K $ (as is 
basically well-known \cite{KbN}).

However, the fact that the range spaces of the $\om$'s depend on $x$ complicates 
calculations. But by taking advantage of the fact that we are working 
on a homogeneous space, we can make a small change that makes the situation considerably 
more transparent
(much as done in connection with notation 4.2 of \cite{R22}).
For each $\om \in \O(G/K)$ define $\hat \om $ on 
$G$ by $\hat\om(x) =\Cad_x^{-1}(\om(x))$.
Then it is easily checked that $\hat \om $ is in the $\ciA$-module
\[
\hat\O(G/K) = \{\om \in C^\infty(G, \fmc): \om(xs) = 
\Cad^{-1}_s(\om(x)) \ \mathrm{for} \ x \in G, s \in K  \}.
\]
This is exactly the complexification of the space given in notation 6.1 
of \cite{R33}, where a proof is indicated for the well-known fact 
that it too can be viewed as the space of smooth cross sections of 
the cotangent bundle of $G/K $. 
(It is an ``induced representation'' \cite{R22}. There are many more 
induced representations in the next pages.)

We then define $\hat \rd$ 
by $\hat \rd f =  \widehat{\rd f}$ for all $f \in \ciA$. Then for $x \in G$
and $X \in \fg$ we have
\begin{align*}
\<X, \hat\rd f_x\> &= \<X, \Cad_x^{-1}(\rd f_x)\> = \<\Ad_x(X), \rd f_x\>    \\
&= D^t_o(f(\exp(-t \Ad_x(X)x)) = D^t_o(f(x\exp(-t X)) , 
\end{align*}
so that $\hat \rd$ is defined in terms of the \emph{right} action of $G $ on 
itself.
Clearly  $\hat \rd$ is again a derivation,
and the first-order differential calculi $( \hat\O(G/K), \hat\rd)$ and
$(\O(G/K), \rd)$ are easily seen to be isomorphic. Thus we can use the latter.
It is easily checked that 
$\widehat{\g_x\om} = \a_x(\hat \om) $, so that the left-translation action $\a$
of $G$ on itself is the appropriate one to use on $\hat\O(G/K) $.

We now assume, as usual, that a $\Cad $-invariant inner product on $\fg'$
has been chosen.
We can then give a simple description of the Clifford bundle for $\cA$,
much as done in section 7 of \cite{R22}.
We first let $\clfm$ be the complex Clifford algebra for $\fm'$ for 
the restriction to $\fm'$ of the given 
inner product on $\fg'$, much as in Section \ref{sec1}. Since the restriction to $K $ of the
action $\Cad$ on $\fm'$ is by
isometries, it extends (as Bogoliubov automorphisms) 
to an action on
$\clfm $, which we denote again by $\Cad $. We then set
\[
\clf G/K) = \{\cW \in C^\infty(G, \clfm): \cW(xs) = 
\Cad^{-1}_s(\cW(x)) \ \mathrm{for} \ x \in G, s \in K  \}.
\]
Then $\clf G/K)$ with pointwise product is easily 
seen to be a C*-algebra that effectively contains 
$\ciA$ in its center.
Since $\fm' \subset \clfm $ and $\fm'$ generates $\clfm $ as a unital algebra,
it is also easily seen that $ \hat\O(G/K) \subset \clf G/K) $ and that
 $ \hat\O(G/K) $ together with $\ciA $ generate $ \clf G/K) $ as a unital algebra. 
For our situation $ \clf G/K) $  is the appropriate Clifford-algebra
bundle.

Much as in Section \ref{secsp}, we choose 
a spinor space $\cS$ for $\bC\ell(\fm')$ (which need not be irreducible
or faithful).
We need to assume that there is a 
unitary representation, $\rho$, of $K$ on $\cS$ that is 
compatible with the restricted action $\Cad$ of $K$ on $ \bC\ell(\fm')$ and 
the representation $\k$ of  $ \bC\ell(\fm')$ on $\cS$, in the 
sense that  
\begin{equation}
\label{comeq}
\k_{\Cad_s(q)} = \rho_s \k_q \rho^{-1}_s
\end{equation}
for all $q \in \bC\ell(\fg')$ and $s \in K$. 
Since $K$ is never semisimple in the situation in which
we are most interested, namely that of coadjoint orbits, 
we can not in general 
apply the arguments used 
for Proposition \ref{compat} to obtain a spinor space that is irreducible.
This is a serious issue, reflecting the fact that not 
all homogeneous spaces are $\Spin^c$.
Much of the next section will be devoted to dealing with this issue
for coadjoint orbits.
But we can always choose $\cS =  \bC\ell(\fm') $ 
(with inner product from its canonical trace)
with $\rho = \Cad $
restricted to $K $. This leads to the Dirac-Hodge
operator, as discussed in section 8 of \cite{R22}.

Whatever way $\cS$ and $\rho $ are chosen, we then set
\[
\cS(G/K)  =  \{\Psi \in C^\infty(G, \cS):  
\Psi(xs)  = \rho_s^{-1}\Psi(x)  \ \mathrm{for \ all} \  s \in K \} .
\]

\begin{proposition}
\label{mulsp}
The evident pointwise 
product of elements of $\cS(G/K)$ by elements of 
 $ \clf G/K) $, using the representation $\k $ of $ \bC\ell(\fm')$ on $\cS$,
carries $\cS(G/K)$ into itself. 
\end{proposition}

\begin{proof}
This uses the
 compatibility relation \eqref{comeq}. Let $\cW \in   \clf G/K) $ and $\Psi \in \cS(G/K)$.
 Then for $x \in G $ and $s \in K $ we have
 \begin{align*}
 (\cW\Psi)(xs) &=\k_{\cW(xs)}\Psi(xs)  
 =  \k_{\Cad_s^{-1}(\cW(x))}(\rho_s^{-1}(\Psi(x))  \\
 &= \rho_s^{-1}(\k_{\cW(x)}\Psi(x)) = \rho_s^{-1}(\cW\Psi)(x) 
 \end{align*}
 as desired.
 \end{proof}
 
 There is an evident pointwise $\ciA$-valued inner product on $\cS(G/K)$,
 using the inner product on $\cS$ and the fact that $\rho$ is unitary.
 We denote it again by $\<\cdot, \cdot \>_\cA$.
 
 Our Dirac operator $ D $ will be an operator on $\cS(G/K) $. 
 Because of the way $\cS(G/K) $ is defined, 
 we cannot simply define $D $ on elementary tensors as 
 done in previous sections. We need to use connections, as done in \cite{R22, R27}.
 We follow somewhat closely the development in section 5 of \cite{R27}.

Let $\fm $ be the orthogonal complement of $ \fk $ for the inner product on
$\fg $ that is dual to our chosen inner product on $\fg'$.
Then $\fm'$ can be viewed as the dual vector space to $\fm $.
The tangent bundle, $\cT(G/K) $, of $G/K $ is given by
\[
\cT(G/K) = \{V \in C^{\infty}(G,\fm): V(xs) = \Ad_s^{-1}(V(x)) \text{ for } x \in G, s \in K\}.
\]
It is an $\ciA$-module for the pointwise product, and $G$ acts on it by translation.  
We denote this translation action again by $\a$.  Each $V \in \cT(G/K)$ determines a derivation, $\d_V$, of $\ciA$ by
\[
(\d_Vf)(x) = D_0^t(f(x \exp(tV(x))).
\]
We denote the complexifcation of  $\cT(G/K) $ by  $\cT^\bC(G/K) $.
For $ V \in  \cT^\bC(G/K) $ we define $\d_V$ in terms of the real 
and imaginary parts of $V $. 

There is an evident $\ciA$-linear pairing of  $\cT^\bC(G/K) $
with $\hat\O(G/K)$, coming from the pairing of $\fm $ with $\fm'$.
We denote it by $\<V, \om\>_\cA$.
Both $\cT^\bC(G/K) $ and $\hat\O(G/K)$ are finitely-generated projective
$\ciA$-modules (in accordance with Swan's theorem) because
they are induced modules. See proposition 2.2 of \cite{R22}. 
It is then easily verified that $\hat\O(G/K)$
is the $\ciA $-module dual of $\cT^\bC(G/K) $
via the pairing mentioned above.

We define a connection, $\nabla^c $, on  $\hat\O(G/K)$ by
\begin{equation}
\label{eq2.1}
(\nabla_V^c(\om))(x) = D_0^t(\om(x \exp(tV(x)))
\end{equation}
for $V \in \cT(G/K)$ and $\om\in \hat\O(G/K)$, extended to
$\cT^\bC(G/K) $ by linearity.  It is often called ``the canonical connection". 
(One can not expect that it is always a Levi-Civita connection
in the sense of definition 8.8 of \cite{GVF}). 

We then extend $\nabla^c $ to  $ \clf G/K) $, and
define a connection, $\nabla^\cS $, on
$\cS(G/K)$, by the same formula, setting
\begin{equation}
\label{eq2.2}
(\nabla_V^c(\cW))(x) = D_0^t(\cW(x \exp(tV(x)))),
\end{equation}
and
\begin{equation}
\label{eq2.3}
(\nabla_V^\cS(\Psi))(x) = D_0^t(\Psi(x \exp(tV(x)))),
\end{equation}
for $V \in \cT(G/K)$, $\cW\in \clf G/K)$, and $\Psi \in \cS(G/K)$, extended to
$\cT^\bC(G/K) $ by linearity. 
These are all evidently $\ciA $-linear in $V$.
We will need shortly the evident Leibniz rule
\begin{equation}
\label{eq6L}
\nabla_V^\cS(f\Psi) = \d_V(f)\Psi + f\nabla_V^\cS(\Psi).
\end{equation}
We can view $\k $ as a bilinear form on
 $\clf G/K) \times  \cS(G/K) $ with values in $\cS(G/K)$, and so we
 have the Leibniz rule
\begin{equation}
\label{eq6.4}
\nabla_V^c(\k_\cW\Psi) = \k_{\nabla_V^c\cW}\Psi + \k_\cW(\nabla_V^c\Psi)
\end{equation}
for any $V \in \cT(G/K)^\bC$, $\cW \in \bC\ell(G/K)$ and $\Psi \in \cS(G/K)$.  
One can check that the connection 
on $\cS(G/K)$ is compatible with the $\ciA$-valued inner product 
in the sense of the Leibniz rule   
\begin{equation*}
\d_V(\<\Psi,\Phi\>_\cA) = \<\nabla_V^c\Psi,\Phi\>_\cA +
 \<\Psi,\nabla_V^c\Phi\>_\cA
\end{equation*}
for any $V \in \cT^\bC(G/K)$ and $\Psi,\Phi \in \cS(G/K)$,
but we do not need this fact later. 
 
We can now define the Dirac operator, $D $, on $ \cS(G/K)$ as follows.  
Let  $\Psi \in \cS(G/K)$. Define $\rd \Psi$ on $\cT^\bC(G/K) $ by
$(\rd \Psi)(V) = \nabla^\cS_V\Psi$. Then $\rd \Psi $ is an $\ciA$-linear
map from $\cT^\bC(G/K) $ into $\cS(G/K)$, and so $\rd \Psi$
can be viewed as an element of $\hat\O(G/K) \otimes  \cS(G/K) $,
since  $\hat\O(G/K)$ is the $\ciA $-dual of  $\cT^\bC(G/K) $.
By means of the inclusion of $\hat\O(G/K)$ into $\bC\ell(G/K)$
we can view $\rd \Psi$ as an element of 
$\bC\ell(G/K) \otimes  \cS(G/K)$. We can then apply $\k $ to obtain 
an element of $\cS(G/K) $. That is, with the above understanding,
we define $D$ on $\cS(G/K) $ by
\[
D(\Psi) = \k(\rd \Psi)
\]
for $\Psi \in  \cS(G/K) $.
We can express $D $ in a more familiar and concrete form by using 
a biframe as follows (much as done in proposition 2.9 of \cite{R27}). 
Let $\{\e_j \} $ be a basis for $\fg'$, 
and let $\{E_j\} $ be the dual basis for $\fg $. Let $ P $ be the 
orthogonal projection from $\fg $ onto $\fm$, and let $ P'$ be the
orthogonal projection, from $\fg'$ onto $\fm'$, 
extended to the complexifications. 
Notice that $ P $ commutes with the restriction to $ K$
of $\Ad $, and similarly for $ P'$. For each $j $ define an element, 
$\mathsf{E}_j$, of $\cT^\bC(G/K)$ by
\[
\mathsf{E}_j(x) = P \Ad^{-1}_x(E_j),
\]
and an element, 
$\eta_j$, of $\O(G/K)$ by
\[
\eta_j(x) = P' \Cad^{-1}_x(\e_j).
\]
Then the pair $(\{\mathsf{E}_j\}, \{\eta_j\})$ is a biframe for $\cT(G/K)$
in the sense that it has the reproducing property
\[
V = \sum  \<V, \eta_j\>_\cA \mathsf{E}_j
\]
for all $V \in \cT^\bC(G/K)$. Then
\[
(\rd \Psi)(V) \ = \ \nabla^\cS_V\Psi \ = \  \nabla^\cS_{ \sum \<V, \eta_j\>_\cA \mathsf{E}_j }\Psi  
 \ =  \sum  \<V, \eta_j\>_\cA \nabla^\cS_{ \mathsf{E}_j}\Psi ,  
\]
so that
\[
\rd \Psi =  \sum  \eta_j \otimes \nabla^\cS_{ \mathsf{E}_j}\Psi  ,
\]
where $\eta_j $ can be viewed as an element of  $\bC\ell(G/K)$
via the inclusion $\O(G/K) \subset \bC\ell(G/K)$.
Then 
\[
D\Psi = \sum \k_{\eta_j}(\nabla^\cS_{ \mathsf{E}_j}\Psi )  .
\]

When we combine the $\ciA$-valued
inner product on  $\cS(G/K) $ with integration
by the left-invariant probability 
measure on $ G/K $, we obtain an ordinary inner product
on  $\cS(G/K) $. On completing $\cS(G/K) $ for this inner product, we obtain a
Hilbert space. In this way we can view $ D $ as a densely defined operator on a
Hilbert space. 

\begin{definition}
\label{defdir2}
The Dirac operator, $D $,  for the given $\Cad$-invariant inner product on $\fg'$, 
and the given spinor bundle, is the operator $ D $ defined above.
\end{definition}

As explained in theorem 1.7i of \cite{Ply} and proposition 9.4 of \cite{GVF} and later pages, 
spinor bundles for Clifford bundles are not in general unique. The tensor product of a 
spinor bundle by a line bundle will be another spinor bundle, and all the irreducible 
spinor bundles are related in this way. Within our setting of equivariant bundles, 
we need to tensor with $G$-equivariant line bundles. These correspond exactly to 
the characters, that is, one-dimensional representations, of $K$. This is all relevant 
to coadjoint orbits because in that case $K $ always has non-trivial characters. All of our discussion of 
this following equation 5.2 of \cite{R27}
carries over to the present situation with only very minor changes. 
So we will not discuss this further in this paper, and refer the interested reader to \cite{R27}.

Because essentially all of the operations defined above commute with the 
actions of G by left translation, it is easily checked that:

\begin{proposition}
\label{dircom}
The Dirac operator $D$ commutes with the action of G on $\cS(G/K)$ by left translation.
\end{proposition}

For $f\in \ciA$ let $M_f$ denote the operator on $\cS(G/K)$ coming from
$\cS(G/K)$ being an $\ciA$-module. Notice that $\hat \rd f \in \bC\ell(G/K)$ 
via $\hat \O(G/K) \subset \bC\ell(G/K)$  so that
$\k_{\hat \rd f}$ is an operator on $\cS(G/K)$. Then
much as in proposition 8.3 of \cite{R22} we have
\begin{proposition}
\label{dirbra}
For any $f \in \ciA$ we have
\[
[D, M_f] = \k_{\hat \rd f}   .
\]
\end{proposition}
\begin{proof}
For $f \in \ciA$ and $\Psi \in \cS(G/K)$ the Leibniz rule
\[
\rd (f\Psi) = (\hat \rd f) \otimes \Psi + f \rd \Psi
\]
follows from equation \ref{eq6L}.
Then, since $\k$ is $\ciA$-linear, we have
\begin{align*}
[D,M_f]\Psi &= D(f\Psi) - fD(\Psi) = \k(\rd(f\Psi)) - f\k(\rd \Psi) \\
&= \k(\hat \rd f \otimes \Psi + f\rd \Psi) - \k(f \rd \Psi) = \k_{\hat \rd f}\Psi .
\end{align*}  
\end{proof}

Recall that for any $f \in \ciA$  and $x \in G$ we have $\hat\rd f_x \in \fm'^\bC$,
so that $\|\hat \rd f_x\|_{\fm'}$ is defined in terms of the inner product on $\fm'^\bC$.

\begin{theorem}
\label{thnorm}
For any $f \in \ciA$ we have
\[
\|[D, M_f]\| = \sup \{\|\hat \rd f_x\|_{\fm'^\bC}: x \in G\}
=  \sup \{\| \rd f_x\|_{\fm'^\bC}: x \in G\}    .
\]
\end{theorem}
\begin{proof}
Since $\bC\ell(G/K)$ is a $*$-subalgebra of the C*-algebra 
$\cA\otimes \bC\ell(\fm') = C(G, \bC\ell(\fm'))$, and we are viewing $\hat \rd f $ 
as an element of $\bC\ell(G/K)$, we have
\[
\|\hat \rd f\| =  \sup \{\|\hat \rd f_x\|_{\bC\ell(\fm')}: x \in G\}.
\]
But, just as seen before Proposition \ref{bigb}, if $\mu \in \fm'^\bC$
then $\|\mu\|_{\bC\ell(\fm')} = \|\mu\|_{\fm'^\bC} $.
The first equality in the statement 
of the theorem follows directly from this.

Because $\Cad $ is isometric on $\fg'^\bC $, we have
$\|\hat \rd f_x\|_{\fm'^\bC} = \|\rd f_x\|_{\fm'^\bC} $ for all $x \in G$. 
This gives the second equality in the statement of the theorem.
\end{proof}

When we recall the notation of Proposition \ref{bigb} and earlier,
we see that we obtain:
\begin{corollary}
\label{cormet}
For any $f \in \ciA$ we have
\[
 \|[D, M_f]\| =  \|[D_o, M_f]\|
\]
\end{corollary}
Thus the corresponding C*-metrics on $\cA$ are equal. This will be
important for the proof of our main theorem.

For the reasons given immediately after Proposition \ref{bigb}, 
we then obtain:

\begin{corollary}
\label{corgrad}
For any $f \in \ciA$ we have
\[
\|[D, M_f]\| = \sup \{\|\grad_x f\|_{\fm^\bC}: x \in G\}   .
\]
\end{corollary}

Now a standard argument (e.g., see the text between definition 9.12 
and exercise 9.7 of \cite{GVF}) shows that if we denote 
by $d$ the ordinary metric on a Riemannian manifold $N$ 
coming from its Riemannian metric, then for any two 
points $p$ and $q$ of $N$ we have
\[
d(p,q) = \sup\{|f(p)-f(q)|: \|\grad_f\|_\infty \le 1\}.
\]
On applying this to $G/K$ , we obtain, 
for $d$ now the ordinary metric on $G/K$ from our Riemannian metric,
\[
d(p,q) = \sup\{|f(p)-f(q)|: \|[D, M_f]\| \le 1\}.
\]
This is the formula on which Connes focused for general Riemannian 
manifolds \cite{Cn7,Cn3}, as it shows that the Dirac operator contains 
all of the metric information (and much more) for the manifold.  
This is his motivation for advocating that metric data 
for ``non-commutative spaces'' be encoded by providing 
them with a ``Dirac operator''. 

We do not need to know the formal self-adjointness of $ D $ in this paper.
We refer the interested reader to theorem 6.1 of \cite{R27}, and to
theorem 8.4 of the most recent arXiv version of \cite{R22}. 
(The published version of \cite{R22} has a serious error in the 
proof of theorem 8.4.)


\section{Complex structure on coadjoint orbits}
\label{cplx}
We want to show that in the case of coadjoint orbits we can choose a 
spinor bundle whose fibers are irreducible. The path to showing this 
involves showing that coadjoint  orbits have complex structures. In fact,
coadjoint orbits can be equipped with K\"ahler  structures, 
as described in section 1 of \cite{R27} (and see also the theorem
at the end of section 3 of \cite{BFR}), but we do not 
need the full force of that fact. We will discuss in Section \ref{secdefit}
why our approach does not always fit perfectly 
with the  K\"ahler structures. 
So here we construct the complex structures directly,
along the lines discussed in \cite{AlC}. This requires 
the basic facts concerning weights and roots of compact
semisimple Lie groups. In the next few paragraphs we follow 
somewhat closely the notation and 
development in section 1 of \cite{R27}.

We assume now that $G $ is a connected compact semisimple 
Lie group, with Lie algebra $\fg $.  The coadjoint orbits are the 
orbits in $\fg'$ for the action $\Cad$.  Fix $\mu_{\dia} \in \fg'$, 
with $\mu_{\dia} \ne 0$.   
Let $K$ denote the $\Cad$-stability subgroup of $\mu_{\dia}$, so that $x \mapsto \Cad_x(\mu_{\dia})$ gives a $G$-equivariant diffeomorphism from $G/K$ onto the $\Cad$-orbit of $\mu_{\dia}$.  We will usually work with $G/K$ rather than the orbit itself. 

We choose an $\Ad $-invariant inner product on $\fg $, for example, the negative
of the Killing form of $\fg$ (since $G $ is compact). 
This actually is not much more general than choosing 
the negative of the Killing form itself, because if $\fg $ is simple, then every
$\Ad $-invariant inner product on $\fg $ is just a scaler multiple of
the negative of the Killing form, whereas if $\fg $ is just semisimple,
then every
$\Ad $-invariant inner product on $\fg $ just arises from taking various 
scaler multiples of
the negatives of the Killing forms on its simple ideals.

The 
action $\Ad$ of $G$ on $\fg$ is by orthogonal operators, 
and so the action $\ad$ of $\fg$ on $\fg$ is by skew-adjoint operators . 
There is a unique $Z_{\dia} \in \fg$ such that 
\begin{equation}
\label{eq1.1}
\mu_{\dia}(X) = \<X,Z_{\dia}\> \text{ for all } X \in \fg.
\end{equation}
It is easily seen that the $\Ad$-stability subgroup of $Z_{\dia}$ is again $K$.

Let $T_{\dia}$ be the closure in $G$ of the one-parameter group $r \mapsto \exp(rZ_{\dia})$, so that $T_{\dia}$ is a torus subgroup of $G$.  Then it is easily seen that $K$ consists exactly of all the elements of $G$ that commute with all the elements of $T_{\dia}$.  Note that $T_{\dia}$ is contained in the center of $K$ (but need not coincide with the center).  Since each element of $K$ will lie in a torus subgroup of $G$ that contains $T_{\dia}$, it follows that $K$ is the union of the tori that it contains, and so $K$ is connected (corollary 4.22 of  \cite{Knp}).  Thus for most purposes we can just work with the Lie algebra, $\fk$, of $K$ when convenient.  In particular, $\fk = \{X \in \fg: [X,Z_{\dia}]  = 0\}$, and $\fk$ contains the Lie algebra, $\ft_{\dia}$, of $T_{\dia}$.

Let $\fm = \fk^{\perp}$ with respect to chosen inner product.  
Since $\Ad$ preserves the inner product, we see that $\fm$ is carried into 
itself by the restriction of $\Ad$ to $K$.  Thus $[\fk,\fm] \subseteq \fm$.  
As we have seen in Section \ref{dhom}, 
$\fm$ can be conveniently identified with the tangent 
space to $G/K$ at the coset $K$ (which corresponds to the point $\mu_{\dia}$ 
of the coadjoint orbit). To define a complex structure on $ G/ K $ we need to 
define a complex structure on $\fm $ 
that commutes with the action of $ K $ via $\Ad $. 

Fix a choice of a maximal torus, $ T $, of $ G $ that contains $T_{\dia}$. 
Then $ T $ is contained in $ K $, and is a maximal torus in $ K $. (Thus we 
are in the setting to which the results of \cite{Kst} apply, since $ K $ has 
the same rank as $ G $. Our $\fk$ is the $\mathfrak{r}$ there, while our $\fm$
is the $\mathfrak{p}$ there, up to complexification. 
But we will not use results from that paper.) We denote the 
Lie algebra of $ T $ by the traditional $\fh $. Since $ K $ is compact, it is reductive,
and so $\fk $ splits as $\fk = \fh_m \oplus \fk_s $ where $\fk_s $ is the 
semisimple subalgebra $\fk_s = [\fk, \fk] $ of $\fk $
and $\fh_\fm $ is the center of $\fk $ (so $\fh_\fm \subseteq \fh$). 
Thus $\fg = \fm \oplus \fh_\fm \oplus \fk_s $.
Furthermore, $\fh $ splits as $\fh_\fm \oplus \fh_s $ 
where $\fh_s$ is a Cartan subalgebra of $\fk_s $.

 For any finite-dimensional unitary representation $(\cH, \pi)$
of $T$ we let $\pi$ also denote the corresponding representation of $\fh$. For each $H \in \fh$
the operator $\pi_H$ is skew-adjoint, and so its eigenvalues are purely imaginary. Since
the $\pi_H$'s all commute with each other, they are simultaneously diagonalizable. Because
we need to keep track of the structure over $\bR$, we will use a convention for the weights of
a representation that is slightly different from the usual convention. We used it previously 
in \cite{R33}. If $\xi \in \cH$ is a common
eigenvector for the $\pi_H$'s, there will be a linear 
functional $\a$ on $\fh$ (with values in $\bR$) 
such that 
\[
\pi_H(\xi) = i\a(H)\xi
\]
for all $H \in \fh$.
For each $\a \in \fh'$ (where $\fh'$ denotes the dual vector space to $\fh$) we set
\[
\cH_\a = \{ \xi \in \cH:\pi_H(\xi) = i\a(H) \xi \ \ \mathrm{for \ all} \  H \in \fh\}  .
\]
If there are non-zero vectors in $\cH_\a$ then we say that $\a$ is a \emph{weight} of the
representation $(\cH, \pi)$. We denote the set of all weights for this representation by
$\D_\pi$. Then
\[
\cH = \bigoplus \{\cH_\a : \a \in \D_\pi \}    .
\]

We let $\fg^\bC$ denote the complexification of $\fg$, with inner product
coming from that on $\fg$, and corresponding unitary representation
$\Ad$ of $G $ on $\fg^\bC$.
The non-zero weights for $\Ad$ and $\ad$
acting on $\fg^\bC$, 
are called the ``roots'' of $G$. We denote the set
of roots simply by $\D$. Because we are dealing with the 
complexification of a representation 
over $\bR$, if $\a \in \D$ then
$-\a \in \D$. 
For each root $\a$ we let $\fg^\bC_\a$ denote the corresponding root space. 
It is 
a standard fact that these root spaces are all of dimension 1, and 
that $[\fg^\bC_\a, \fg^\bC_{-\a}]$ is not of dimension 0 (so is of dimension 1). 
In the standard way \cite{Knp, Srr, Smn} we make a choice, $\D^+$, of  
positive roots. 
We want to choose usual elements $H_\a$, $E_\a$ and $F_\a$ 
for each $\a \in \D^+$. 
but we will need to choose them in a careful way 
so that they mesh well with representations.

Let $(\cH, \pi)$ be a finite-dimensional unitary 
representation of $G$. We extend the corresponding
representation of $\fg$  to a representation (still denoted by $\pi$) of $\cfg$.
Let $W \in \cfg$ with $W = X + iY$ for $X, Y \in \fg$. Since elements of $\fg $
act on $\cH $ as skew-adjoint operators,
\[
(\pi_W)^* = (\pi_X)^* + (i\pi_Y)^* = \pi_{(-X + iY)}    .
\]
Thus it is appropriate to define an involution on $\cfg$ by $(X+iY)^* = -X + iY$,
so that $(\pi_W)^* = \pi_{W^*}$ for all $W \in \cfg$ (as in \cite{Hll, Smn}). 
Notice that for all $W, Z \in \cfg$
we have $[W, Z]^* = [Z^*, W^*]$.

The following result is well-known. It occurs with proof as proposition 1.1
of \cite{R33}.

\begin{proposition}
\label{HEF}
With notation as above, for each $\a \in \D^+$ we can choose $H_\a \in i\fh$ and
$E_\a \in \fg^\bC_\a$  such that $[E_\a, E_\a^*] = H_\a$ and
$[H_\a, E_\a] = 2E_\a$. Setting
$F_\a = E_\a^*$, we then obtain $[H_\a, F_\a] = -2F_\a$.
\end{proposition}

Notice that $\cfk = \{X \in \cfg : [Z_{\dia} , X] = 0\} $.
If $ X\in \fg_\a^\bC $ then $[Z_{\dia} , X] = i\a(Z_{\dia} )X $.
It follows that 
\[
\cfm =\bigoplus \{ \fg_\a^\bC: \a(Z_{\dia}) \neq 0\}, 
\]
and that
\[
\cfk = \cfh \oplus \bigoplus \{ \fg_\a^\bC: \a(Z_{\dia}) = 0\}.
\]
It then follows that
\[
\fk_s^\bC = \fh_s^\bC  \oplus \bigoplus \{ \fg_\a^\bC: \a(Z_{\dia}) = 0\}  .
\]
Thus it is appropriate to set
\[
\D_s = \{\a \in \D:  \a(Z_{\dia}) = 0\} = 
\{\a \in \D:  \a(Z) = 0 \ \mathrm{for \ all} \  Z \in \fh_\fm^\bC\}  ,
\]
and to identify $\D_s $ with the root system for $\fk_s^\bC $ with respect
to its Cartan subalgebra $\fh_s^\bC $. It is also appropriate to set
\[
\D_\fm = \{\a \in \D:  \a(Z_{\dia}) \neq 0\} = 
\{\a \in \D:   \ \mathrm{there \  is \ 
a} \   Z \in \fh_\fm^\bC \ \mathrm{with}  \ \a(Z) \neq 0  \}  .
\]
Note that $\D_\fm $ is ``invariant" with respect to $\D_s $
in the sense that if $\a\in \D_\fm $ and $\b \in \D_s $ and if
$\a +\b\in\D $, then $\a +\b\in\D_\fm $.

We need a partial order on $\D_\fm $, given by specifying
a subset $\D_\fm^ + $ with the usual properties, 
but which in addition is $\D_s $-invariant in the sense that if
$\a\in\D_\fm^ + $ and $\b\in\D_s $ and if $\a +\b \in \D $, then $\a +\b \in \D_\fm^+ $.
In general there are many such partial orders, coming from 
Weyl-chamber-type considerations. (See \cite{AlC}.) 
But there is one such order canonically associated with
$Z_{\dia} $, namely specified by
\[
\D_\fm^+ =\{\a\in\D_\fm: \a(Z_{\dia})>0\}   .
\]
Notice that this makes sense because $\a(Z_{\dia}) \in \bR $ since 
$Z_{\dia} \in \fh $ and $\a \in \fh'$.
It is clear that $\D_\fm^+ $  is  $\D_s $-invariant.

It is then natural to set $\fn^+ =\bigoplus\{\fg_\a^\bC: \a\in \D_\fm^+\}  $, 
and similarly for $\fn^-$.
Then $\cfm = \fn^+ \oplus \fn^-$.
 
\begin{notation}
Define an operator, $J $, on $\cfm $ by 
\begin{equation*}
J(X) =\left\{
\begin{array}{rl}
iX & \text{if  } X \in \fn^+ ,  \\
-iX & \text{if  } X \in \fn^-  .
\end{array}
\right.
\end{equation*}
\end{notation} 
Notice that $J $ is isometric.

We show now that $J $ commutes with the $\ad $-action of $\fk^\bC$ on $\cfm $.
This uses the $\D_s $-invariance of $\D_\fm^+ $.
Let $\b\in \D_s $. Then for any $\a\in \D_\fm^+ $ we have
\[
J(\ad_{E_\b}(E_\a)) = J([E_\b, E_\a]) =i[E_\b, E_\a] =[E_\b, iE_\a]
= \ad_{E_\b}(J(E_\a))
\]
since $\a +\b\in\D_\fm^+ $ if $\a +\b\in\D $. If instead $ Z\in \cfh \subseteq \cfk$ then
\[
J(\ad_Z(E_\a))  = J(i\a(Z)E_\a) =  i\a(Z)J(E_\a)     = \ad_Z(J(E_\a))  .
\]
Thus $J $ commutes with the $\ad $-action of $\fk^\bC$ on $\fn^+ $. A similar
calculation shows that $J $ commutes with the $\ad $-action of $\fk^\bC$ on $\fn^- $,
and so on $\cfm$, as desired.

It is clear from the definition of $ J $ that 
$ J^2 = -I^\cfm $, for $I^\cfm $ the identity operator on $\cfm $.

Ordinary complex conjugation on $\cfg $ with respect to $\fg $ carries $\fg_\a^\bC $ onto
$\fg_{-\a}^\bC $ for each $ \a\in\D $. It follows that complex conjugation carries $\fn^+ $ onto
$\fn^- $, and consequently, complex conjugation commutes with $ J $. 
It follows that $J $ carries
$\fm $ onto itself, and so is a complex structure on $\fm $. Of course $J $ commutes 
with the actions of $ K $ and $\fk $ on $\fm $.

Let $\a \in\D_\fm^+$, and consider $ E_\a\in \fg^\bC_\a $. Then $\bar E_\a\in \fg^\bC_{-\a} $,
where the bar denotes the ordinary complex conjugation. Set
\[
X_\a = E_\a  +\bar E_\a \quad \quad \mathrm{and} \quad \quad Y_\a = i(E_\a  - \bar E_\a) .
\] 
They are invariant under complex conjugation, and so they are in $\fg $, and they span
$(\fg^\bC_\a  \oplus \fg^\bC_{-\a}) \cap \fg $ over $\bR $, which we denote by
$\fm_\a $. (So care must be taken not to confuse $\fm_\a $ with root spaces of $\cfg $.)
It is easily checked that 
\[
J(X_\a) = Y_\a   \quad \quad \mathrm{and} \quad \quad J(Y_\a) = -X_\a  ,
\] 
and that for any $Z\in \fh_\fm $ we have
\[
[Z,X_\a] = \a(Z) Y_\a =  \a(Z) J(X_\a)    \quad \mathrm{and} 
 \quad [Z,Y_\a] = -\a(Z) X_\a =   \a(Z)  J(Y_\a),
\] 
Thus if we view $J $ as ``multiplication by $i $", we see that
$\fm_\a $ is a weight space (of dimension 1) for the representation of
$\fh_\fm $ on $\fm $. Note that $ J $ is isometric on $\fm $, 
since earlier we saw that it is isometric on $\cfm $.

\begin{definition}
\label{complex}
The operator $J $ on $\fm $ defined above is the complex structure on
$\fm $ canonically associated to the element $\mu_\dia $.
 When $\fm $ is equipped with this complex
structure, we will denote it by $\fm_ J $.
\end{definition}

We remark that many elements of $\fg'$ can determine the same complex structure 
on $\fm $
--- all the ones in the same Weyl-type chamber.

\begin{definition}
\label{dualcom}
We will denote the adjoint of $ J $ on
$\fm' $ again by $ J $. It is the complex structure on
$\fm' $ canonically associated to the element $\mu_\dia $.
 When $\fm' $ is equipped with this complex
structure, we will often denote it by $\fm'_ J $.
\end{definition}

Then $ J $ on $\fm'$ is isometric, and will commute with the representation
$\Cad$ of $ K $ on $\fm'_ J$.


\section{Yet more about spinors}
\label{mspin}
We are now exactly in position to use some of the main results of section 4 of \cite{R27}.
We consider a general even-dimensional Hilbert space $\fm $ over $\bR $
(that will later be our $\fm'$),
that is equipped with an isometric complex structure $J $.
We form the complex Clifford algebra $\bC\ell(\fm)$ for the given inner product.
Let $O(\fm) $ be the group of orthogonal transformations of $\fm $.
By the universal property of Clifford algebras, each element of
$O(\fm) $ determines an automorphism of $\bC\ell(\fm)$ (a ``Bogoliubov
automorphism''). We denote the corresponding action of $O(\fm) $ on
$\bC\ell(\fm)$ by $\b $.

We seek to construct an irreducible representation, $\k $, of $\bC\ell(\fm) $
on a Hilbert space $\cS $ that is suitably compatible with the action $\b $.
To construct this, we use the complex structure $ J $.
When we view $\fm $ as a complex Hilbert space using $ J $, we will denote
it by $\fm_ J $. Let $ U(\fm_ J) $ denote the group of unitary operators on
$\fm_ J $. It is the subgroup of those elements of $ O(\fm) $ that commute with
$ J $. Notice that the complex dimension of $\fm_ J $ is half of the real dimension
of $\fm $.

Let $\cS $ be the exterior algebra over $\fm_ J $. 
By the universal property of exterior algebras, each element of
$U(\fm_ J) $ determines an automorphism of $\cS$ (a ``Bogoliubov
automorphism''). We denote the corresponding action of $U(\fm_ J) $ on
$\cS$ by $\rho $.
There is a standard way
of defining an irreducible representation, $\k $, of $\bC\ell(\fm) $ on $\cS $,
called the Fock representation, in terms of
annihilation and creation operators. This is described in section 4 of \cite{R27},
strongly influenced by the thorough exposition in \cite{GVF} beginning 
with definition 5.6. (which uses the opposite sign convention than we
use for the definition of Clifford algebras). See also
the the discussion after corollary 5.17 of \cite{LwM}. We will not
describe the construction here.
But by examining the explicit construction, as done in \cite{R27}, 
we are able to obtain 
the following crucial result, which is just a restatement of proposition 4.4 
of \cite{R27} with very minor changes of notation:

\begin{proposition} 
\label{sppat}
The  representation $\rho $ of
$ U(\fm_J) $ on $\cS $ is compatible with the action $\b $ of
$ U(\fm_J) $ on $\bC\ell(\fm) $ in the sense that
\[
\k_{\b_ R(q)} =\rho_R \k_q \rho^{-1}_ R 
\]
 for all
$ R\in U(\fm_ J) $ and $q\in \bC\ell(\fm) $. 
\end{proposition}
We refer the reader to \cite{R27} for the proof.


\section{Dirac operators for coadjoint orbits}
\label{dorb}
In this section we combine the results of several previous sections to construct
Dirac operators for coadjoint orbits of compact semisimple Lie groups. 
We use the notation of Section \ref{cplx}.  
We apply the results of Section \ref{cplx}, but in the role of the
$\fm $ in that section we will use $\fm' \subset \fg'$, with its
complex structure $J $ defined at the end of Section \ref{cplx}. 
We form the complex Clifford algebra
$\bC\ell(\fm') $, the exterior algebra $\cS $ over $\fm'_ J $, 
and the irreducible representation
$\k $ of $\bC\ell(\fm') $ on $\cS $.
We saw that the representation $\Cad $ restricted to $K $ on $\fm'$ extends
to an action $\b $ of $K $ on $\bC\ell(\fm') $. We also saw that this representation
commutes with $ J $, and so is a unitary representation of $K $ on $\fm'_ J $,
that is, a homomorphism from $K $ into $ U(\fm_ J) $. 
Thus it extends to a unitary representation $\rho $ of $K $ on $\cS $.
Of crucial importance, from Proposition \ref{sppat} we obtain the compatibility 
relation
\[
\k_{\Cad_ s(q)} =\rho_s (\k_q)\rho^{-1}_s 
\]
 for all $ s\in K $ and $q\in \bC\ell(\fm) $.  
 For the present situation, this is exactly the compatibility
relation \eqref{comeq} that was assumed in Section \ref{dhom}.

As in Section \ref{dhom}, we form the Clifford bundle $\clf G/K)$.
Using the exterior algebra $\cS $, we construct the spinor
bundle $\cS(G/ K) $ as in Section \ref{dhom}, 
with the representation $\rho $ of $ K $ on it. 
We let $\k $ be the pointwise action
of $\clf G/ K) $ on $\cS(G/ K) $. As seen in Section \ref{dhom}, the
compatibility relation is needed in order to ensure that $\k $
carries $\cS(G/ K) $ into itself.

We now let $ D $ be the operator on  $\cS(G/ K) $ 
constructed exactly as done in 
Section  \ref{dhom} before Definition \ref{defdir2}. 

\begin{definition}
\label{defdir3}
The operator, $D $, defined above is the Dirac operator 
on $ G / K $ (i.e. on $\cA = C(G/K) $) for the given 
element $\mu_{\dia} \in \fg'$ and the given
$\Cad$-invariant inner product on $\fg'$. 
\end{definition}

As we will discuss in Section \ref{secdefit}, $D $ is not always the Dirac 
operator corresponding to the K\"ahler structure on $G/K$ determined by
$\mu_{\dia} $ and constructed in \cite{R27}.

From Proposition
\ref{dircom} we know that
$D$ commutes with the action of G on $\cS(G/K)$ by left translation.
From Proposition \ref{dirbra} we immediately obtain:
\begin{proposition}
\label{dirbra2}
For any $f \in \ciA$ we have
\[
[D, M_f] = \k_{\hat \rd f}   .
\]
\end{proposition}

From Theorem \ref{thnorm} we immediately obtain:
\begin{theorem}
\label{thnorm2}
For any $f \in \ciA = C^\infty(G/K)$ we have
\[
\|[D, M_f]\| = \sup \{\|\hat \rd f_x\|_{\fm'^\bC}: x \in G\}
 = \sup \{\| \rd f_x\|_{\fm'^\bC}: x \in G\}  .
\]
\end{theorem}

Then from Corollary \ref{cormet} we immediately obtain
\begin{corollary}
\label{cormt2}
For any $f \in \ciA$ we have
\[
 \|[D, M_f]\| =  \|[D_o, M_f]\|   .
\]
\end{corollary}
Thus the corresponding C*-metrics on $\cA$ are equal.
This fact will be important in the next sections. 
The comments after Corollary \ref{corgrad} apply equally well here.


\section{Bridges with symbols}
\label{brsy}
The definition of quantum Gromov-Hausdorff distance between compact
C*-metric spaces has evolved over the years since the first definition
was proposed in \cite{R6}. At present the definition with the best properties is
\Lat's dual Gromov-Hausdorff propinquity \cite{Ltr4}, and it forms
the base for the spectral propinquity \cite{Ltr7}. 
But its definition can be  somewhat 
difficult to work with directly for some classes of examples, 
including the examples we are considering. \Lat \ had slightly 
earlier introduced a somewhat stronger definition that 
he called the Gromov-Hausdorff propinquity \cite{Ltr2}. 
This definition works well for our examples. So in this section we will recall the definition of the 
Gromov-Hausdorff propinquity and explain 
how it relates to our situation. With this as preparation, in the next section we will 
prove that a suitable sequence of matrix algebras, equipped with C*-metrics
coming from Dirac operators, 
converges for the Gromov-Hausdorff propinquity to the
coadjoint orbit for a given highest-weight vector.
(This will imply convergence also for the dual Gromov-Hausdorff propinquity). 

 For any two unital C*-algebras $\cA$ and $\cB$, a
bridge from  $\cA$ to $\cB$ in the sense of \Lat \ \cite{Ltr2} is a quadruple
$(\cD, \pi_\cA, \pi_\cB, \om )$ for which $\cD$ is a unital C*-algebra,
$\pi_\cA$ and $\pi_\cB$ are unital injective homomorphisms of
$\cA$ and $\cB$ into $\cD$, and $\om$ is a self-adjoint element
of $\cD$ such that 1 is an element of the spectrum
of $\om$ and  $\|\om \| = 1$. Actually, \Lat \ only requires a looser 
but more complicated condition on $\om$, but the
above condition will be appropriate for our examples. Following
\Lat, \ we will call $\om$ the ``pivot'' for the bridge.
We will often omit mentioning the injections $\pi_\cA$ and $\pi_\cB$
when it is clear what they are from the context, and accordingly we
will often write as though $\cA$ and $\cB$ are unital subalgebras of $\cD$.

For our applications, $\cA$ will be $ C(G/ K)$ for $G $ and $K $ as 
in previous sections, and $\cB =\cB(\cH)$ will be the 
matrix algebra corresponding 
to an irrep  $(\cH, U) $ of $G $. Let $\a $ be the action of $G$ 
on $\cB$ by conjugation by $U$.
Let $P $ be the rank-one projection on the highest-weight 
space of the irrep. We assume that $K$ is the
stabilizer subgroup of $P$ for $\a$.
For our bridge
we take $\cD$ to be the C*-algebra 
\[
\cD = \cA \otimes \cB = C(G/K, \cB)  .
\]
We take $\pi_\cA$ to be the injection of $\cA$ into $\cD$
defined by
\[
\pi_\cA(a) = a \otimes 1_\cB
\]
for all $a \in \cA$, where $1_\cB$ is the identity element of $\cB$. 
The injection $\pi_\cB$ is defined similarly.
We define the pivot $\om$ to be the coherent state associated
to the irrep, that is, $\om$ is the function in $C(G/K, \cB)$ defined by
\begin{equation}
\label{pivot}
\om(x) = \a_x(P)
\end{equation}
for all $x \in G/ K$.
We notice that $\om$ is actually 
a non-zero projection
in $\cD$, and so it satisfies the requirements for being a pivot.
We will denote the bridge $(\cD, \om)$ by $\Pi$. 

\begin{definition}
\label{mybrg}
We will call the bridge $\Pi $ constructed just above
\emph{the bridge from $ G/K $ (or from $\cA = C(G/K) $) to
the irrep $(\cH, U) $ (or to $\cB = \cB(\cH) $)}.
\end{definition}

Any choice of C*-metrics $L^\cA$ and $L^\cB$ on 
unital C*-algebras $\cA$ and $\cB$
can be used to measure any given
bridge $\Pi = (\cD, \om) $ between $\cA$ and $\cB$.
\Lat \  \cite{Ltr2} defines the ``length'' of the
bridge by first defining its ``reach''
and its ``height''.  

\begin{definition}
\label{reach}
Let $\cA$ and $\cB$ be unital C*-algebras, and let
$\Pi = (\cD, \om)$ be a bridge from $\cA$ to $\cB$ . 
Let $L^\cA$ and $L^\cB$ be C*-metrics on $\cA$ and $\cB$. Set 
\[
\cL_\cA^1 = \{a \in \cA:  a = a^* \ \mathrm{and} \ L^\cA(a) \leq 1\}   ,
\]
and similarly for $\cL_\cB^1$. (This is slightly different from 
equation \eqref{ball}.) We can view these as subsets of $\cD$. 
Then the \emph{reach} of $\Pi$ is given by:
\[
\mathrm{reach}(\Pi) = \mathrm{Haus}_\cD\{\cL_\cA^1  \om \ , \     \om  \cL_\cB^1\}   ,
\]
where $\mathrm{Haus}_\cD$ denotes the Hausdorff distance with respect
to the norm of $\cD$, and where the product defining $\cL_\cA^1  \om$ and
$ \om  \cL_\cB^1$ is that of $\cD$.  
\end{definition}

\Lat \ shows
just before definition 3.14 of \cite{Ltr2} that, 
under conditions that include the case in which 
$(\cA, L^\cA)$ and $(\cB, L^\cB)$ are 
C*-metric spaces, the reach of $\Pi$ is finite. 

To define the height of $\Pi$ we need to consider the state space, $S(\cA)$,
of $\cA$, and similarly for $\cB$ and $\cD$. Even more, we set
\[
S_1(\om) = \{\phi \in S(\cD): \phi(\om) = 1\}     ,   
\]
the ``level-1 set of $\om$''. The elements of $S_1(\om)$ are ``definite'' 
on $\om$ in the sense \cite{KR1} that 
for any $\phi \in S_1(\om)$ we have
\[
\phi(d\om) = \phi(d) = \phi(\om d)     .
\]
for all $d \in \cD$.
Since $L^\cA$ is a C*-metric, it determines by formula \eqref{stmet}
an ordinary metric, $\rho_\cA $, on $S(\cA)$, which metrizes the weak-$*$
topology, for which $S(\cA) $ is compact. 
Define $\rho_\cB$ on $S(\cB)$ similarly. 
\begin{notation}
We denote by $S_1^\cA(\om)$ the restriction of the
elements of $S_1(\om)$ to $\cA$. We define $S_1^\cB(\om)$ 
similarly. 
\end{notation}

\begin{definition}
\label{height}
Let $\cA$ and $\cB$ be unital C*-algebras and let
$\Pi = (\cD, \om)$ be a bridge from $\cA$ to $\cB$ . 
Let $L^\cA$ and $L^\cB$ be C*-metrics on $\cA$ and $\cB$. 
The \emph{height} of the bridge $\Pi$ is given by
\[
\mathrm{height}(\Pi) =
\max\{\mathrm{Haus}_{\rho_\cA}(S_1^\cA(\om), S(\cA)) , \ 
\mathrm{Haus}_{\rho_\cB}(S_1^\cB(\om) , S(\cB))\}  ,
\]
where the Hausdorff distances are with respect to the indicated
metrics. 
The length of $\Pi$ is then defined by
\[
\mathrm{length}(\Pi) = \max\{\mathrm{reach}(\Pi), \mathrm{height}(\Pi)\}  .
\]
\end{definition}

\Lat \ defines the length of a finite path ( a ``trek'') of bridges to be the sum of the lengths
of the individual bridges. He then defines the Gromov-Hausdorff propinquity
between two compact C*-metric spaces to be the infimum
of the lengths of all finite paths between them. (This gives the triangle inequality.) 
He proves the remarkable fact that if
the propinquity between two compact C*-metric spaces is 0, 
then they are isometric in the sense that there is an isomorphism 
between the C*-algebras 
that carries the C*-metric on one to the C*-metric on the other.
Thus the propinquity is a metric on the set of isometry classes of compact 
C*-metric spaces. But we will not need to deal directly with 
finite paths of bridges because we will prove that, for the sequences of 
bridges that we will 
construct, already their lengths will converge to 0.

For the main context of this paper there is extra structure 
available to help with measuring the lengths of bridges. 
The next paragraph is strongly motivated by the discussion of 
bridges with conditional expectations in sections 4 and 5 of \cite{R29}.

\begin{definition}
\label{symb}
Let $\cA$ and $\cB$ be unital C*-algebras, 
 and let
$\Pi = (\cD, \om)$ be a bridge from $\cA$ to $\cB$. By a \emph{pair of symbols} for
$\Pi$ we mean a pair of unital completely positive maps
$(\s^\cA, \s^\cB) $ such that $\s^\cA $ maps $\cD $ to $\cA $ and
$\s^\cA(\om) = 1_\cA $, while
$\s^\cB $ maps $\cD $ to $\cB $ and $\s^\cB(\om) = 1_\cB $. \\
(Thus $\s^\cA $ and $\s^\cB$ are ``definite" on $\om $. 
We do not require 
any relation between these two maps.)
\end{definition}

For our applications we define $\s^\cA $ by
\begin{equation}
\label{sya}
\s^\cA(F)(x) = \mathrm{tr}_\cB(F(x)\a_x(P)) 
\end{equation}
for any $ F\in \cD = C(G/K, \cB) $ , 
where $\mathrm{tr}_\cB $ is the un-normalized 
trace on $\cB $.
We use the term ``symbol" in definition \ref{symb} because 
the restriction of this particular $\s^\cA $ to $\cB $ is exactly
the Berezin contravariant symbol map, and the 
restriction to $\cA $ of the particular
 $\s^\cB $ that we will define just below
is exactly the Berezin covariant symbol map, 
that play an essential role in \cite{R7}.

For our applications we define $\s^\cB $ by
\begin{equation}
\label{syb}
\s^\cB(F) = d_\cH\int_{G/ K} F(x)\a_x(P)) \ dx,
\end{equation}
where $d_\cH $ is the dimension of the Hilbert 
space of the irrep, 
and $dx $ refers to the $G $-invariant probability measure 
on $ G/ K $. It is easily seen \cite{R29} that
$\s^\cA $ and $\s^\cB $ are unital and completely positive,
and furthermore that they intertwine the actions of
$G $ on $\cA$ and $\cB$ with the diagonal action of $G$ on 
$\cD =\cA \otimes \cB $.

\begin{definition}
\label{sycomp}
Let $\cA$ and $\cB $ be unital C*-algebras and 
and let
$\Pi = (\cD, \om)$ be a bridge from $\cA$ to $\cB$. 
Let
$L^\cA$ and $L^\cB$ be C*-metrics on $\cA$ and $\cB$. 
Let $(\s^\cA, \s^\cB) $ be a pair of symbols for $\Pi$. 
We say that this pair of symbols is \emph{compatible} with the
C*-metrics $L^\cA$ and $L^\cB$ if their restrictions to $\cB$
and $\cA$ satisfy
\[
L^\cA(\s^\cA(b)) \leq L^\cB(b) \quad \quad \mathrm{and} 
\quad \quad L^\cB(\s^\cB(a)) \leq L^\cA(a)
\]
for all $ a\in\cA $ and $b\in\cB $. (Here we are viewing
$ \cA $ and $\cB $ as subalgebras of $\cD$ in the evident way
discussed in the second paragraph of this section.)
\end{definition}

We now show how to use a compatible pair of symbols to 
obtain an upper bound for the reach of
a bridge $\Pi$.
Let $b \in \cL^1_\cB$ be given. As an approximation to $\om b$ by
an element of the form $a\om$ for some $a \in \cL^1_\cA$ we take
$a = \s^\cA(b)$. It is indeed in $\cL^1_\cA$ by the 
compatibility condition. This prompts us to set
\begin{equation}
\label{gamb1}
\g^\cB = 
\sup\{\|\s^\cA(b)\om - \om b\|_\cD :  b \in \cL^1_\cB \}  .
\end{equation}
Interchanging the roles of $\cA$ and $\cB$, we define $\g^\cA$
similarly. We then see that
\begin{equation}
\label{inreach}
\mathrm{reach}(\Pi) \leq \max\{\g^\cA, \g^\cB\}  .
\end{equation}
We will see shortly 
why this upper bound is useful for our applications.

We now consider the height of $\Pi$. For this we need to consider $S_1(\om)$
as defined above. Because 
$\s^\cA$ is positive and unital, its
composition with any $\mu \in S(\cA)$ is in $S(\cD)$. 
By definition $\s^\cA(\om) = 1_\cA$. Thus for every $\mu \in S(\cA)$
we obtain an element, $\phi_\mu$, of $S_1(\om)$, defined by
\[
\phi_\mu(d)  = \mu(\s^\cA(d))  
\]  
for all $d\in D $. This provides us with a substantial collection of elements
of $S_1(\om)$.
Since to estimate the height of $\Pi$ we need to
estimate the distance from each $\mu \in S(\cA)$ to
$S_1^\cA(\om)$, we can hope that $\phi_\mu$ restricted
to $\cA$ is relatively close to $\mu$. Accordingly, for any
$a \in \cA$ we
compute
\[
|\mu(a) - \phi_\mu(a)| = |\mu(a) - \mu(\s^\cA(a))|
\leq \|a - \s^\cA(a)\|  . 
\]
This prompts us to set 
\begin{equation}
\label{delD}
\d^\cA = \sup \{\|a - \s^\cA(a)\| : a \in \cL^1_\cA\}.
\end{equation}
Then we see that
\[
\rho_{L^\cA}(\mu, \ \phi_\mu |_\cA) \leq \d^\cA  
\]
for all $\mu \in S(\cA)$.
We define $\d^\cB$ in the same way, and obtain the
corresponding estimate for the distances from elements
of $S(\cB)$ to the restriction of $S_1(\om)$ to $\cB$. 
In this way we see that
\begin{equation}
\label{eqheight}
\mathrm{height}(\Pi) \leq \max\{\d^\cA, \d^\cB\}  .
\end{equation}
(Notice that $\d^\cA$ involves what $\s^\cA$ does on $\cA$,
whereas $\g^\cA$ involves what $\s^\cB$ does on $\cA$.)

While this bound is natural within this context, it turns out not to be
so useful for our main applications. The following steps might 
not initially seem to be natural, but in the next section
we will
see that for our main applications they are quite 
useful. Our
notation is as above. For any $\nu \in S(\cB)$ we easily see that 
$\nu \circ \s^\cB \circ \s^\cA \in S(D)$. 
But from the relation between the symbols and
$\om $ it is easily seen that 
$\nu \circ \s^\cB \circ \s^\cA(\om) = 1 $, so that 
$\nu \circ \s^\cB \circ \s^\cA $ is in  $S_1(\om) $.
Let us
denote its restriction to $\cB$ by $\psi_\nu$, so that
$\psi_\nu\in S_1^\cB(\om)$. Then $\psi_\nu$ can be used as an
approximation to $\nu$ by an element of $S_1^\cB(\om)$. Now
for any $b \in \cB$ we have
\[
|\nu(b) - \psi_\nu(b)| = 
|\nu(b) - (\nu \circ \s^\cB \circ \s^\cA)(b)| 
\leq \|b - \s^\cB(\s^\cA(b))\|  .
\]
\begin{notation}
\label{bert}
In terms of the above notation we set
\[
\hat\d^\cB = \sup \{\|b - \s^\cB(\s^\cA(b))\| : b \in \cL^1_\cB\}.
\]
\end{notation}
It follows that
$
\rho_{L^\cB}(\nu, \ \psi_\nu ) \leq \hat \d^\cB  
$
for all $\nu \in S(\cB)$, so that
\begin{equation}
\label{hausB}
\mathrm{Haus}_{\rho_\cB}(S_1^\cB(\om) , S(\cB))\}  \leq  \hat\d^\cB.
\end{equation}
We define $\hat \d^\cA$ in the same way, and obtain the
corresponding bound for the distances from elements 
of $S(\cA)$ to $S_1^\cA(\om)$. 
In this way we obtain:
\begin{proposition}
\label{alth}
For notation as above,
\[
\mathrm{height}(\Pi) \leq 
\max\{\min\{ \d^\cA, \hat \d^\cA\}, \min\{\d^\cB, \hat \d^\cB\}\}  .
\]
\end{proposition}


\section{The proof that matrix algebras converge to coadjoint orbits}
\label{qgh}

We continue with the notation of the previous sections, so $G $ is a 
connected compact semisimple Lie group.
Not every $\mu\in\fg'$ is associated with an irrep of $G $.
To be associated with an irrep, the restriction of $\mu $ to any Cartan
subalgebra $\fh $ of $\fg $ must exponentiate to a one-dimensional representation of the corresponding maximal torus, and so in particular must be ``integral". 
But then it can correspond to a weight of many different irreps. To be 
viewed as corresponding to a unique (equivalence class of an) irrep, 
$\mu $ must be ``dominant"
with respect to some choice of Cartan subalgebra and positive root system.

So let us fix an irrep $(\cH, U) $ of $ G $. 
Fix a choice of a maximal torus in $G $ with Cartan subalgebra $\fh$,
and a choice of a positive root system $\D^+$ for it. Then choose elements
$H_\a, E_\a, F_\a$ in $\cfg$ as described in Proposition \ref{HEF}.
Let $\xi_\dia $ be a highest-weight vector of length 1 for the irrep.
As a weight vector it is an eigenvector of $U_H$  for all 
$H \in \fh$. The fact that it is a highest-weight vector means
exactly that $U_{E_\a}\xi_\dia = 0$ for all $\a \in \D^+$. Define $\md$ in
$\fg'$ by
\[
\md(X) = -i\<\xi_\dia, U_X\xi_\dia\>     .
\]
(We take the inner product on $\cH$ to be linear in the
second variable.)
Up to the sign, $\md$ is exactly the ``equivariant momentum map'' of
equation 23 of \cite{Ln2}.
Because $U_X$ is skew-symmetric for all $X \in \fg$, we
see that $\md$ is $\bR$-valued on $\fg$. 
Note that $\md$ does not depend on the phase of $\xi_\dia$.
 We extend $\md$ to $\cfg$ and $\cfh$ in
the usual way. Because $\xi_\dia$ is a highest-weight vector, 
we clearly have
$\md(E_\a) = 0$ for all $\a \in \D^+$, and then $\md(F_\a) = 0$ for 
all $\a \in \D^+$ because $F_\a = E_\a^*$. Furthermore, because
$[E_\a, E_a^*] = H_\a$ and
$[H_\a, E_\a] = 2E_\a$ and $[H_\a, F_\a] = -2F_\a$, the
triplet $(H_\a, E_\a, F_\a)$ generates via $U$ a representation 
of $sl(2,\bC)$, 
for which the spectrum of $U_{H_\a}$ must consist
of integers. In particular, $i\md(H_\a)$ is an integer, necessarily
non-negative, in fact equal to $\|F_\a \xi_\dia \|^2$. 
We see in this way that $\md$ is a quite special
element of $\fg'$.

Let $\mu$ denote the weight of $\xi_\dia$, so that 
$U_H(\xi_\dia) = i\mu(H)\xi_\dia$ for all $H \in \fh$. 
Comparison with the definition of $\md$ shows that $\mu$ is
simply the restriction of $\md$ to $\fh$.  It is clear
that $\md$ is determined by $\mu$ in the sense that $\md$
has value $0$ on the $\Kil$-orthogonal complement of $\cfh$.
Thus from now on
we will let $\md$ also denote the weight of $\xi_\dia$.
(Thus the special properties of $\md$ mean that,
as a weight, $\md$ is a ``dominant integral weight''.) 

We now construct the sequence of matrix algebras that we 
will show converges to 
the coadjoint orbit of $\md$.
For each positive integer $m$ 
let $(\cH^m, U^m)$ be an irreducible representation of $G$
with highest weight $m\mu_\dia$. 
All the $m\mu_\dia$'s will have
the same $\Cad$-stability group, $ K$.
Then let $\cB^m = \cL(\cH^m)$
with action $\a$ of $G$ using $U^m$, and let $P^m$ be the 
projection on the
highest-weight vector in $\cH^m$ (which is just the tensor product
of $m$ copies of $\xi_\dia$).  
As before, we let
$\cA = C(G/ K)$. Then for each $m$ we construct
as for Definition \ref{mybrg} the bridge 
$\Pi_m = (\cD^m, \om^m)$ from $\cA $ to $\cB^m $, using $ P^m $.

We assume that a $\Cad $-invariant inner product has been 
chosen for $\fg'$. Let $ D^\cA $ be the corresponding 
Dirac operator for $\cA $ constructed for Definition \ref{defdir3}.
For each integer $ m $ let $ D^m $ ($= D^{\cB^m} $)  be the corresponding 
Dirac operator for $\cB^m $ constructed as for Definition \ref{matdir}.
Let $ L^{D^\cA} $ be the C*-metric corresponding to $ D^\cA $, 
and for each $m $ let $ L^{D^m} $ be the C*-metric corresponding 
to $ D^m $.

For each $m $ we want to measure the bridge $\Pi_m $ using
the C*-metrics $ L^{D^\cA} $ and $ L^{D^m} $. 
For this purpose, we want to show that the pair of symbols 
$(\s_m^\cA, \s^{\cB^m}) $,
defined as in equations \eqref{sya} and \eqref{syb} using $ P^m $
(and restricted to $\cB^m$ and $\cA$ respectively, as before), 
is compatible with these C*-metrics, in the sense of 
Definition \ref{sycomp}.
But this is awkward
to do directly because $ D^\cA $ and $ D^m $ have been 
constructed in somewhat different ways. However, we were careful 
to show in Corollary \ref{cormt2} that  $ L^{D^\cA} =   L^{D_o^\cA} $,
and in Proposition  \ref{deq}   that  $ L^{D^m} =   L^{D_o^m} $,
where $D_o^\cA $ and $D_o^m $ are defined as in Definition
\ref{gendir}. Thus it suffices to show
that the pair of symbols 
$(\s_m^\cA, \s^{\cB^m}) $ is compatible with the C*-metrics
$ L^{D_o^\cA} $ and $ L^{D_o^m} $. But these C*-metrics
were both defined by the same construction, that of Section \ref{sec1},
so we can apply the results of that section. Because the restriction of
$\s_m^\cA $ to $\cB^m $ is unital completely positive and intertwines
the actions of $ G $ on $\cA $ and $\cB^m $, it follows from
Corollary \ref{ucplip} that for any $b\in \cB^m $ we have
\[
 L^{D_o^\cA}(\s_m^\cA(b)) \leq  L^{D_o^m}(b)   .
\]
In the same way but using $\s^{\cB^m} $, we find that  for 
any $a\in \cA $ we have
\[
 L^{D_o^m}(\s^{\cB^m}(a)) \leq L^{D_o^\cA}(a)   .
\]
Since $ L^{D^\cA} =   L^{D_o^\cA} $, and  
$ L^{D^m} =   L^{D_o^m} $ for each $m $, we see
that we have obtained:

\begin{proposition}
\label{compsym}

With notation as above, for each $m $ the pair of symbols 
$(\s_m^\cA, \s^{\cB^m}) $ is compatible with
the C*-metrics $ L^{D^\cA} $ and $ L^{D^m} $. 
\end{proposition}

Thus we can use the pair of symbols 
$(\s_m^\cA, \s^{\cB^m}) $ in conjunction with
$ L^{D^\cA} $ and $ L^{D^m} $ to 
measure the bridge $\Pi_m $. So in terms of them, define the constants 
$\g^{D^\cA}_m$ and $\g^{D^m}$ 
by Equation \eqref{gamb1}, and
$\hat\d^{D^\cA}_m$ and $\hat\d^{D^m}$ by Notation \ref{bert}.
Then by Equation \ref{reach} and Proposition \ref{alth} we have
\[
\mathrm{reach_D}(\Pi_m) \leq \max\{\g^{D^\cA}_m, \g^{D^m}\} 
\quad \mathrm{and} \quad
\mathrm{height_D}(\Pi_m) \leq \max\{\hat\d^{D^\cA}_m, \hat\d^{D^m}\} 
\]
(We will not use $\d^{D^\cA}_m $, and $\d^{D^m} $.)
The subscript D here on ``$\mathrm{reach_ D} $'' and ``$\mathrm{height_ D} $'' 
is to indicate that here we use the Dirac operators,
in contrast to another bridge-length that we are about to introduce.
We thus obtain:

\begin{proposition}
\label{length}
For notation as above, for each $m $ we have
\[
\mathrm{length_ D}(\Pi_m) \leq 
\max\{\g^{D^\cA}_m, \g^{D^m}, \hat\d^{D^\cA}_m, \hat\d^{D^m}\} .
\]
\end{proposition}

Our objective now is to maneuver so as to be able to take 
advantage of the bounds
obtained in \cite{R6, R29} for the case in which the
C*-metrics are defined in terms of continuous length-functions
$\ell $ on $G $.
The proof of those bounds is fairly complicated. But appealing
to them permits
us to avoid needing to give a quite complicated proof here.

Recall the definition of the C*-metric $ L_ \ell  $ defined by 
Equation \eqref{defell}, where now our chosen $\ell $ is defined in terms of
the Riemannian metric on $ G $ corresponding to the chosen 
inner product on $\fg'$. When we apply it to $\cA = C(G/ K) $ and to
$\cB^m =\cB(\cH_m) $ we will denote it by $ L_ \ell^\cA $ and $ L_ \ell^m $
respectively.
For each $m $ the pair of symbols 
$(\s_m^\cA, \s^{\cB^m}) $ is also compatible with
the C*-metrics  $ L_ \ell^\cA $ and $ L_ \ell^m $. 
This is already proposition 1.1 of \cite{R6} together with an argument
in the paragraph before notation 2.1 of \cite{R6}, 
but it is also easily checked directly. Thus we can use 
$(\s_m^\cA, \s^{\cB^m}) $ in conjunction with
 $ L_ \ell^\cA $ and $ L_ \ell^m $ to 
measure the bridge $\Pi_m $. 
 
Recall the general definition 
\[
\cL_\cA^1 = \{a \in \cA:  a = a^* \quad \mathrm{and} \quad L^\cA(a) \leq 1\}   
\]
given in Definition \ref{reach}. When this definition of $\cL_\cA^1 $ is applied to the
C*-metrics $ L^{D^\cA} $ and $L_ \ell^\cA$ we will denote it by
$\cL_{D^\cA}^1 $ and $\cL_{\ell\cA} ^1$ respectively, while when it is applied 
to the C*-metrics $  L^{D^m} $ and $L_ \ell^m$
we will denote it by  $\cL_{D^m} ^1$ and
$\cL_{\ell m}^1 $ respectively. 
From the fact that $ L^{D^\cA} =   L^{D_o^\cA} $ and 
$ L^{D^m} =   L^{D_o^m} $, it follows from Corollary \ref{metcor} that 
\begin{equation}
\label{inell}
L_ \ell^\cA(a) \leq L^{D^\cA}(a) \quad \mathrm{  and}  \quad 
 L_ \ell^m(b) \leq  L^{D^m}(b) 
\end{equation} 
for all $a\in \cA $ and $b\in \cB^m $.
From the inequalities \eqref{inell} it follows immediately that
\begin{equation}
\label{inball}
\cL_{D^\cA}^1 \subseteq  \cL_{\ell\cA} ^1 \quad \mathrm{and}
\quad \cL_{D^m} ^1 \subseteq      \cL_{\ell m}^1   .
\end{equation}

Recall from Equation \eqref{gamb1} that the general definition of $\g^\cB $ 
is given by
\[
\g^\cB = 
\sup\{\|\s^\cA(b)\om - \om b\|_\cD :  b \in \cL^1_\cB \}  .
\]
When this definition is applied using $\s_m^\cA $, $\om^m $, and
$\cL_{D^m} ^1 $ or $   \cL_{\ell m}^1 $,
we will denote it by
$\g^{D^m} $ (as above) or $\g_\ell^m $ respectively; 
while when this definition is applied using
$\s^{\cB^m} $, $\om^m $, and
$\cL_{D^\cA}^1 $ or $ \cL_{\ell\cA} ^1 $
we will denote it by
$\g^{D^\cA}_m $  (as above) or $\g_{\ell m}^\cA  $ respectively. 
From the containments of display
\eqref{inball} it follows immediately that
\begin{equation}
\label{ingam}
\g^{D^\cA}_m \leq \g_ {\ell m}^\cA  \quad \mathrm{and} \quad
\g^{D^m} \leq \g_\ell^m   . 
\end{equation}
From this and inequality \eqref{inreach} we immediately obtain:

\begin{proposition}
\label{compreach}
For all $m $ we have $\mathrm{reach_ D}(\Pi_m)\leq 
\max\{ \g_ {\ell m}^\cA , \g_ \ell^m\}$.
\end{proposition}

But the paragraph just before proposition 6.3 of \cite{R29} 
(where the $\g^\cA_m $
and $\g^{\cB^m} $ there are our $\g_{\ell m}^\cA $ and $ \g_ \ell^m $)
explains how the results in sections 10 and 12 of \cite{R21}
prove that as $m $ goes to $ \infty $ both sequences
$\{\g_{\ell m}^\cA \}$ and $\{ \g_ \ell^m \}$ converge to 0.
We thus obtain:

 \begin{proposition}
\label{reachlim}
The sequence $\{\mathrm{reach_ D}(\Pi_m)\}$ converges to 0
as $m $ goes to $ \infty $.
\end{proposition}
 
We now turn to considering the height of $ \Pi_m $.
For any $\mu \in S(\cA) $ set $\phi_ \mu =\mu\otimes\tau^m $ where
$\tau^m $ is the tracial state on $ \cB^m$. Then $\phi_\mu \in S(\cD^m) $, 
and simple computations show that $\phi_\mu(\om^m) = 1 $ so that
$\phi_\mu \in S_1(\om^m) $. It is easy to see that 
$\phi_\mu |\cA =\mu $. In this way we see that $ S_1^\cA(\om^m ) = S(\cA) $. 
Thus $\mathrm{Haus}_{\rho_\cA}(S_1^\cA(\om), S(\cA)) = 0 $,
regardless of what the metric $\rho_\cA $ is. It follows that
\[
\mathrm{height_ D}(\Pi_m) =
\mathrm{Haus}_{\rho_{D^m}}(S_1^{\cB^m}(\om) , S(\cB^m))
\]
where $\rho_{D^m} $ is the metric on $S(\cB^m) $ determined by
the Dirac operator $ D^m $.

Recall from Notation \ref{bert} that the general definition of $\hat\d^\cB$ 
is given by
\[
\hat\d^\cB = \sup \{\|b - \s^\cB(\s^\cA(b))\| : b \in \cL^1_\cB\}.
\] 
When this definition is applied using $\s^{\cB^m}  \circ\s_m^\cA $, and
$\cL_{D^m} ^1 $ or $   \cL_{\ell m}^1 $,
we will denote it by
$\hat\d^{D^m} $ (as above) or $\hat\d_\ell^m $ respectively.
From equation \eqref{hausB} we see that
\[
\mathrm{Haus}_{\rho_{D^m}}(S_1^{\cB^m}(\om) , S(\cB^m))\}  \leq  \hat\d^{D^m}.
 \] 
But from the containments of display
\eqref{inball} it follows immediately that
\begin{equation*}
\label{indel}
\hat\d^{D^m}  \leq \hat\d_\ell^m    . 
\end{equation*}
From this and inequality \eqref{alth} we immediately obtain:

\begin{proposition}
\label{compheight}
For all $m $ we have $\mathrm{height_ D}(\Pi_m)\leq  \hat\d_\ell^m   $ .
\end{proposition}

But the paragraph just before proposition 6.7 of \cite{R29} 
(where the $\hat \d^{\cB^m} $ there is our $\hat \d_ \ell^m $)
explains how theorem 11.5 of \cite{R21} gives a
proof that as $m $ goes to $ \infty $ the sequence
 $\{\hat \d_ \ell^m \}$ converge to 0.
We thus obtain:

 \begin{proposition}
\label{heightlim}
The sequence $\{\mathrm{height_ D}(\Pi_m)\}$ converges to 0
as $m $ goes to $ \infty $.
\end{proposition}
 
Combining this with Proposition \ref{reachlim}, we obtain 
the main theorem of this paper: 

\begin{theorem}
\label{lengthlim}
The sequence $\{\mathrm{length_ D}(\Pi_m)\}$ converges to 0
as $m $ goes to $ \infty $. Thus the sequence $\{(\cB^m, L^{D^m})\} $
of compact C*-metric spaces converges to the 
compact C*-metric space $\{(\cA, L^{D^\cA})\} $
for \Lat 's propinquity.
\end{theorem}
 

\section{The linking Dirac operator}
\label{liD}

In this section we offer some generalities that place bridges 
in a more Dirac-operator-like setting.

Let  $\cA$ and $\cB$ be unital C*-algebras, and let $\Pi =(\cD, \om) $ be a bridge from
 $\cA$ to $\cB$ (so we view $\cA$ and $\cB$ as unital subalgebras of $\cD $). 
 Let $\cE $ be a unital C*-algebra that contains $\cD $ as a unital C*-subalgebra. 
 Then $(\cE,\om) $ is equally well a bridge from  $\cA$ to $\cB$. 
 The corresponding bridge-norm 
 $ N_\Pi $ on  $\cA \oplus \cB$ defined by 
\begin{equation}
\label{bnorm}
 N_\Pi(a,b) =\|a\om - \om b\| 
 \end{equation}
is the same regardless of whether we use the bridge $(\cD, \om) $ 
or the bridge $(\cE,\om) $. And every state in $S_1(\om) $ for $ \cD $ 
has an extension (perhaps not unique) 
to a state on $\cE $,  which is then clearly in the  $S_1(\om) $ 
 for $ \cE $. From these observations, it is easily seen that for any given 
 C*-metrics $L^\cA$ and $L^\cB$ on $\cA$ and $\cB$, 
 the reach and the height of the bridges $(\cD,\om) $ and $(\cE,\om) $ 
 are the same. Thus these two bridges have the same length.
 
Suppose now that $(\cH,\pi) $ is a faithful (non-degenerate) 
representation of $\cD $. We can then view $\cD $ as a unital subalgebra 
of $\cB(\cH) $. The above comments then apply, and $(\cB(\cH), \om) $ 
is a bridge from $\cA$ to $\cB$ that has the same length as the bridge
$(\cD,\om) $.

\begin{definition}
\label{hilbd}
Let  $\cA$ and $\cB$ be unital C*-algebras. 
By a \emph{Hilbert bridge}  from $\cA$ to $\cB$ we mean a quadruple 
$(\cH, \pi^\cA, \pi^\cB, \om) $, where $\cH $ is a Hilbert space, 
$\pi^\cA $ and $\pi^\cB $ are faithful (non-degenerate) representations of  
$\cA$ and $\cB$ on $\cH $, and $\om $ is a self-adjoint operator on $\cH $ 
that has 1 in its spectrum. We will often view $\cA$ and $\cB$ as 
unital C*-subalgebras
of $\cB(\cH) $, and omit $\pi^\cA $ and $\pi^\cB $ in our notation.
\end{definition}

As mentioned after Definition \ref{height},
\Lat \ defines his propinquity in terms of finite paths (which he calls "treks") 
of bridges between C*-metric spaces (See \cite{Ltr2}.) 
He defines the length of a trek to be the sum of the lengths of the bridges in the trek,
and he defines the propinquity from $\cA$ to $\cB$ to be the infimum of the lengths
of all treks from $\cA$ to $\cB$. Actually, \Lat \ does not require that the pivot
$\om $ be self-adjoint. I am unaware of an example where this makes a difference, 
but since the propinquity is defined as an infimum, the propinquity requiring pivots
to be self-adjoint will be greater than or equal to the propinquity not requiring that.

From the comments made just before Definition \ref{hilbd}, 
it is clear that if one insists on defining the propinquity using only Hilbert bridges, 
it will never-the-less coincide with the usual propinquity (using self-adjoint pivots).

Suppose now that $\Pi = (\cH, \om) $ is a Hilbert bridge from $\cA$ to $\cB$.
Let $\cH_o = \cH \oplus \cH$. For $\cA$ and $\cB$ viewed as
subalgebras of $\cB(\cH)$, define a representation of $\cA \oplus \cB$
on $\cH_o$ by 
\begin{equation*}
(a,b) \to 
\begin{pmatrix}
a & 0    \\
0 & b
\end{pmatrix}   ,
\end{equation*}
and define a bounded operator, $D_\om$, on $\cH_o$ by
\begin{equation*}
D_\om =
\begin{pmatrix}
0 & \om    \\
\om & 0
\end{pmatrix}   .
\end{equation*}
Then
\[
[D_\om, (a,b)] = \begin{pmatrix}
0 & \om b - a\om   \\
\om a - b\om & 0
\end{pmatrix}    , 
\]
so that
\[
\| [D_\om, (a,b)] \| = \|a\om -\om b\| \vee \|a^*\om - \om b^*\|   ,
\]
where $\vee$ means ``maximum''. We have used here the fact that 
$\om $ is self-adjoint.
Notice that the bridge seminorm $N_\Pi$ (equation \eqref{bnorm})
of $\Pi $ is in general not a $*$-seminorm. 
Much as in theorem 6.2 of \cite{R21}, define a $*$-seminorm, $\hat N_\Pi$,
on $\cA \oplus \cB$ by
\[
\hat N_\Pi(a, b) = N_\Pi(a,b) \vee N_\Pi(a^*, b^*)   .
\]
Then we see that
\[
\| [D_\om, (a,b)] \| = \hat N_\Pi(a, b)    .
\]
Of course, $\hat N_\Pi$
agrees with $N_\Pi$ on self-adjoint elements.

Let $(\cH^\cA, D^\cA)$ and $(\cH^\cB, D^\cB)$ be Dirac
operators for $\cA$ and $\cB$, so that there are nice dense subalgebras  
$\cA^\infty$ and $\cB^\infty$ whose elements have bounded 
commutators with the Dirac operators. Let  $L^\cA$ and 
$L^\cB$ be the corresponding C*-metrics.
Then for any $r\in \bR^+ $
we construct a Dirac-like
operator, $D_r$,  for $\cA \oplus \cB$ on 
\[
\cH^\cA \oplus \cH_o \oplus \cH^\cB
\]
given by 
\[
D_r = D^\cA \oplus  r^{-1}D_\om \oplus D^\cB.
\]
We let $\cA \oplus \cB$ act on $\cH^\cA \oplus \cH_o \oplus \cH^\cB $
in the evident way. Then for $a \in\cA^\infty$ 
and $b \in \cB^\infty$ we have
\[
[D_r, (a,b)] = [D^\cA, a] \oplus 
r^{-1}[D_\om, (a,b)] 
\oplus  [D^\cB, b]    .
\]
Thus
\[
\| [D_r, (a,b)]  \| = L^\cA(a) \vee r^{-1}  \hat N_\Pi(a, b)  \vee   L^\cB(b)    .
\]
It is not difficult to check that $D_r$ determines a 
C*-metric, $L_r $, on $\cA \oplus \cB$, and that $L_r$ 
is compatible with $L^\cA $
and $L^\cB $ (in the sense 
that its quotients on $\cA^\infty$ and $\cB^\infty$ coincide with $L^\cA $
and $L^\cB $, as discussed around definition 4.7 of \cite{R32})
if and only if $r \geq r_\Pi$, where $r_\Pi$ is the reach
of the bridge. (This is proposition 4.8 of \cite{R32}.)
Thus $D_r$ can be used to bound 
the quantum Gromov-Hausdorff distance
(as defined in \cite{R6})
between the C*-metric spaces $(\cA, L^\cA)$ and $\cB, L^\cB)$.
In particular, $r_\Pi$ can be defined as the smallest $r $ such that
$ D_r $ is compatible with $L^\cA $ and $L^\cB $. Thus we can determine
the reach of the bridge $\Pi $ by considering the $ D_r $'s. But I do not see 
a way of determining the height of $\Pi $ directly in terms of the $ D_r $'s. 
For that purpose one seems to need to remember $\om$ and the
subspace $\cH_o$.


\section{Deficiencies}
\label{secdefit}

We have given above a somewhat unified construction of Dirac operators 
for matrix algebras and homogeneous spaces, in such a way that when 
the homogeneous space is the coadjoint  orbit of an integral weight vector, 
the corresponding sequence of matrix algebras converges to the coadjoint 
orbit for \Lat's \ propinquity. Heuristically, the matrix algebras, 
with metric shape given by their Dirac operators, converge for a suitable 
quantum Gromov-Hausdorff distance to the coadjoint orbit with metric 
shape given by its Dirac operator

There are several deficiencies with this picture, which we discuss here. 
The first deficiency is that Dirac operators for Riemannian metrics are usually 
defined by means of 
the Levi--Civita connection, which is the unique torsion-free connection 
compatible with the Riemannian metric. But our construction does not always 
yield that Dirac operator. 
In essence, we have been using the dual of the ``canonical connection" 
on the tangent bundle. As is well-known, and explained in section 6 of \cite{R22}, 
for homogeneous spaces
the Levi--Civita connection agrees with the canonical connection exactly for 
the symmetric spaces. Many coadjoint orbits are not symmetric spaces, 
and so for those our approach does not give the usual Dirac operator. 
But our approach does give the usual Dirac operator for those coadjoint 
orbits that are symmetric spaces, which include the 2-sphere and complex 
projective spaces. For non-commutative algebras such as our matrix algebras 
it is far from clear how one might define ``torsion-free" for connections. (But
see definition 8.8 of \cite{GVF}.)

However, we saw in Theorem \ref{thnorm2} that for the Dirac operators we constructed we have
\[
\|[D, M_f]\| 
 = \sup \{\| \rd f_x\|_{\fm'^\bC}: x \in G\}
\]
for any $f \in \ciA = C^\infty(G/K)$. Whereas on the last page of \cite{R22}
we saw that the same result was obtained when using the Levi--Civita connection
(though phrased in terms of the tangent bundle instead of the cotangent bundle, so
$\| \rd f_x\|_{\fm'^\bC} =  \|\grad_f(x)\|$).
In \cite{R22} only the Hodge-Dirac operator was considered, that is, the spinor 
bundle was assumed to be $\bC\ell(\fm) $ itself. But by arguments elaborating 
on the proof of Proposition \ref{prodnorm} one can see that $\|[D, M_f]\| $ is 
independent of the choice of spinor bundle. The same formula for $\|[D, M_f]\| $ is 
discussed in the comments following proposition 5.10 of \cite{R27} specifically 
for the case of homogeneous spaces that can have a $G$-invariant almost complex 
structure, which as we saw includes the case of coadjoint orbits.
Now a standard argument (e.g., following definition~9.13 of \cite{GVF}) shows that if 
we denote by $\rho$ the ordinary metric on a Riemannian manifold $N$ coming from 
its Riemannian metric, then for any two points $p$ and $q$ of $N$ we have
\[
\rho(p,q) = \sup\{|f(p)-f(q)|: \|\grad_f\|_\infty \le 1\}.
\]
On applying this to either the Dirac operator using the canonical connection, 
or the Dirac operator using the Levi--Civita connection
 we obtain
 \begin{equation}
 \label{conmet}
\rho(p,q) = \sup\{|f(p)-f(q)|: \|[D,M_f]\| \le 1\}
\end{equation}
for $\rho$ now the 
ordinary metric on $G/K$ from our Riemannian metric. Thus as far as the 
metric aspects are concerned, these two Dirac operators are equivalent, and 
we can recover the ordinary metric from either one. 
 Consequently, the matrix algebras equipped with their Dirac operators 
 from our construction do still converge to the coadjoint orbit equipped 
 with the Dirac operator for the  Levi--Civita connection.
Formula \eqref{conmet}
is the formula on which Connes focused for general Riemannian 
manifolds \cite{Cn7,Cn3, CnMr}, as it shows that the Dirac operator contains all of 
the metric information (and much more) for the manifold.  This is his motivation 
for advocating that metric data for ``non-commutative spaces'' be encoded by 
providing them with a ``Dirac operator''. 

A second deficiency of our somewhat unified approach is that coadjoint orbits 
actually carry a $G$-invariant K\"ahler structure, which includes not just a complex 
structure $J$, but also closely related symplectic and Riemannian structures. 
In the paragraph of \cite{BFR} that contains equations 3.53 through 3.59 a proof is 
given that the inner product on $\fm $ (or equivalently on $\fm'$) corresponding to 
the Riemannian metric of the K\"ahler structure can not 
be extended to a $G $-invariant 
inner product on $\fg $ unless
the coadjoint orbit is a symmetric space. Since the inner products on $\fm'$ that 
we have been using in our approach have always been restrictions of $ G $-invariant 
inner products on $ \fg'$, the Dirac operators we have constructed can 
not be the ones for the Riemannian metric of the
K\"ahler structure except if the coadjoint orbit is a symmetric space. 
(The results in \cite{Kst} do not apply to the  K\"ahler Riemannian metrics for the same
reason, as seen from condition (b) following equation 1.3 of \cite{Kst}.  
Dirac operators for all the K\"ahler Riemannian metrics are constructed in \cite{R27}). 
In particular, if the coadjoint orbit
corresponding to a highest-weight vector is not a symmetric space, 
then the sequence of matrix algebras equipped with the Dirac operators that we 
have constructed will not converge to the coadjoint orbit when that orbit is equipped 
with the Riemannian metric of the K\"ahler structure. This follows from \Lat's \ remarkable 
theorem \cite{Ltr2} that if the propinquity between two C*-metric spaces is 0 then 
they are isometrically isomorphic.

Thus our general approach works quite well for coadjoint orbits that are symmetric spaces, 
and somewhat less well otherwise.

At present it is far from clear to me how one might modify our general approach in 
a way that would correct either of these two deficiencies.


\section{Comparisons with Dirac operators in the literature}
\label{compare}

At the end of Section \ref{fuzzy} we already discussed for the case of 
the sphere the relations between our construction of Dirac
operators and various proposals in the literature. Here we briefly discuss
for other coadjoint orbits the relations between our construction and 
several proposals in the literature.

In \cite{DHMO} Dirac operators are constructed on matrix algebras
related to complex projective spaces, using a Schwinger-Fock construction. 
The methods are quite different
from those of the present paper, and our discussion for the case of
the sphere given at the end of Section \ref{fuzzy} applies here also.
(Since $\bC P^n$ is not a spin-manifold when $n$ is even, one
does not expect to have a charge-conjugation operator in those cases.)
One of the authors of \cite{DHMO} says in \cite{Hut}  that the ``formulation''
``seems ill suited to couple the fermions with gauge fields''. In the last
section of \cite{Hut} that author briefly introduces a Dirac operator
on the relevant matrix algebras that seems to be more closely related
to that in the present paper. But that Dirac operator is not explored 
there, though some clarification of the notation used there is given
in \cite{Hut2} (though that paper does not mention matrix algebras).

In \cite{AcD} the full flag-manifold for $G = SU(3)$ is considered, so the subgroup 
$K $ is the maximal torus of $G $. Then $ G/K$ is not a symmetric space, 
but it is a spin manifold. For the corresponding matrix algebras, twisted
Dirac operators for many projective modules are constructed. 
The dimension of $\fg'$ is 8, and so the irreducible representation
of $\bC\ell(\fg')$ has dimension 16. That is the dimension
of the spinor spaces used in constructing the Dirac operators,
as it would be for our construction
in the sections above.
The Ginsparg-Wilson method is used, and already in the
untwisted case the formula for the Dirac operator has extra terms
(that could be interpreted as related to curvature) in comparison
to the Dirac operators in our present paper. 
Since $G/K$ is not a
symmetric space, it is very unlikely
that the Dirac operator would relate well to 
the K\"ahler structure. There is no discussion of charge conjugation. 
It would be interesting
to find a natural framework that would lead to Dirac operators
of the kind given in this paper.

Constructing matrix-algebra approximations to spheres of
dimension at least 3 is more complicated since they do not
admit a symplectic structure. But one can use their close
relationship to coadjoint orbits, as shown, for example,
in \cite{MdHIO}


\section{Spectral Propinquity}
\label{spectral}

As in \cite{Ltr7}, we say that a spectral triple $(\cA, \cH, D)$ is 
metric if the seminorm $L^D$ on $\cA$ defined by 
$ L^D(a) = \| [D, a] \|$ is a C*-metric, the issue being whether the 
topology it determines on $ S(\cA) $ coincides with the 
weak-$*$ topology. In definition 4.2 of \cite{Ltr7} 
(and just before theorem 1.21 of \cite{Ltr9})
\Lat \ defines a metric on 
isomorphism classes of spectral triples, which he calls 
the \emph{spectral propinquity}. In this section we will examine 
the spectral propinquity for the spectral triples studied in the 
earlier sections of this paper. We will not give here a complete
explanation of the general spectral propinquity. We will
give some definitions and proofs, but we refer the reader
to \Lat's papers for more details.

\Lat's  spectral \ppq is based on his dual Gromov-Hausdorff 
propinquity \cite{Ltr4},
which is based on ``tunnels" instead of ``bridges". We recall 
its definition here, but in a restricted simpler form that is sufficient
for our needs. We first recall:

\begin{definition}
Let $(\cD, L^\cD)$ and $(\cA, L^\cA)$ be compact C*-metric spaces. 
By a \emph{quantum isometry} from  $(\cD, L^\cD)$ to $(\cA, L^\cA)$
we mean a unital $*$-homomorphism, $\pi$, from $\cD $ onto
$\cA $ such that $L^\cA$ is the quotient seminorm of $L^\cD$,
that is, that for all $a \in \cA$ we have
\[
L^\cA(a) = \inf \{ L^\cD(d): \pi(d) = a\}.
\]
If $\pi$ is actually a $*$-isomorphism such that $L^\cD = L^\cA \circ \pi $,
then we say that $\pi $ is a \emph{full isometry}. 
\end{definition}

The reason for using the term ``isometry'' here is that if
$\pi $ is a quantum isometry then 
the corresponding map $\pi^*$ from $S(\cA)$ into $S (\cD) $
is an isometry for the metrics $\rho^{L^\cA}$ 
and  $\rho^{L^\cD}$,
as already seen in proposition 3.1 of \cite{R6}. If $\pi$
is a full isometry, then $\pi^*$ is an isometry onto
$S (\cD) $ (and conversely).

\begin{definition}
\label{tunnel}
Let $(\cA, L^\cA)$ and $(\cB, L^\cB)$ be compact C*-metric spaces. 
By a \emph{tunnel} from  $(\cA, L^\cA)$ to $(\cB, L^\cB)$
we mean a compact C*-metric space $(\cD, L^\cD)$ together with
quantum isometries $\pi^\cA$ and $\pi^\cB$ from $\cD$
onto $\cA $ and $\cB $ respectively. 
\end{definition}

This definition is quite reminiscent of the set-up for the definition of ordinary
Gromov-Hausdorff distance. For our setting in which $\cA = C(G/K)$
and $\cB^m = \cB(\cH^m)$ the tunnel we use is $(\cD^m, L^{\cD^m})$
where
 \[
\cD^m = \cA \oplus \cB^m 
\]
with its evident projections onto $\cA $ and $\cB $, while $L^{\cD^m}$
is defined by
\begin{equation}
 \label{brm}
L^{\cD^m}(a,b) = L^\cA(a) \vee L^{\cB^m}(b) \vee r^{-1}N(a,b)
\end{equation}
where $N$ is the bridge-norm defined in equation \eqref{bnorm}
using the pivot as defined in equation \eqref{pivot}, and $r$
is large enough to ensure that the evident projections onto $\cA $ 
and $\cB $ are quantum isometries (and $\vee$ means ``max''). 
It is easy to see that this means 
exactly that $r \geq r_\Pi $ where $ r_\Pi $ is the reach of $\Pi$
as defined in Definition \ref{reach}. (We remark that the seminorm
defined by equation \eqref{brm} looks just like the
seminorm in the statement of theorem 5.2 of \cite{R6} for the
``bridges'' as defined in that paper, and used in definition 6.1
and theorem 6.2 of \cite{R21}. So the terminology of
``bridges'' and ``tunnels'' has gotten a bit scrambled.) 

\Lat \ then defines for any tunnel its ``extent'', which is a specific way of defining its length. Let $\cT= (\cD, L^\cD, \pi^\cA, \pi^\cB)$ be a tunnel from $(\cA, L^\cA)$ to 
$(\cB, L^\cB)$ as in Definition \ref{tunnel}. Let $\pi^\cA_*$ denote 
composition of elements of $S(\cA)$ with $\pi^\cA$. Then, as indicated above,
$\pi^\cA_*$, is an isometry from  $S(\cA)$ into $S(\cD)$. In the same way 
we define  $\pi^\cB_*$. The extent of $\cT$,
$\mathrm{ext}(\cT)$, is defined to be
\[
\mathrm{ext}(\cT) = \max\{ \mathrm{dist}_H(\pi^\cA_*(S(\cA)), S(\cD)), \ 
\mathrm{dist}_H(\pi^\cB_*(S(\cB)), S(\cD))\}  ,
\]  
where $\mathrm{dist}_H$ denotes Hausdorff distance with respect to the 
metric on $S(\cD)$ determined by $L^\cD$. \Lat \ then defines the ``dual propinquity'' between $(\cA, L^\cA)$ and $(\cB, L^\cB)$ to be the infimum of
the extents of all tunnels from $(\cA, L^\cA)$ to $(\cB, L^\cB)$ 
(with no need for ``treks'' etc.).
\Lat \ showed in \cite{Ltr4} the remarkable fact that the dual propinquity
is a metric on isometry classes of compact C*-metric spaces, and that
this metric is complete on the space of isometry classes. 
He also shows that the propinquity, which we used
in earlier parts of this paper, dominates the dual propinquity, so our earlier 
convergence results, such as Theorem \ref{lengthlim}, 
imply convergence for the dual propinquity too.
But the propinquity is probably not complete, though there is no 
known counter-example as far as I know.  

\Lat \ defines the spectral propinquity between metric spectral triples 
of the form $(\cA, \cH, D)$
in two steps. In the first step, 
the Hilbert space $\cH $ is viewed as a left module over $\cA $ that 
is equipped with a suitable seminorm, using $D$. 
\Lat \ had earlier \cite{Ltr6} defined 
a metric on isometry classes of such modules, which he 
called the \emph{modular propinquity}. (It uses ``modular tunnels''.)
In the second step 
one considers a covariance condition with respect to the 
one-parameter group of unitary operators generated by $D$.
Here we will just examine the modular propinquity for the spectral triples 
studied in the earlier sections of this paper. 

In \cite{Ltr8} \Lat \ defines the modular propinquity for correspondences
that are equipped with a suitable seminorm. 
We recall \cite{Blk2} that, 
given C*-algebras $\cA$ and $\cC$, 
an $\cA$-$\cC$-correspondence is a right Hilbert-$\cC$-module $\cM$ (i.e.
a right $\cC$-module $\cM$ with a $\cC$-valued inner product) together with 
a $*$-homomorphism of $\cA$ into the C*-algebra of adjointable
$\cC$-endomorphisms of $\cM$.

We will only need the special case in which $\cC = \bC$, for which
$\cM$ is a Hilbert space, and the action of $\cA$ is through a 
$*$-representation on $\cM$. 
In fact we will often use the letter $\cH$ for $\cM$. 
But to align with the terminology for general correspondences, 
we will refer to this special case as a ``special correspondence".
(These should not be confused with the ``metrized quantum vector bundles''
of definition 1.10 of \cite{Ltr9}, which are right Hilbert-$\cC$-modules,
but there is no algebra $\cA$ acting on the left.) We will tend to use
\Lat's notation from \cite{Ltr7}.

Here is the ``special'' version of definition 2.2 of \cite{Ltr7} and
definition 1.10 of \cite{Ltr9}.

\begin{definition} 
We say that $(\cA, L^\cA, \cM, \sfD)$ is 
a \emph{special metrical C*-correspondence} if  
$(\cA, L^\cA)$ is a compact C*-metric 
space, while $\cM$ is a Hilbert space on which there is a 
$*$-representation of $\cA $, and
$\sfD$ is a norm defined on a dense
 subspace $\dom(\sfD)$ of $\cM $, such that
\begin{enumerate}
\item For all $\xi \in \dom(\sfD)$ we have $\|\xi\| \leq \sfD(\xi)$.
\item The subset $\{\xi \in \dom(\sfD):\sfD(\xi) \leq 1\}$ 
of $\cM$ is totally bounded.
\item For all $a \in \cA$ and all $\xi \in \dom(\sfD)$ we have
\[
\sfD(a\xi) \leq (\|a\| + L^\cA(a) ) \sfD(\xi).
\]
(In particular, if $L^\cA(a) < \infty$, then $a\xi \in  \dom(\sfD)$.)
\end{enumerate}
\end{definition} 

For our purposes, the main general example is given by theorem 2.7
of \cite{Ltr7}, which we now state, with our notation and terminology.

\begin{theorem}
\label{spnorm}
Let $(\cA, \cH, D)$ be a metrical spectral triple, with
$L^D(a) = \|[D,a]\|$ the corresponding C*-metric on $\cA$.
Define the norm $\sfD$ on $\dom(D)$ to be the graph-norm of
$D$, that is,
\[
\sfD(\xi) = \|\xi\| + \|D\xi\|.
\]
Then $(\cA, \cH, L^D, \sfD)$ is a special metrical C*-correspondence.
\end{theorem}

Accordingly we 
let $ \sfD^\cA$ and $ \sfD^m$
be the norms on $\cS^\cA$ and $\cS^m$ defined in terms of $D_o^\cA$ and $D_o^m$ much as in Theorem \ref{spnorm}.
We want to 
show that the special metrical C*-correspondences
$(\cB^m, L^{D_o^m},  \cS^m, \sfD^m)$
converge to the special metrical C*-correspondence
$(\cA,  L^{D_o^\cA},\cS^\cA, \sfD^\cA)$
for the modular propinquity.

The definition of the modular propinquity involves a definition of what is meant by
a tunnel between two metrical C*-correspondences. (\Lat \ calls these
``modular tunnels''). We will not give the general definition. Instead we 
now begin constructing the modular tunnels we need. We start with the 
core of the construction. To some extent we follow the steps that \Lat \ 
used to construct 
the corresponding modular tunnels for the case of non-commutative tori 
in section 3.3 of \cite{Ltr9}. We use the notation of Section \ref{sec1} 
as well as that following Definition \ref{tunnel}. 

Recall that $\cS^\cA = \ciA \otimes \cS$ and $\cS^m = \cB^m \otimes \cS$, dense
subspaces of the corresponding Hilbert spaces.
We let 
$\th_m^\cA = \s^\cA_m \otimes I^\cS$ and
$\th^{\cB^m} = \s^{\cB^m}  \otimes I^\cS$, where $ \s^\cA_m$ and
$ \s^{\cB^m} $ are as defined in the paragraph preceding Proposition          \ref{compsym}. So $\th_m^\cA  $ maps $\cS^m$ to $\cS^\cA$, 
while $\th^{\cB^m} $
maps $\cS^\cA$ to $\cS^m$.
Since $ \s^\cA_m$ and
$ \s^{\cB^m} $ intertwine the actions $\a$ of $G $ on $\cA$ (and so on $\ciA$)
and $\cB^m$, as mentioned just before 
Definition \ref{sycomp},  $\th_m^\cA$ and $\th^{\cB^m}$ intertwine
the actions $\a \otimes I^\cS$ of $G$ on $\ciA \otimes \cS$ 
and $\cB^m \otimes \cS$. 
From formula \ref{eqdefdir} it follows that 
$D_o^\cA \th_m^\cA = \th_m^\cA D_o^m$ and
$D_o^m \th^{\cB^m} = \th^{\cB^m}D_o^\cA$.
It is easily seen that $ \s^\cA_m$ has operator norm 1
(because we are essentially using normalized traces on
$\cA$ and $\cB^m$). Since $ \s^{\cB^m} $ is easily seen to be the Hilbert-space
adjoint of $ \s^\cA_m$ (as shown early in section 2 of \cite{R6}),
$\th^{\cB^m}$ is the Hilbert-space
adjoint of $\th_m^\cA$, and they both have operator norm 1 (i.e.
are contractions of norm 1) for the Hilbert-space norms
on  $\cS^\cA$ and $\cS^m$.

The following proposition is the analog of much of theorem 3.25 of \cite{Ltr9}. Notice that it does not explicitly involve the C*-algebras
$\cA$ and $\cB^m$, just as theorem 3.25 does not involve the C*-algebras
$\cA_n$ and $\cA_\infty$ of \cite{Ltr9}. 

\begin{proposition}
\label{tuncor}
With notation as above, let $\cS_t = \cS^\cA \oplus\cS^m$, with its evident
pre-Hilbert-space structure. For any $\e >0$ define 
a norm $\tn$ (for ``tunnel norm'') 
on $\cS_t$ by 
\begin{equation}
\label{bigtun}
\tn(\xi, \eta) = \sfD^\cA(\xi) \vee \sfD^m(\eta) \vee
(1/\e)\|\xi - \theta^{\cA}_m(\eta)\| 
\end{equation}
for $(\xi, \eta) \in \cS_t$. For any $\e > 0$ 
there is a natural number $N$ such that if $m \geq N$ then
$ \sfD^\cA$ and $ \sfD^m$ coincide with the quotient norms of $\tn$ on
$\cS^\cA$ and $\cS^m$ for the evident projections from $\cS_t$
onto $\cS^\cA$ and $\cS^m$.
\end{proposition}

\begin{proof}
Let $\eta \in \cS^m$ be given. Choose $\xi \in \cS^\cA$ to be 
$\xi = \th_m^\cA \eta$, so that the third term in the definition
of $\tn$ is 0. Thus to show that the quotient norm of $\tn$
on $ \cS^m$  is $ \sfD^m$ we only need to show that
$ \sfD^\cA(\xi) \leq \sfD^m(\eta)$. But this follows immediately 
from the facts that $\th_m^\cA$ is of norm 1 and
$D_o^\cA \th_m^\cA = \th_m^\cA D_o^m$ as seen above.
This part of the proof is independent of any choice of $\e$.

Let $\e > 0$ be given, and let $\xi \in \cS^\cA$ be given. 
Choose $\eta \in \cS^m$ to be 
$\eta = \th^{\cB^m} \xi$. Then $\sfD^m(\eta)  \leq  \sfD^\cA(\xi)$ because 
$\th^{\cB^m}$ is of norm 1 and
$D_o^{\cB^m} \th^{\cB^m} = \th^{\cB^m}D_o^\cA$ as seen above.
Thus we need to show that there is an $N$ such that if $m \geq N$
then $(1/\e)\|\xi - \theta^{\cA}_m(\eta)\|  \leq \sfD^\cA(\xi)$, that is,
\begin{equation}
\label{eqtun}
\|\xi - \theta^{\cA}_m( \th^{\cB^m} \xi)\| \leq \e(\|\xi\| + \|D_o^\cA\xi\|). 
\end{equation}
Notice that since $\th_m^\cA$ and $\th^{\cB^m}$ have norm 1, the
left-hand side of equation
\eqref{eqtun} 
is $\leq 2\|\xi\|$. Thus
equation 
\eqref{eqtun} 
will be satisfied if 
\[
2\|\xi\| \leq  \e(\|\xi\| + \|D_o^\cA\xi\|).
\] 
Suppose now that $\xi$ is an
eigenvector for $D_o^\cA$ with eigenvalue $\l$, so that the right-hand
side of equation \eqref{eqtun} is $ \e(1 + |\l|)\|\xi\|$. It is then clear that 
equation  \eqref{eqtun}  is satisfied if $ |\l|  \geq   2/\e $.

In Proposition \ref{selfad} we saw that there is a basis for $\cS^\cA$
consisting of eigenvectors of $D_o^\cA$. As mentioned there, the
eigenvalues of $D_o^\cA$, counted with multiplicity, converge in 
absolute value to $\infty$.  (A somewhat indirect 
proof of this fact is given in theorem 5.5 of \cite{GaG}, whose method
is to compare $D_o^\cA$ to the usual Dirac operator on $G$ itself, which
is handled by the usual methods for Dirac operators on compact 
manifolds, as presented for example in section 4.2 of \cite{Frd}.) Accordingly,
the span of the eigenspaces of $D_o^\cA$ for eigenvalues $\l$ for which
$|\l| \leq 2/\e$, is finite-dimensional. We denote this span by $\cK$.
 
Then for $\cK^\perp$ we can choose an orthonormal basis, $\{\xi_j\}$,
consisting of eigenvectors of $D_o^\cA$, and for each $j$ the eigenvalues $\l_j$
for $\xi_j$ will satisfy $|\l_j| > 2/\e$. If $\xi \in \cK^\perp$ is in the domain of
$D_o^\cA$ so that for its expansion $\xi = \sum z_j\xi_j$ (with $z_j \in \bC$)
the sequence $\{|z_j\l_j|\}$ is square-summable, we have
\[
\|D_o^\cA\xi\|^2 = \sum |z_j|^2|\l_j|^2 \geq (\sum |z_j|^2)(2/\e)^2 
= (2/\e)^2\|\xi\|^2.
\]
Thus $\|\xi\| \leq (\e/2)\|D_o^\cA\xi\|$, so that 
$\|\xi - \theta^{\cA}_m( \th^{\cB^m} \xi)\| \leq \e(\|\xi\| + \|D_o^\cA\xi\|)$,
as desired.

Finally, let $\xi$ be an element of $\cK$.
In proposition 4.13 of \cite{Sai} Sain proves, in the more general context
of compact quantum groups, that (for our notation above)      
$\s^\cA_m( \s^{\cB^m}(a)) $ converges in norm to $a$
for every $a\in \cA$ as $m$ goes to $\infty$. (The proof is not difficult.)
It follows that $\s^\cA_m( \s^{\cB^m}(a)) $ converges to $a$
with respect to the Hilbert-space norm on $L^2(\cA, \tau)$,
and from this it follows that 
$ \theta^{\cA}_m( \th^{\cB^m} \xi)$ converges to $\xi$
for each $\xi$ in $\cK$ (since $\cK$ is contained in
the algebraic tensor product $\ciA \otimes \cS$ as follows from
Proposition \ref{selfad}). Since $\cK$ is finite-dimensional, it
follows that we can find a natural number $N$ such that
for every $m \geq N$ and every $\xi \in \cK$ we have
\begin{equation}
\label{estim}
\|\xi - \theta^{\cA}_m( \th^{\cB^m} \xi)\|  \leq \e\|\xi\|
 \leq \e(\|\xi\| + \|D_o^\cA\xi\|).
 \end{equation}
 
Putting together the steps above, we find that if $m \geq N$ then
$ \sfD^\cA$ coincides with the quotient norm of $\tn$ on
$\cS^\cA$, as needed. 
\end{proof}

We now put the above result into the framework
that \Lat \ uses in \cite{Ltr9} to treat the spectral propinquity
for Dirac operators on quantum tori. We follow closely the
pattern around theorem 3.25 of \cite{Ltr9}. 

To begin with, according to the proof of the triangle 
inequality for the dual modular 
propinquity given in theorem 3.1 of \cite{Ltr8},
and as done for non-commutative tori in theorem 3.25 of \cite{Ltr9},
we must view $\cS_t = \cS^\cA \oplus\cS^m$ 
as a Hilbert module over the 2-dimensional C*-algebra
$\cC = \bC \oplus \bC$ (i.e. the algebra of functions on
a 2-point space). 
This means that we view $\cS_t$ as a module over $\cC$ in the
evident way, and that we define a $\cC$-valued inner product
on $\cS_t$ by
\[
\<(\xi, \eta), (\xi', \eta')\>_\cC 
= (\<\xi, \xi'\>_{\cS^\cA} , \<\eta, \eta'\>_{\cS^m})   .
\]
For each $\e>0$ as used in the above
proposition, we define a C*-metric, $Q_\e$, on $\cC$ 
by $Q_\e(z,w) = (1/\e)|z-w|$ (so the distance between the two points
is $\e$).  This makes $\cC $ into a compact quantum
metric space. Then we have the evident projections of $\cC$
onto the two copies of $\bC$ (one-point spaces with trivial
metric), so that $(\cC, Q_\e)$ is a tunnel between these two one-point
spaces.   

Altogether the above structures, with the properties we have
obtained, form a modular tunnel
between the special metrical C*-correspondences
$(\cB^m, L^{D_o^m},  \cS^m, \sfD^m)$ and
$(\cA,  L^{D_o^\cA},\cS^\cA, \sfD^\cA)$
for large enough $m$, 
except that there is one further
property that must be verified, namely what \Lat \ calls
the ``inner Leibniz property''. This requires that
\[
Q_\e((\xi, \eta), (\xi', \eta')) \leq 2\tn((\xi, \eta))\tn((\xi', \eta'))   .
\]
For this we may need to make the natural number $N$
of Proposition \ref{tuncor} somewhat larger.
Specifically, we replace $\e$ with $\e/2$ and
choose $N$ large enough that in
equation \eqref{bigtun} we have $1/\e$ replaced by $2/\e$. 
In particular, we have
$\|\xi -\theta^{\cA}_m\eta\| \leq (\e/2)TN((\xi, \eta))$
for all $(\xi,\eta) \in \cS_t$ , and
similarly for $(\xi', \eta')$. 
Then 
\begin{align*} 
 &|\<\xi, \xi'\>_{\cS^\cA} - \<\eta, \eta'\>_{\cS^m}|   \\
 &\leq  |\<\xi, \xi'\>_{\cS^\cA} -  
 \<\theta^{\cA}_m\eta, \theta^{\cA}_m\eta'\>_{\cS^\cA}| 
 + | \<\theta^{\cA}_m\eta, \theta^{\cA}_m\eta'\>_{\cS^\cA}
 -   \<\eta, \eta'\>_{\cS^m}|    ,
\end{align*}
and, for the first term,
\begin{align*} 
&|\<\xi, \xi'\>_{\cS^\cA} -  \<\theta^{\cA}_m\eta, \theta^{\cA}_m\eta'\>_{\cS^\cA}| \\ 
&\leq  |\<\xi -\theta^{\cA}_m\eta, \xi'\>_{\cS^\cA}| + 
| \<\theta^{\cA}_m\eta, \xi' -\theta^{\cA}_m\eta'\>_{\cS^\cA}|   \\
 & \leq  \|\xi -\theta^{\cA}_m\eta\| \|\xi'\| + 
\| \theta^{\cA}_m\eta\| \|\xi' -\theta^{\cA}_m\eta'\|   \\
 & \leq (\e/2)\tn(\xi, \eta)\sfD^\cA(\xi') + \sfD^m(\eta)(\e/2)\tn(\xi', \eta')   \\
 &\leq \e \tn(\xi, \eta)\tn(\xi', \eta')  ,
\end{align*}
while for the second term, since $\theta^{\cB^m}$ is the 
adjoint of $\theta^{\cA}_m$,
\begin{align*}
&|\<\theta^{\cA}_m\eta, \theta^{\cA}_m\eta'\>_{\cS^\cA}
 -   \<\eta, \eta'\>_{\cS^m}|  =
|\<\theta^{\cB^m}\theta^{\cA}_m\eta, \eta'\>_{\cS^\cA}
 -   \<\eta, \eta'\>_{\cS^m}|   \\
& = |\<\theta^{\cB^m}\theta^{\cA}_m\eta  - \eta , \eta'\>_{\cS^\cA}|   
\leq \|\theta^{\cB^m}\theta^{\cA}_m\eta  - \eta\| \|\eta'\|  \\
& \leq  \|\theta^{\cB^m}\theta^{\cA}_m\eta  - \eta\| \tn(\xi', \eta')  .
\end{align*}
To handle the term $\|\theta^{\cB^m}\theta^{\cA}_m\eta  - \eta\|$
we argue much as we did to obtain equation \eqref{eqtun}. We
seek to show that
\begin{equation}
\label{eqin}
\|\theta^{\cB^m}\theta^{\cA}_m\eta  - \eta\| \leq \e(\|\eta\| + \|D_o^m\eta\|).
\end{equation} 
(so that $\|\theta^{\cB^m}\theta^{\cA}_m\eta  - \eta\| \leq \e\tn(\xi, \eta)$ ).
The left-hand side of equation \eqref{eqin} is $\leq 2\|\eta\|$. Thus
equation \eqref{eqin} will be satisfied if 
\[
2\|\eta\| \leq  \e(\|\eta\| + \|D_o^m\eta\|).
\] 
Suppose now that $\eta$ is an
eigenvector for $D_o^m$ with eigenvalue $\l$, so that the right-hand
side of equation \eqref{eqin} is $ \e(1 + |\l|)\|\eta\|$. It is then clear that 
equation  \eqref{eqin}  is satisfied if $ |\l|  \geq   2/\e $.
Let $\cK_m$ be the span of the eigenspaces of $D_o^m$ for 
eigenvalues $\l$ for which $|\l| \leq 2/\e$. Then we can argue exactly 
as in the fourth paragraph of the proof of Proposition \ref{tuncor} 
to conclude that if $\eta \in \cK_m^\perp$ then
$\|\eta\| \leq (\e/2)\|D_o^m\eta\|$, so that 
$\|\eta -  \th^{\cB^m} (\theta^{\cA}_m(\eta))\| \leq \e(\|\eta\| + \|D_o^m\eta\|)$.
 
To treat $\cK_m$ we recall first that $\s_m^\cA$ is always injective. 
(This follows immediately from theorem 3.1 of \cite{R6}, for which it is 
crucial that the range of $P^m$ is spanned by a \emph{highest}-weight 
vector.) Thus $\theta_m^\cA$ is injective. Furthermore, from Proposition
\ref{ucp} we see that
$\theta_m^\cA$ carries eigenvectors of $D_o^m$ to eigenvectors of
$D_o^\cA$ with the same eigenvalue (but maybe of different norm), so that
$\theta_m^\cA$ carries $\cK_m$ into $\cK$.
We now use the arguments in section 4 of
\cite{Sai}. As long as $\e < 1/2$, which we now assume (increasing $N$
if necessary), equation \eqref{estim} implies that $\theta_m^\cA \theta^{\cB^m}$,
as an operator on $\cK$, satisfies
$\|I_\cK - \theta_m^\cA \theta^{\cB^m}\| < 1/2$, 
so that $\theta_m^\cA \theta^{\b^m}$ is invertible
and $\|(\theta_m^\cA \theta^{\cB^m})^{-1}\| < 2$. 
In particular, $\theta_m^\cA $ is onto $\cK$, and so is invertible
as an operator from $\cK_m$ onto $\cK$. Consequently,
$\theta^{\cB^m}$ is invertible as an operator from $\cK$
to $\cK_m$, and 
$(\theta^{\cB^m})^{-1} = (\theta_m^\cA \theta^{\b^m})^{-1}\theta_m^\cA$,
so that $\|(\theta^{\cB^m})^{-1}\| < 2$. 
Since
\[
I_{\cK_m} -  \th^{\cB^m} \theta^{\cA}_m
= \theta^{\cB^m}(I_{\cK_m} - \theta_m^\cA \theta^{\cB^m})(\theta^{\cB^m})^{-1},
\]
we see, using equation \eqref{estim} and our replacement of $\e$
by $\e/2$, that 
\[
\|I_{\cK_m} -  \th^{\cB^m} \theta^{\cA}_m\| 
\leq 2\|I_{\cK_m} - \theta_m^\cA \theta^{\cB^m}\| \leq \e
\]
for all $m \geq N$. Thus
\[
\|\eta - \th^{\cB^m} \theta^{\cA}_m \eta\| 
\leq \e\|\eta\| \leq \e\sfD^m(\eta)
\]
for all $\eta \in \cK_m$. Putting this together with the result of
the previous paragraph, we obtain inequality \eqref{eqin},
so that $\|\theta^{\cB^m}\theta^{\cA}_m\eta  - \eta\| \leq \e\tn(\xi, \eta)$.
Putting this together with the several inequalities obtained before
inequality \eqref{eqin}, we obtain
\[
|\<\xi, \xi'\>_{\cS^\cA} - \<\eta, \eta'\>_{\cS^m}| 
\leq 2\e \tn(\xi, \eta)\tn(\xi', \eta') ,
\]
so that
\[
Q_\e((\xi, \eta), (\xi', \eta')) 
= (1/\e)|\<\xi, \xi'\>_{\cS^\cA} - \<\eta, \eta'\>_{\cS^m}|  
\leq 2\tn((\xi, \eta))\tn((\xi', \eta'))   ,
\]
as needed.

When we apply the definition of the ``extent'' of a modular tunnel,
as given in definition 4.2 of \cite{Ltr8}, to the
modular tunnel sketched above, the extent of this
modular tunnel is the extent of the tunnel
$(\cC, Q_\e)$, which is $\e$ (just as in the case
for the non-commutative tori, for which see the last paragraph 
in the proof of theorem 3.25 of \cite{Ltr9}).

For metrical C*-correspondences the modular propinquity
between them is defined to be the infimum of the extents
of all modular tunnels between them.
For the special metrical C*-correspondences corresponding
to the spectral triples $(\cA, \cS^\cA, D_o^\cA)$ and
$(\cB^m, \cS^m, D_o^m)$, the tunnels constructed
in Proposition \ref{tuncor} and the following discussion 
show that the sequence
$\{(\cB^m, L^{D_o^m},  \cS^m, \sfD^m)\}$ of
special metrical C*-correspondences converges
to the special metrical C*-correspondence
$(\cA,  L^{D_o^\cA},\cS^\cA, \sfD^\cA)$
for the dual modular propinquity.

In theorem 4.8 of \cite{Ltr8} \Lat \ proves the remarkable
fact that if the dual modular propinquity between two
metrical C*-correspondences is 0
(using a change of terminology he made in later papers),
then these two metrical C*-correspondences are fully
isometric, meaning unitary equivalent in an appropriate sense.
The application of this to spectral triples is given in proposition
2.25 of \cite{Ltr7}.
For $\cA = C(G/K)$ as usual, and with $D$ being the Dirac operator
defined in Definition \ref{defdir3}, acting on $\cS(G/K)$,
it is easily seen that the special metrical C*-correspondence
corresponding to the spectral triple $(\cA, \cS(G/K), D)$,
as defined in Theorem \ref{spnorm},
is not unitarily equivalent to the special metrical C*-correspondence
corresponding to the spectral triple
$(\cA,\cS^\cA, D^\cA_o)$. Consequently the sequence of
spectral triples $\{(\cB^m, \cS^m, D_o^{\cB^m})\}$
can not converge to the spectral triple $(\cA, \cS(G/K), D)$
for the spectral propinquity. Recall from 
Definition \ref{matdir} that when the unitary
representation $(\cH ,U)$ is faithful on $\fg$ the Dirac operators
for the corresponding matrix algebras coincide 
with the $D_o^m$'s
(and when the unitary representation is not faithful, the corresponding
Dirac operators are close to that).

I expect that when the covariance condition in the definition of the spectral propinquity, as described in section 3 of \cite{Ltr7}, is taken into account,
it will be seen that the spectral triples $\{(\cB^m, \cS^m, D_o^{\cB^m})\}$ 
do converge to the spectral triple $(\cA,\cS^\cA, D^\cA_o)$ for the spectral
propinquity. But I have not checked that.


\providecommand{\bysame}{\leavevmode\hbox to3em{\hrulefill}\thinspace}
\providecommand{\MR}{\relax\ifhmode\unskip\space\fi MR }
\providecommand{\MRhref}[2]{%
  \href{http://www.ams.org/mathscinet-getitem?mr=#1}{#2}
}
\providecommand{\href}[2]{#2}


\end{document}